\documentclass[]{article}

\usepackage{float}
\usepackage[utf8]{inputenc}

\usepackage[color=green]{todonotes}
\usepackage{amsmath, amsthm, amssymb, amsfonts}
\usepackage{thm-restate}
\usepackage{mathtools}
\usepackage{tikz}
\usepackage{multirow}
\usepackage[symbol]{footmisc}

\usepackage{enumerate}
\usepackage{hyperref}
\usepackage{cleveref}
\usepackage{complexity}
\newcommand{\threesat}{\lang{3SAT}}

\usepackage{caption}
\usepackage{subcaption}
\usepackage{microtype}
\usepackage{lscape}
\newtheorem{theorem}{Theorem}
\newtheorem*{theorem*}{Theorem}
\newtheorem*{3DC}{3-Decomposition Conjecture}
\newtheorem*{HISTDC}{HIST-Extension Conjecture}
\newtheorem{corollary}[theorem]{Corollary}
\newtheorem{lemma}[theorem]{Lemma}
\newtheorem{observation}[theorem]{Observation}
\theoremstyle{definition}
\newtheorem{definition}[theorem]{Definition}
\newtheorem{remark}[theorem]{Remark}

\definecolor{mcol}{RGB}{25,25,255}
\definecolor{tcol}{RGB}{25,200,25}
\definecolor{ccol}{RGB}{228,26,28}

\definecolor{myRed}{RGB}{255,51,51}
\definecolor{myBlue}{RGB}{51,51,255}
\definecolor{myDefGreen}{RGB}{0,190,76}
\definecolor{myBlack}{RGB}{0,0,0}

\tikzstyle{iedge}=[line width=1]
\tikzstyle{ivertexgrey} = [circle,draw=black!50,fill=black!50,inner sep=0.cm, minimum size=1.3mm,]
\tikzstyle{ivertexblack} = [circle,draw=black,fill=black,inner sep=0.cm, minimum size=1.3mm,]
\tikzstyle{bn}=[draw,circle,fill=white,minimum size=39pt,inner sep=0pt]
\tikzstyle{v}=[draw,circle,fill=black,minimum size=10pt,inner sep=0pt]
\tikzstyle{svertex}=[circle,inner sep=0.cm, minimum size=1.3mm, fill=black, draw=black]
\tikzstyle{novertex}=[rectangle]
\tikzset{font={\fontsize{8pt}{12}\selectfont}}
\usetikzlibrary{shapes,calc}
\usetikzlibrary{arrows,arrows.meta,petri,topaths}
\usetikzlibrary{decorations,decorations.pathreplacing}

\tikzset{node/.style={circle,draw,fill=black,scale=0.5}}
\tikzset{edge/.style={very thick}}
\tikzset{>=latex'}

\makeatletter
\tikzset{
	dot diameter/.store in=\dot@diameter,
	dot diameter=2pt,
	dot spacing/.store in=\dot@spacing,
	dot spacing=4pt,
	dots/.style={
		line width=\dot@diameter,
		line cap=round,
		dash pattern=on 0pt off \dot@spacing
	}
}
\makeatother
\tikzset{dotted/.style={very thick,dot diameter = 2pt, dot spacing=4pt,dots}}

\newcommand{\transformationArrow}[5][]%
{
	\draw[-{Latex[length=5*#5pt,width=4*#5pt]}, very thick, line width=#5pt] (#2,#3) to node[pos=0.5, yshift = 10pt] {\small#1} (#2+#4,#3);
}

\newcommand{\drawVertex}[7]%
{
	\node[circle,scale=#3/2] (#41) at (#1,#2) {};
	\node[circle,scale=#3/2] (#411) at (#1,#2+#3/2) {};
	\node[circle,scale=#3/2] (#422) at (#1-#3/3,#2-#3/3) {};
	\node[circle,scale=#3/2] (#433) at (#1+#3/3,#2-#3/3) {};
	\def\startlist{1,1,1}
	\def\endlist{11,22,33}
	\foreach \x [count=\c,
	evaluate=\c as \y using {{\startlist}[\c-1]},
	evaluate=\c as \z using {{\endlist}[\c-1]}] 
	in #5%
	{
		\draw[edge,\x] (#4\y) to (#4\z);
	}	
	\node[circle,scale=#3/2,draw=#6,fill=#6] at (#41) {};
	\node[circle,scale=#3/2,draw=#7,fill=#7] at (#411) {};
	\node[circle,scale=#3/2,draw=#7,fill=#7] at (#422) {};
	\node[circle,scale=#3/2,draw=#7,fill=#7] at (#433) {};
}

\newcommand{\drawVertexDown}[7]%
{
	\node[circle,scale=#3/2] (#41) at (#1,#2) {};
	\node[circle,scale=#3/2] (#411) at (#1,#2-#3/2) {};
	\node[circle,scale=#3/2] (#422) at (#1-#3/3,#2-#3/3) {};
	\node[circle,scale=#3/2] (#433) at (#1+#3/3,#2-#3/3) {};
	\def\startlist{1,1,1}
	\def\endlist{11,22,33}
	\foreach \x [count=\c,
	evaluate=\c as \y using {{\startlist}[\c-1]},
	evaluate=\c as \z using {{\endlist}[\c-1]}] 
	in #5%
	{
		\draw[edge,\x] (#4\y) to (#4\z);
	}	
	\node[circle,draw=#6,fill=#6,scale=#3/2] at (#41) {};
	\node[circle,scale=#3/2,draw=#7,fill=#7] at (#411) {};
	\node[circle,scale=#3/2,draw=#7,fill=#7] at (#422) {};
	\node[circle,scale=#3/2,draw=#7,fill=#7] at (#433) {};
}

\newcommand{\drawTriangle}[7]%
{
	\node[circle,scale=#3/2] (#41) at (#1,#2+#3/2) {};
	\node[circle,scale=#3/2] (#42) at (#1-2*#3/3,#2-#3/2) {};
	\node[circle,scale=#3/2] (#43) at (#1+2*#3/3,#2-#3/2) {};
	\node[circle,scale=#3/2] (#411) at (#1,#2+#3/2+#3/2) {};
	\node[circle,scale=#3/2] (#422) at (#1-2*#3/3-#3/3,#2-#3/2-#3/3) {};
	\node[circle,scale=#3/2] (#433) at (#1+2*#3/3+#3/3,#2-#3/2-#3/3) {};
	\def\startlist{1,1,2,1,2,3}
	\def\endlist{2,3,3,11,22,33}
	\foreach \x [count=\c,
	evaluate=\c as \y using {{\startlist}[\c-1]},
	evaluate=\c as \z using {{\endlist}[\c-1]}] 
	in #5%
	{
		\draw[edge,\x] (#4\y) to (#4\z);
	}
	\node[circle,draw=#6,fill=#6,scale=#3/2] at (#41) {};
	\node[circle,draw=#6,fill=#6,scale=#3/2] at (#42) {};
	\node[circle,draw=#6,fill=#6,scale=#3/2] at (#43) {};
	\node[circle,scale=#3/2,draw=#7,fill=#7] at (#411) {};
	\node[circle,scale=#3/2,draw=#7,fill=#7] at (#422) {};
	\node[circle,scale=#3/2,draw=#7,fill=#7] at (#433) {};
}

\newcommand{\drawKTwoThree}[7]%
{
	\node[circle,scale=#3/2] (#41) at (#1-#3,#2-#3/2) {};
	\node[circle,scale=#3/2] (#42) at (#1,#2-#3/2) {};
	\node[circle,scale=#3/2] (#43) at (#1+#3,#2-#3/2) {};
	\node[circle,scale=#3/2] (#44) at (#1-#3/2,#2+#3/2) {};
	\node[circle,scale=#3/2] (#45) at (#1+#3/2,#2+#3/2) {};
	\node[circle,scale=#3/2] (#411) at (#1-#3-#3/3,#2-#3/2-#3/3) {};
	\node[circle,scale=#3/2] (#422) at (#1,#2-#3/2-#3/2) {};
	\node[circle,scale=#3/2] (#433) at (#1+#3+#3/3,#2-#3/2-#3/3) {};
	\def\startlist{4,4,4,5,5,5,1,2,3}
	\def\endlist{1,2,3,1,2,3,11,22,33}
	\foreach \x [count=\c,
	evaluate=\c as \y using {{\startlist}[\c-1]},
	evaluate=\c as \z using {{\endlist}[\c-1]}] 
	in #5%
	{
		\draw[edge,\x] (#4\y) to (#4\z);
	}
	\node[circle,draw=#6,fill=#6,scale=#3/2] at (#41) {};
	\node[circle,draw=#6,fill=#6,scale=#3/2] at (#42) {};
	\node[circle,draw=#6,fill=#6,scale=#3/2] at (#43) {};
	\node[circle,draw=#6,fill=#6,scale=#3/2] at (#44) {};
	\node[circle,draw=#6,fill=#6,scale=#3/2] at (#45) {};
	\node[circle,scale=#3/2,draw=#7,fill=#7] at (#411) {};
	\node[circle,scale=#3/2,draw=#7,fill=#7] at (#422) {};
	\node[circle,scale=#3/2,draw=#7,fill=#7] at (#433) {};
}

\newcommand{\drawEdge}[7]%
{
	\node[circle,scale=#3/2] (#41) at (#1-#3/2,#2) {};
	\node[circle,scale=#3/2] (#42) at (#1+#3/2,#2) {};
	\node[circle,scale=#3/2] (#411) at (#1-#3/2-#3/3,#2+#3/3) {};
	\node[circle,scale=#3/2] (#422) at (#1+#3/2+#3/3,#2+#3/3) {};
	\node[circle,scale=#3/2] (#413) at (#1-#3/2-#3/3,#2-#3/3) {};
	\node[circle,scale=#3/2] (#424) at (#1+#3/2+#3/3,#2-#3/3) {};
	\def\startlist{1,1,2,1,2}
	\def\endlist{2,11,22,13,24}
	\foreach \x [count=\c,
	evaluate=\c as \y using {{\startlist}[\c-1]},
	evaluate=\c as \z using {{\endlist}[\c-1]}] 
	in #5%
	{
		\draw[edge,\x] (#4\y) to (#4\z);
	}
	\node[circle,draw=#6,fill=#6,scale=#3/2] at (#41) {};
	\node[circle,draw=#6,fill=#6,scale=#3/2] at (#42) {};
	\node[circle,scale=#3/2,draw=#7,fill=#7] at (#411) {};
	\node[circle,scale=#3/2,draw=#7,fill=#7] at (#422) {};
	\node[circle,scale=#3/2,draw=#7,fill=#7] at (#413) {};
	\node[circle,scale=#3/2,draw=#7,fill=#7] at (#424) {};
}

\newcommand{\drawSquare}[7]%
{
	\node[circle,scale=#3/2] (#41) at (#1-#3/2,#2+#3/2) {};
	\node[circle,scale=#3/2] (#42) at (#1+#3/2,#2+#3/2) {};
	\node[circle,scale=#3/2] (#43) at (#1-#3/2,#2-#3/2) {};
	\node[circle,scale=#3/2] (#44) at (#1+#3/2,#2-#3/2) {};
	\node[circle,scale=#3/2] (#411) at (#1-#3/2-#3/3,#2+#3/2+#3/3) {};
	\node[circle,scale=#3/2] (#422) at (#1+#3/2+#3/3,#2+#3/2+#3/3) {};
	\node[circle,scale=#3/2] (#433) at (#1-#3/2-#3/3,#2-#3/2-#3/3) {};
	\node[circle,scale=#3/2] (#444) at (#1+#3/2+#3/3,#2-#3/2-#3/3) {};
	\def\startlist{1,1,2,3,1,2,3,4}
	\def\endlist{2,3,4,4,11,22,33,44}
	\foreach \x [count=\c,
	evaluate=\c as \y using {{\startlist}[\c-1]},
	evaluate=\c as \z using {{\endlist}[\c-1]}] 
	in #5%
	{
		\draw[edge,\x] (#4\y) to (#4\z);
	}
	\node[circle,draw=#6,fill=#6,scale=#3/2] at (#41) {};
	\node[circle,draw=#6,fill=#6,scale=#3/2] at (#42) {};
	\node[circle,draw=#6,fill=#6,scale=#3/2] at (#43) {};
	\node[circle,draw=#6,fill=#6,scale=#3/2] at (#44) {};
	\node[circle,scale=#3/2,draw=#7,fill=#7] at (#411) {};
	\node[circle,scale=#3/2,draw=#7,fill=#7] at (#422) {};
	\node[circle,scale=#3/2,draw=#7,fill=#7] at (#433) {};
	\node[circle,scale=#3/2,draw=#7,fill=#7] at (#444) {};
}

\newcommand{\drawDomino}[7]%
{
	\node[circle,scale=#3/2] (#41) at (#1-#3/2,#2+3*#3/4) {};
	\node[circle,scale=#3/2] (#42) at (#1+#3/2,#2+3*#3/4) {};
	\node[circle,scale=#3/2] (#43) at (#1-#3/2,#2-3*#3/4) {};
	\node[circle,scale=#3/2] (#44) at (#1+#3/2,#2-3*#3/4) {};
	\node[circle,scale=#3/2] (#45) at (#1-#3/2,#2) {};
	\node[circle,scale=#3/2] (#46) at (#1+#3/2,#2) {};
	\node[circle,scale=#3/2] (#411) at (#1-#3/2-#3/3,#2+3*#3/4+#3/3) {};
	\node[circle,scale=#3/2] (#422) at (#1+#3/2+#3/3,#2+3*#3/4+#3/3) {};
	\node[circle,scale=#3/2] (#433) at (#1-#3/2-#3/3,#2-3*#3/4-#3/3) {};
	\node[circle,scale=#3/2] (#444) at (#1+#3/2+#3/3,#2-3*#3/4-#3/3) {};
	\def\startlist{1,1,2,5,5,6,3,1,2,3,4}
	\def\endlist{2,5,6,6,3,4,4,11,22,33,44}
	\foreach \x [count=\c,
	evaluate=\c as \y using {{\startlist}[\c-1]},
	evaluate=\c as \z using {{\endlist}[\c-1]}] 
	in #5%
	{
		\draw[edge,\x] (#4\y) to (#4\z);
	}
	\node[circle,draw=#6,fill=#6,scale=#3/2] at (#41) {};
	\node[circle,draw=#6,fill=#6,scale=#3/2] at (#42) {};
	\node[circle,draw=#6,fill=#6,scale=#3/2] at (#43) {};
	\node[circle,draw=#6,fill=#6,scale=#3/2] at (#44) {};
	\node[circle,draw=#6,fill=#6,scale=#3/2] at (#45) {};
	\node[circle,draw=#6,fill=#6,scale=#3/2] at (#46) {};
	\node[circle,scale=#3/2,draw=#7,fill=#7] at (#411) {};
	\node[circle,scale=#3/2,draw=#7,fill=#7] at (#422) {};
	\node[circle,scale=#3/2,draw=#7,fill=#7] at (#433) {};
	\node[circle,scale=#3/2,draw=#7,fill=#7] at (#444) {};
}

\newcommand{\drawDominoFlip}[7]%
{
	\node[circle,scale=#3/2] (#41) at (#1-#3/2,#2+3*#3/4) {};
	\node[circle,scale=#3/2] (#42) at (#1+#3/2,#2+3*#3/4) {};
	\node[circle,scale=#3/2] (#43) at (#1-#3/2,#2-3*#3/4) {};
	\node[circle,scale=#3/2] (#44) at (#1+#3/2,#2-3*#3/4) {};
	\node[circle,scale=#3/2] (#45) at (#1-#3/2,#2) {};
	\node[circle,scale=#3/2] (#46) at (#1+#3/2,#2) {};
	\node[circle,scale=#3/2] (#411) at (#1-#3/2-#3/3,#2+3*#3/4+#3/3) {};
	\node[circle,scale=#3/2] (#422) at (#1+#3/2+#3/3,#2+3*#3/4+#3/3) {};
	\node[circle,scale=#3/2] (#433) at (#1-#3/2-#3/3,#2-3*#3/4-#3/3) {};
	\node[circle,scale=#3/2] (#444) at (#1+#3/2+#3/3,#2-3*#3/4-#3/3) {};
	\def\startlist{1,1,2,5,5,6,3,1,2,3,4}
	\def\endlist{2,5,6,6,4,3,4,11,22,33,44}
	\foreach \x [count=\c,
	evaluate=\c as \y using {{\startlist}[\c-1]},
	evaluate=\c as \z using {{\endlist}[\c-1]}] 
	in #5%
	{
		\draw[edge,\x] (#4\y) to (#4\z);
	}
	\node[circle,draw=#6,fill=#6,scale=#3/2] at (#41) {};
	\node[circle,draw=#6,fill=#6,scale=#3/2] at (#42) {};
	\node[circle,draw=#6,fill=#6,scale=#3/2] at (#43) {};
	\node[circle,draw=#6,fill=#6,scale=#3/2] at (#44) {};
	\node[circle,draw=#6,fill=#6,scale=#3/2] at (#45) {};
	\node[circle,draw=#6,fill=#6,scale=#3/2] at (#46) {};
	\node[circle,scale=#3/2,draw=#7,fill=#7] at (#411) {};
	\node[circle,scale=#3/2,draw=#7,fill=#7] at (#422) {};
	\node[circle,scale=#3/2,draw=#7,fill=#7] at (#433) {};
	\node[circle,scale=#3/2,draw=#7,fill=#7] at (#444) {};
}

\newcommand{\drawGFour}[7]%
{
	\node[circle,scale=#3/2] (#41) at (#1-#3/2,#2+3*#3/4) {};
	\node[circle,scale=#3/2] (#42) at (#1+#3/2,#2+3*#3/4) {};
	\node[circle,scale=#3/2] (#43) at (#1-#3/2,#2-3*#3/4) {};
	\node[circle,scale=#3/2] (#44) at (#1+#3/2,#2-3*#3/4) {};
	\node[circle,scale=#3/2] (#45) at (#1-#3/2,#2) {};
	\node[circle,scale=#3/2] (#46) at (#1+#3/2,#2) {};
	\node[circle,scale=#3/2] (#411) at (#1-#3/2-#3/3,#2+3*#3/4+#3/3) {};
	\node[circle,scale=#3/2] (#422) at (#1+#3/2+#3/3,#2+3*#3/4+#3/3) {};
	\node[circle,scale=#3/2] (#433) at (#1-#3/2-#3/3,#2-3*#3/4-#3/3) {};
	\node[circle,scale=#3/2] (#444) at (#1+#3/2+#3/3,#2-3*#3/4-#3/3) {};
	\def\startlist{1,1,2,5,6,5,6,1,2,3,4}
	\def\endlist{2,5,6,3,4,4,3,11,22,33,44}
	\foreach \x [count=\c,
	evaluate=\c as \y using {{\startlist}[\c-1]},
	evaluate=\c as \z using {{\endlist}[\c-1]}] 
	in #5%
	{
		\draw[edge,\x] (#4\y) to (#4\z);
	}
	\node[circle,draw=#6,fill=#6,scale=#3/2] at (#41) {};
	\node[circle,draw=#6,fill=#6,scale=#3/2] at (#42) {};
	\node[circle,draw=#6,fill=#6,scale=#3/2] at (#43) {};
	\node[circle,draw=#6,fill=#6,scale=#3/2] at (#44) {};
	\node[circle,draw=#6,fill=#6,scale=#3/2] at (#45) {};
	\node[circle,draw=#6,fill=#6,scale=#3/2] at (#46) {};
	\node[circle,scale=#3/2,draw=#7,fill=#7] at (#411) {};
	\node[circle,scale=#3/2,draw=#7,fill=#7] at (#422) {};
	\node[circle,scale=#3/2,draw=#7,fill=#7] at (#433) {};
	\node[circle,scale=#3/2,draw=#7,fill=#7] at (#444) {};
}

\newcommand{\drawGFive}[7]%
{
	\node[circle,scale=2*#3/3] (#41) at (#1-#3/2,#2+#3/2) {};
	\node[circle,scale=2*#3/3] (#42) at (#1+#3,#2+#3/2) {};
	\node[circle,scale=2*#3/3] (#43) at (#1-#3/2,#2-#3) {};
	\node[circle,scale=2*#3/3] (#44) at (#1+#3,#2-#3) {};
	\node[circle,scale=2*#3/3] (#45) at (#1,#2) {};
	\node[circle,scale=2*#3/3] (#46) at (#1+#3,#2) {};
	\node[circle,scale=2*#3/3] (#47) at (#1,#2-#3) {};
	\node[circle,scale=2*#3/3] (#48) at (#1-#3,#2-#3/4) {};
	\node[circle,scale=2*#3/3] (#411) at (#1-#3/2-#3/2,#2+#3/2+#3/2) {};
	\node[circle,scale=2*#3/3] (#422) at (#1+#3+#3/2,#2+#3/2+#3/2) {};
	\node[circle,scale=2*#3/3] (#433) at (#1-#3/2-#3/2,#2-#3-#3/2) {};
	\node[circle,scale=2*#3/3] (#444) at (#1+#3+#3/2,#2-#3-#3/2) {};
	\def\startlist{5,5,6,7,1,2,3,8,8,8,1,2,3,4}
	\def\endlist{6,7,4,4,5,6,7,1,2,3,11,22,33,44}
	\foreach \x [count=\c,
	evaluate=\c as \y using {{\startlist}[\c-1]},
	evaluate=\c as \z using {{\endlist}[\c-1]}] 
	in #5%
	{
		\ifthenelse{\c=9}
		{
			\draw[edge,\x] (#4\y) to[bend left = 20] (#4\z);
		}
		{
			\draw[edge,\x] (#4\y) to[] (#4\z);
		}
	}
	\node[circle,scale=2*#3/3,draw=#6,fill=#6] at (#41) {};
	\node[circle,scale=2*#3/3,draw=#6,fill=#6] at (#42) {};
	\node[circle,scale=2*#3/3,draw=#6,fill=#6] at (#43) {};
	\node[circle,scale=2*#3/3,draw=#6,fill=#6] at (#44) {};
	\node[circle,scale=2*#3/3,draw=#6,fill=#6] at (#45) {};
	\node[circle,scale=2*#3/3,draw=#6,fill=#6] at (#46) {};
	\node[circle,scale=2*#3/3,draw=#6,fill=#6] at (#47) {};
	\node[circle,scale=2*#3/3,draw=#6,fill=#6] at (#48) {};
	\node[circle,scale=2*#3/3,draw=#7,fill=#7] at (#411) {};
	\node[circle,scale=2*#3/3,draw=#7,fill=#7] at (#422) {};
	\node[circle,scale=2*#3/3,draw=#7,fill=#7] at (#433) {};
	\node[circle,scale=2*#3/3,draw=#7,fill=#7] at (#444) {};
}

\newcommand{\drawPetMinusV}[7]%
{
	\node[circle,scale=#3] (#45) at (#1-#3*0.59,#2+#3*0.81) {};
	\node[circle,scale=#3] (#46) at (#1+#3*0.59,#2+#3*0.81) {};
	\node[circle,scale=#3] (#47) at (#1-#3*0.95,#2-#3*0.31) {};
	\node[circle,scale=#3] (#48) at (#1+#3*0.95,#2-#3*0.31) {};
	\node[circle,scale=#3] (#49) at (#1,#2-#3) {};
	\node[circle,scale=#3] (#41) at (#1-#3*1.18,#2+#3*1.62) {};
	\node[circle,scale=#3] (#42) at (#1+#3*1.18,#2+#3*1.62) {};
	\node[circle,scale=#3] (#43) at (#1-#3*1.9,#2-#3*0.62) {};
	\node[circle,scale=#3] (#44) at (#1+#3*1.9,#2-#3*0.62) {};
	\node[circle,scale=#3] (#433) at (#1-#3*1.9-2*#3/3,#2-#3*0.62-2*#3/3) {};
	\node[circle,scale=#3] (#444) at (#1+#3*1.9+2*#3/3,#2-#3*0.62-2*#3/3) {};
	\node[circle,scale=#3] (#499) at (#1,#2-2*#3) {};
	\def\startlist{3,1,2,5,8,7,6,9,1,2,3,4,3,4,9}
	\def\endlist{1,2,4,8,7,6,9,5,5,6,7,8,33,44,99}
	\foreach \x [count=\c,
	evaluate=\c as \y using {{\startlist}[\c-1]},
	evaluate=\c as \z using {{\endlist}[\c-1]}] 
	in #5%
	{
		\draw[edge,\x] (#4\y) to[] (#4\z);
	}
	\node[circle,scale=#3,draw=#6,fill=#6] (#45) at (#1-#3*0.59,#2+#3*0.81) {};
	\node[circle,scale=#3,draw=#6,fill=#6] (#46) at (#1+#3*0.59,#2+#3*0.81) {};
	\node[circle,scale=#3,draw=#6,fill=#6] (#47) at (#1-#3*0.95,#2-#3*0.31) {};
	\node[circle,scale=#3,draw=#6,fill=#6] (#48) at (#1+#3*0.95,#2-#3*0.31) {};
	\node[circle,scale=#3,draw=#6,fill=#6] (#49) at (#1,#2-#3) {};
	\node[circle,scale=#3,draw=#6,fill=#6] (#41) at (#1-#3*1.18,#2+#3*1.62) {};
	\node[circle,scale=#3,draw=#6,fill=#6] (#42) at (#1+#3*1.18,#2+#3*1.62) {};
	\node[circle,scale=#3,draw=#6,fill=#6] (#43) at (#1-#3*1.9,#2-#3*0.62) {};
	\node[circle,scale=#3,draw=#6,fill=#6] (#44) at (#1+#3*1.9,#2-#3*0.62) {};
	\node[circle,scale=#3,draw=#7,fill=#7] (#433) at (#1-#3*1.9-2*#3/3,#2-#3*0.62-2*#3/3) {};
	\node[circle,scale=#3,draw=#7,fill=#7] (#444) at (#1+#3*1.9+2*#3/3,#2-#3*0.62-2*#3/3) {};
	\node[circle,scale=#3,draw=#7,fill=#7] (#499) at (#1,#2-2*#3) {};
}

\newcommand{\drawBag}[5]%
{%
	\node[draw, ellipse, scale=0.8, minimum width=3cm,fill=#5] (#2) at #1 
	{
		\hspace{-15pt}
		\begin{tabular}{c}
			#3 \\ 
			#4
		\end{tabular}
		\hspace{-15pt}
	};
}

\newcommand{\petv}{\ensuremath{\text{Pet}-v}}
\newcommand{\calV}{\ensuremath{\mathcal{V}}}

\title{Reductions for the 3-Decomposition Conjecture}
\author{Oliver Bachtler and Irene Heinrich\footnote{The work of Irene Heinrich was supported by the
		European Research Council (ERC) under the European Union’s Horizon 2020
		research and innovation programme (EngageS: grant agreement No.\ 820148).}}

\begin{document}
\maketitle

\begin{abstract}
	The 3-decomposition conjecture is wide open.
	It asserts that every finite connected cubic graph can be decomposed into a spanning tree, a disjoint union of cycles, and a matching.
	We show that every such decomposition is derived from a homeomorphically irreducible spanning tree (HIST).
	This allows us to propose a novel reformulation of the 3-decomposition conjecture: the HIST-extension conjecture.
	
	We also prove that the following graphs are reducible configurations with respect to the 3-decomposition conjecture: the triangle, the $K_{2,3}$, the Petersen graph with one vertex removed, the claw-square, the twin-house, and the domino.
	As an application, we show that all 3-connected graphs of tree-width at most~3 or of path-width at most~4 satisfy the 3-decomposition conjecture and that a 3-connected minimum counterexample to the conjecture is triangle-free, all cycles of length at most~6 are induced, and every edge is in the centre of an induced $P_6$.
	
	Finally, we automate the na\"ive part of the process of checking whether a configuration is reducible and we prove that all graphs of order at most~20 satisfy the 3-decomposition conjecture.\end{abstract}

\section{Introduction}
Removing the edges of a spanning tree from a cubic graph results in a graph of maximum degree~2, which is necessarily a union of vertex-disjoint paths and cycles.
The 3-decomposition conjecture, postulated by Hoffmann-Ostenhof, see~\cite{HoffmannOstenhof2011, Cameron2011}, asserts that the spanning tree can be chosen so that all of the arising paths are of length~0 or~1.

\begin{3DC}
	Every finite connected cubic graph can be decomposed into a spanning tree, a 2-regular graph, and a (possibly empty) matching.
\end{3DC}
Any such decomposition $(T,C,M)$, consisting of a spanning tree $T$, a 2-regular graph $C$, and a matching $M$, is called a \emph{3-decomposition}.
So far, the 3-decomposition conjecture remains wide open.
It has been proved for the following classes: planar cubic graphs~\cite{HoffmannOstenhof2018}, traceable cubic graphs~\cite{Abdolhosseini2016, Liu2020}, generalised Hamiltonian cubic graphs~\cite{Bachtler2020, Xie2020}, and claw-free graphs~\cite{Aboomahigir2018, Hong2020}.
All of these results exploit one of the following two approaches:
\begin{enumerate}[(A)]
	\item \label{itm: reduction} \emph{Reduction:} prove a set of configurations to be unavoidable for the considered graph class. 
	Show that  each configuration is reducible with respect to the 3-decomposition conjecture. Moreover, prove that the reductions preserve the property of being in the considered class, cf.~\cite{HoffmannOstenhof2018, Aboomahigir2018, Hong2020, Lyngsie2019}.
	\item \emph{Manipulation of a global substructure:} pick a substructure which has a property that is guaranteed for the graphs in the considered class (for example a Hamiltonian path). Manipulate it such that the final result is a 3-decomposition, cf.~\cite{Abdolhosseini2016, Liu2020, Bachtler2020, Xie2020,BJSW21}.
\end{enumerate}
In this article, we focus on approach~\eqref{itm: reduction} as a powerful tool for tackling the 3-decomposition conjecture. 
We contribute new reducible configurations for the 3-decomposition conjecture and divide the reductions into two classes: coloured and uncoloured ones.

\paragraph{Coloured reductions.}
\emph{Coloured reductions} are applicable to graphs whose edges are given one of two colours (green and black) if the considered configuration respects the edge-colouring.
We show that the coloured reductions given in \Cref{fig: ExtensionOfDecomposition} already suffice to reduce each 3-decomposition to a 3-decomposition of a graph with a \emph{homeomorphically irreducible spanning tree} (\emph{HIST}) which is a spanning tree without degree-2 vertices (see~\cite{Chen2013} for an overview on the study of HISTs.)
\begin{restatable}{theorem}{hist}
	\label{thm: reducible or HIST}
	Let $G$ be a cubic graph with a 3-decomposition $(T, C, M)$ in which exactly the edges of $T$ are coloured green.
	If~$T$ is a HIST, then $M = \emptyset$.
	Otherwise $G$ with its 3-decomposition $(T,C,M)$ can be obtained by a Tutte-extension or a diamond-extension (see Figure~\ref{fig: ExtensionOfDecomposition}) from a smaller cubic graph $G'$ with 3-decomposition $(T',C,M')$ where exactly the edges of $T'$ are green.
	
	In particular, every 3-decomposition can be reduced to a 3-decomposition with a HIST by a finite sequence of Tutte- and diamond-reductions.
\end{restatable}

Observe that a cubic graph which has a HIST trivially satisfies the 3-decomposition conjecture since the edges outside the HIST form a 2-regular graph.
We exploit our two coloured reductions to propose the new \emph{HIST-extension conjecture} and to prove that it is equivalent to the 3-decomposition conjecture.

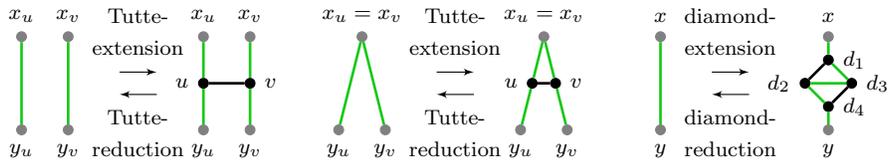
\begin{figure}[ht]
	\centering
	\begin{tikzpicture}[scale=.62]
		\begin{scope}[shift={(0,0)}]
			\begin{scope}[shift={(-.3,0)}]
				\draw
				(-.7,1) node (lo) [ivertexgrey,label=above:{$x_u$}] {}
				(-.7,-1) node (lu) [ivertexgrey,label=below:{$y_u$}] {}
				(.3,1) node (ro) [ivertexgrey,label=above:{$x_v$}] {}
				(.3,-1) node (ru) [ivertexgrey,label=below:{$y_v$}] {}
				;
				\draw[iedge, color=tcol] (lu)--(lo) (ru)--(ro);
			\end{scope}
			\begin{scope}[shift={(1.5,0)}, scale=.6]
				\draw
				(-1,.4) node (lo) {}
				(1,.4) node (ro) {}
				(0,0) node (c) {}
				(-1,-.4) node (lu) {}
				(1,-.4) node (ru) {};
				\draw[->] (lo)--(ro);
				\draw[<-] (lu)--(ru);
				\node[below of=c, align=center, yshift=9.5] (tutte-ext) {Tutte- \\ reduction};
				\node[above of=c, align=center, yshift=-9] (tutte-ext) {Tutte- \\ extension};
			\end{scope}
			\begin{scope}[shift={(3.6,0)}]
				\draw
				(-.7,1) node (lo) [ivertexgrey,label=above:{$x_u$}] {}
				(-.7,0) node (lm) [ivertexblack,label=left:{$u$}] {}
				(-.7,-1) node (lu) [ivertexgrey,label=below:{$y_u$}] {}
				(.3,1) node (ro) [ivertexgrey,label=above:{$x_v$}] {}
				(.3,0) node (rm) [ivertexblack,label=right:{$v$}] {}
				(.3,-1) node (ru) [ivertexgrey,label=below:{$y_v$}] {}
				;
				\draw[iedge, color=tcol] (lu)--(lm)--(lo) (ru)--(rm)--(ro);
				\draw[iedge] (lm)--(rm);
			\end{scope}
		\end{scope}
		\begin{scope}[shift={(6.8,0)}]
			\begin{scope}[shift={(-.3,0)}]
				\draw
				(-.2,1) node (o) [ivertexgrey,label=above:{$x_u=x_v$}] {}
				(-.7,-1) node (lu) [ivertexgrey, label=below:{$y_u$}] {}
				(.3,-1) node (ru) [ivertexgrey, label=below:{$y_v$}] {}
				;
				\draw[iedge, color=tcol] (lu)--(o)--(ru);
			\end{scope}
			\begin{scope}[shift={(1.5,0)}, scale=.6]
				\draw
				(-1,.4) node (lo) {}
				(1,.4) node (ro) {}
				(0,0) node (c) {}
				(-1,-.4) node (lu) {}
				(1,-.4) node (ru) {};
				\draw[->] (lo)--(ro);
				\draw[<-] (lu)--(ru);
				\node[below of=c, align=center, yshift=9.5] (tutte-ext) {Tutte- \\ reduction};
				\node[above of=c, align=center, yshift=-9] (tutte-ext) {Tutte- \\ extension};
			\end{scope}
			\begin{scope}[shift={(3.6,0)}]
				\draw
				(-.2,1) node (o) [ivertexgrey,label=above:{$x_u=x_v$}] {}
				(-.7,-1) node (lu) [ivertexgrey,label=below:{$y_u$}] {}
				(.3,-1) node (ru) [ivertexgrey,label=below:{$y_v$}] {}
				(-.45,0) node (lm) [ivertexblack, label=left:{$u$}] {}
				(.05,0) node (rm) [ivertexblack, label=right:{$v$}] {}
				;
				\draw[iedge, color=tcol] (lu)--(lm)--(o)--(rm)--(ru);
				\draw[iedge] (lm)--(rm);
			\end{scope}
		\end{scope}
		\begin{scope}[shift={(13,0)}]
			\begin{scope}[shift={(-.3,0)}]
				\draw
				(0,1) node (lo) [ivertexgrey, label=above:{$x$}] {}
				(0,-1) node (lu) [ivertexgrey, label=below:{$y$}] {}
				;
				\draw[iedge, color=tcol] (lu)--(lo);
			\end{scope}
			\begin{scope}[shift={(1.2,0)}, scale=.6]
				\draw
				(-1,.4) node (lo) {}
				(1,.4) node (ro) {}
				(0,0) node (c) {}
				(-1,-.4) node (lu) {}
				(1,-.4) node (ru) {};
				\draw[->] (lo)--(ro);
				\draw[<-] (lu)--(ru);
				\node[below of=c, align=center, yshift=9.5] (tutte-ext) {diamond- \\ reduction};
				\node[above of=c, align=center, yshift=-9] (tutte-ext) {diamond- \\ extension};
			\end{scope}
			\begin{scope}[shift={(3.3,0)}]
				\draw
				(0,1) node (lo) [ivertexgrey, label=above:{$x$}] {}
				(0,-1) node (lu) [ivertexgrey, label=below:{$y$}] {}
				(0,-.5) node (du) [ivertexblack, label=right:{$d_4$}]{}
				(0,.5) node (do) [ivertexblack, label=right:{$d_1$}]{}
				(-.5,0) node (dl) [ivertexblack, label=left:{$d_2$}]{}
				(.5,0) node (dr) [ivertexblack, label=right:{$d_3$}]{}
				;
				\draw[iedge, color=tcol] (lu)--(du)--(dl)--(dr)--(do)--(lo);
				\draw[iedge] (du)--(dr) (dl)--(do);
			\end{scope}
		\end{scope}	
	\end{tikzpicture}	\caption{Extensions and reductions. Exactly the tree-edges are coloured green. Grey vertices are those with neighbours outside of the configuration.
		\emph{Tutte-extension}: subdivide two green edges and join the subdivision vertices with a black edge. The two green edges may be non-incident (left figure) or incident (middle figure).
		\emph{Diamond-extension}: remove a green edge with ends $x$ and $y$; add a new diamond in which all edges are green except two non-incident ones; if $d_1$ and $d_4$ denote the two degree-2-vertices of the diamond, then join $d_1$ with $x$ and $d_4$ with $y$ by green edges.
		The inverse operations are the corresponding \emph{reductions}.
	} \label{fig: ExtensionOfDecomposition}
\end{figure}

\begin{HISTDC}
	Let $G$ be a finite connected cubic graph.
	Then there exists a cubic graph $G'$ admitting a HIST $T'\subseteq G'$ such that if exactly the edges of $T'$ are coloured green, then $G$ can be obtained from $G'$ by a finite sequence of Tutte- and diamond-extensions (see \Cref{fig: ExtensionOfDecomposition}).
\end{HISTDC}

We emphasise the strong similarities between the HIST-extension conjecture and the following theorem of Wormald.
\begin{theorem}[\cite{Wormald79}]
	Every connected cubic graph other than the~$K_4$ can be obtained from a smaller cubic graph by one of the following three operations:
	\begin{enumerate}[(i)]
		\item subdivide two edges and join the subdivision vertices by a new edge, \label{itm: Tuttestep}
		\item replace an edge by a diamond and join the degree-2 vertices of the diamond with the former ends of the replaced edge, \label{itm: wormald diamond}
		\item take the disjoint union with a $K_4$, subdivide an edge in each component of this union and join the subdivision vertices. \label{itm: wormald house}
	\end{enumerate}
\end{theorem}
The HIST-extension conjecture postulates that is is possible to replace the~$K_4$ in Wormald's theorem by the substantially larger infinite ground set of all cubic graphs which allow for a homeomorphically irreducible spanning tree but, in turn, to restrict Wormald's three operations to two coloured variants of the operations~\eqref{itm: Tuttestep} and~\eqref{itm: wormald diamond}.

We derive from~\Cref{thm: reducible or HIST} the equivalence of the two conjectures.

\begin{restatable}{corollary}{equivalentconjectures}
	The 3-decomposition conjecture and the HIST-extension conjecture are equivalent.
\end{restatable}

\paragraph{Uncoloured reductions.}
A subcubic graph $S$ is a \emph{reducible configuration} if no minimal 3-connected counterexample to the 3-decomposition conjecture contains $S$ as a subgraph.
We show that the following six configurations are reducible: 
the triangle, the $K_{2,3}$, the \petv, the \emph{claw-square}, the \emph{twin-house}, and the \emph{domino}.
\begin{restatable}{theorem}{reduciblegraphs}
	\label{reducible-graphs}
	The graphs in \Cref{fig:reducible-graphs} are reducible.
\end{restatable}
Roughly speaking, the reducibility proofs proceed as follows:
given a cubic graph $G$ containing the reducible configuration $S$, replace $S$ by some smaller graph $S'$.
This yields a new graph $G'$ containing $S'$ and we show that every local behaviour that a 3-decomposition might exhibit on $S'$ can be extended to $S$.
Consequently, these proofs are constructive: we can obtain a 3-decomposition of a graph $G$ containing $S$ from a 3-decomposition of the graph $G'$.

\begin{figure}[htb]
	\centering
	\begin{tikzpicture}[scale=.75]
		\drawTriangle{-8.6}{-0.1}{1}{tr}{{black,black,black,black,black,black}}{black}{black!50}
		\node[] at (-8.6,-1.5) {triangle};
		
		\drawKTwoThree{-5.75}{-0.1}{1}{k33}{{black,black,black,black,black,black,black,black,black}}{black}{black!50}
		\node[] at (-5.75,-1.5) {$K_{2,3}$};
		
		\drawPetMinusV{-2.5}{-.08}{0.5}{pet}{{black,black,black,black,black,black,black,black,black,black,black,black,black,black,black}}{black}{black!50}
		\node[] at (-2.5,-1.5) {\petv};
		
		\drawGFive{0.25}{.05}{0.75}{g5}{{black,black,black,black,black,black,black,black,black,black,black,black,black,black}}{black}{black!50}
		\node[] at (.45,-1.5) {claw-square};
		
		\drawGFour{2.9}{0}{1}{g4}{{black,black,black,black,black,black,black,black,black,black,black}}{black}{black!50}
		\node[] at (2.9,-1.5) {twin-house};
		
		\drawDomino{5.3}{0}{1}{g3}{{black,black,black,black,black,black,black,black,black,black,black}}{black}{black!50}
		\node[] at (5.3,-1.5) {domino};
	\end{tikzpicture}
	\caption{Reducible configurations for the 3-decomposition conjecture.
	The names refer to the graphs induced by the black vertices.
	\label{fig:reducible-graphs}}
\end{figure}
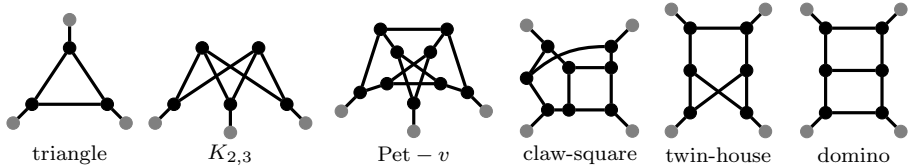

\paragraph{Applications and minimum counterexamples.}
Using the new reducible configurations enables us to extend the list of graph classes which satisfy the 3-decomposition conjecture.
Due to the constructive nature of the reducibility proofs, we obtain a method of obtaining decompositions for these graphs, too.
\begin{restatable}{theorem}{twthree}
	\label{tw-three}
	Every 3-connected cubic graph of tree-width~3 satisfies the 3-decomposition conjecture.%
	\ifnum\value{page}<10%
		\footnote{This result is contained in the PhD thesis of the second author~\cite{Heinrich2019, Heinrich2020} but is not published elsewhere.\label{diss}}%
	\else%
		\textsuperscript{\ref{diss}}
	\fi%
\end{restatable}

\begin{restatable}{theorem}{pwfour}
	\label{pw-four}
	Every 3-connected cubic graph of path-width~4 satisfies the 3-decomposition conjecture.
\end{restatable}

As a further use of the new reducible configurations, we gain new insights into the structure of (potential) minimum counterexamples to the 3-decomposition conjecture.
These results are summarised in the following theorem.
\begin{restatable}{theorem}{propertiescounterexample}
	\label{thm: min counterexample}
	If $G$ is a minimum counterexample to the 3-decomposition conjecture amongst all 3-connected cubic graphs, then it does not contain any of the graphs in Figure~\ref{fig:reducible-graphs} as a subgraph. 
	In particular,
	\begin{enumerate}[(i)]
		\item the girth of $G$ is at least~4, \label{itm: trianglefree}
		\item every cycle of length 4, 5, or 6 in $G$ is induced, and \label{itm: induced cycles}
		\item every edge of $G$ is the centre edge of an induced $P_6$. \label{itm: P6}
	\end{enumerate}
\end{restatable}

\paragraph{Automatically checking reducibility}
At the core of our reducibility proofs lies the aforementioned task of checking that \emph{every} behaviour which a 3-decomposition might exhibit on the smaller graph $S'$ can be extended to $S$.
This leads to a lot of cases, many of which are straight-forward.
We identify these \emph{na\"ively extendable} cases and, in the final part of this article, we demonstrate how these can be handled algorithmically in order to focus on the remaining harder cases.
We also prove that checking for na\"ive extendability is an \NP-complete problem.

\begin{restatable}{theorem}{npcomplete}
	\label{naive-extensions-np-complete}
	Determining whether a forest is na\"ively extendable is \NP-complete.
\end{restatable}

Finally, we verify the conjecture for all small graphs employing a computer.
\begin{restatable}{theorem}{smallgraphs}
	\label{small-graphs-3-dec}
	All graphs of order at most 20 satisfy the 3-decomposition conjecture.
	In particular, a minimum counterexample has order at least~22.
\end{restatable}

\paragraph{Outline.}
We specify the basic notation and definitions in the next section.
In Section~\ref{sec: extended or HIST} we prove that every 3-decomposition stems from an extension of a decomposition with a HIST.
We formalise the concepts of extensions and reductions, which are fundamental for this paper, in \Cref{sec:ext-red}, where we also show initial properties concerning these.
We give proofs of reducibility of the six configurations shown in \Cref{fig:reducible-graphs} in Section~\ref{sec: reducible configurations} (and partly in the appendix).
Na\"ive extensions and their complexity are discussed in Section~\ref{sec: naive and complex}.
In \Cref{sec: applications and small graphs} we present all the consequences that follow from these reducible configurations, together with our computer-aided proof that such a counterexample has at least 22 vertices.

\section{Preliminaries}
All graphs in this paper are simple and finite.
Most of the following notation is based on \cite{Die10}.
A \emph{graph} $G$ has \emph{vertex set} $V(G)$ and \emph{edge set} $E(G)$.
We write $uv$ for an edge with ends~$u$ and~$v$ and, for $A,B\subseteq V(G)$, $E(A,B)$ denotes the set of edges with one end in $A$ and the other in $B$.
The set of \emph{neighbours} of a vertex $v$ in $G$ is $N_G(v)$ or $N(v)$.
The graph obtained from $G$ by removing a subset $V'$ of $V(G)$ and all edges with at least one end in $V'$ is denoted by $G-V'$.
If $V' = \{v\}$ is a singleton, then we write $G-v$ as a shorthand for $G-\{v\}$.
Similarly, for an edge set $E'\subseteq E(G)$ we write $G-E'$ for the graph with vertex set $V(G)$ and edge set $E\setminus E'$ and we abbreviate $G-\{e\}$ to $G-e$.
A \emph{path} is a graph~$P$ with $V(P)=\{v_1,\ldots,v_n\}$ and $E(P) =\{v_1v_2,\ldots,v_{n-1}v_n\}$, for which we write $P=v_1\ldots v_n$.
Similarly, we write $v_1\ldots v_nv_1$ for the \emph{cycle} $P+v_nv_1$.
The complete graph, the cycle, and the path on $n$ vertices are denoted by $K_n$, $C_n$, and $P_n$ respectively and $K_{r_1,r_2}$ is the \emph{complete bipartite graph} with parts of cardinality $r_1$ and $r_2$.

A \emph{tree-decomposition} $(T, \calV)$ of a graph $G$ consists of a tree $T$
and a family of vertex sets $\calV=\{V_i\colon i\in V(T)\} \subseteq 2^{V(G)}$ indexed by the vertices $i$ of $T$.
It satisfies that 
\begin{enumerate}[(i)]
	\item every vertex $v\in V(G)$ is contained is some set $V_i$,
	\item for each edge $uv\in E(G)$ there exists a set $V_i$ such that $\{u,v\}\subseteq V_i$, and
	\item the set $\{i\colon v\in V_i\}$ induces a subtree of $T$ for all $v\in V(G)$.
\end{enumerate}
The \emph{width} of $(T, \calV)$ is defined as $\max\{|V'|\colon V' \in \calV \}-1$.
If $T$ is a path, then $(T,\calV)$ is a \emph{path-decomposition of~$G$}.
Furthermore, the \emph{tree-width} (respectively \emph{path-width}) of~$G$ is the minimal width of its tree-decompositions (respectively path-decompositions).

A \emph{decomposition} of a graph~$G$ is a list of subgraphs such that every edge of~$G$ is contained in exactly one of the subgraphs.
We say that $G$ \emph{decomposes} into the graphs in the list.
In the context of decompositions 1-regular subgraphs are called \emph{matchings}.
A \emph{3-decomposition $(T,C,M$)} of a graph~$G$ is a decomposition of $G$ into a spanning tree $T$, a 2-regular subgraph~$C$, and a matching $M$.
An edge of  $T$, $C$, and $M$ is called \emph{$T$-edge}, \emph{$C$-edge}, and \emph{$M$-edge} and coloured green, red, and blue when we depict 3-decompositions in figures, respectively.

\section{The HIST-extension conjecture}
\label{sec: extended or HIST}
A graph is \emph{homeomorphically irreducible} if it does not contain vertices of degree~2.
We use \emph{HIST} as an abbreviation for {homeomorphically irreducible spanning tree}.
\begin{lemma}[3-decompositions are preserved by Tutte- and diamond-extensions and reductions]
	\label{3-dec-preserved}
	Let $G'$ be a graph with a spanning tree $T'$.
	Colour the edges of~$T'$ green and all other edges of $G'$ black.
	If $(G, T)$ is obtained from $(G', T')$ by either a Tutte- or a diamond-extension (see \Cref{fig: ExtensionOfDecomposition}), then $G-E(T)$ decomposes into cycles and and a matching if and only if $G'-E(T')$ does.
\end{lemma}

\begin{proof}
	First observe that the respective extensions and reductions preserve the property that the green edges form a spanning tree (both only insert subdivision vertices on green edges and connect them by black edges).
	Since $G'-E(T')$ is a graph of maximum degree~2 it decomposes into a 2-regular graph $C'$ and a disjoint union of paths $P'$.
	If $G$ is obtained from $G'$ by a Tutte-extension, then $G$ decomposes into $T$, the 2-regular graph $C'$, and the disjoint union of $P'$ with an additional $K_2$.
	Otherwise $G$ is a diamond-extension of $G'$ and decomposes into $T$, $C'$, and the disjoint union of $P'$ with a $2K_2$.
	In both cases~$P'$ is a matching if and only if $G-E(T)-E(C')$ is a matching.
	This settles the claim.
\end{proof}

\hist*

\begin{proof}
	Throughout this proof, a graph is always given with a 3-decomposition.
	The edges of the tree in the decomposition are coloured green and all other edges are black.
	If $T$ is homeomorphically irreducible, then $G-E(T)$ has only vertices of degree~0 and~2.
	In particular, $G-E(T)$ is a disjoint union of isolated vertices and cycles and, hence, $M = \emptyset$.
	
	From now on we may assume that $T$ contains a vertex $v$ with $\deg_T(v) =2$.
	Since $(T, C, M)$ is a 3-decomposition of $G$, we obtain that $v$ is incident to an $M$-edge $uv$ in~$G$.
	We denote the neighbours of $u$ and $v$ as depicted in \Cref{fig:3-dec-hist-proof-deg2}.
	Since $uv$ is an M-edge, the edges $ux_u$, $uy_u$, $vx_v$, and $vy_v$ are $T$-edges.
	
	We can apply a Tutte-reduction here unless one of the edges $x_uy_u$ or $x_vy_v$ is present in $G$ since these would create parallels.
	Hence, if neither of these edges is in $G$, then $(G, (T, C, M))$ can be obtained via a Tutte-extension from $G-\{u,v\}+\{x_uy_u, x_vy_v\}$ with the 3-decomposition $(T-\{x_uu, uy_u, x_vv, vy_v\}+\{x_uy_u, x_vy_v\}, C, M-uv)$.
	
	By symmetry, we may assume that $x_uy_u \in E(G)$, as shown in \Cref{fig:3-dec-hist-proof-xuyu}.
	This implies that $x_uy_u$ is an $M$-edge (a $T$-edge would cause the cycle $ux_uy_uu$ in $T$ whereas a $C$-edge would disconnect the $T$-edges $ux_u$ and $uy_u$ from the rest of the tree).
	We may conclude that $x_u\tilde{x}_u$ and $y_u\tilde{y}_u$ are $T$-edges, where $\tilde{x}_u$ (respectively $\tilde{y}_u$) denotes the unique neighbour of $x_u$ (respectively $y_u$) outside the triangle $ux_uy_uu$.
	In particular, we obtain $\tilde{x}_u \neq \tilde{y}_u$ since $T$ is cycle-free.
	This is illustrated in \Cref{fig:3-dec-hist-proof-notv}.
	
	The only obstacle to applying a Tutte-reduction using the $M$-edge $x_ux_v$ is the case where either $\tilde{x}_u=v$ or $\tilde{y}_u = v$ since this creates a parallel to the $M$-edge $uv$.
	Thus, if $v \notin \{\tilde{x}_u$, $\tilde{y}_u\}$, then $G$ can be obtained by a Tutte-extension of $G'\coloneqq G-\{x_u, x_v\} + \{u\tilde{x}_u, u\tilde{y}_u\}$ with the 3-decomposition $(T-E(\tilde{x}_ux_uuy_u\tilde{y}_u)+\{\tilde{x}_uu, u\tilde{y}_u \}, C, M-\{x_uy_u\})$.
	
	Otherwise $\tilde{x}_u = v$ or $\tilde{y}_u = v$.
	Say, $\tilde{x}_u = v$ (the other case works similarly by interchanging the roles of $x$ and $y$).
	We assume, as illustrated in \Cref{fig:3-dec-hist-proof-xut}, that $x_u = x_v$ (otherwise rename $x_v$ and $y_v$).
	Observe that $y_v\neq \tilde{y}_u$ since this would cause a cycle in $T$.
	If~$y_v$ and~$\tilde{y}_u$ are not adjacent, then $G$ can be obtained by a diamond-extension of the graph $G-\{y_u, x_u, u, v\} + \{\tilde{y}_u\tilde{v}\}$ with the 3-decomposition $(T-E(\tilde{y}_uy_uux_uv\tilde{v})+\{\tilde{y}_u\tilde{v}\}, C,  M-\{y_ux_u, uv\})$.
	If, otherwise, $\tilde{v}\tilde{y}_u \in E(G)$, then this edge is an $M$-edge (a $T$-edge would cause a cycle in $T$ and a $C$-edge would disconnect $T$).
	Let $\tilde{y}_v'$ be the unique vertex in $N_G(y_v)\setminus \{\tilde{x}_u, \tilde{y}_u\}$ and let $\tilde{y}_u'$ be the unique neighbour of $\tilde{y}_u$ which is not in $\{y_u, y_v\}$.
	Since $y_u$ and $v$ already have three neighbours, we obtain that $y_u\tilde{y}_u'$ and $v\tilde{y}_v'$ are not edges of $G$.
	In particular, $G$ can be obtained by a Tutte-extension of $G-\{y_v,\tilde{y}_u\}+\{v\tilde{y}_v', y_u\tilde{y}_u'\}$ with the 3-decomposition $(T-\{\tilde{y}_u'\tilde{y}_u,\tilde{y}_uy_u, vy_v, y_v\tilde{y}_v'\}+\{v\tilde{y}_v', y_u\tilde{y}_u'\}, C, M-y_v\tilde{y}_u)$.
	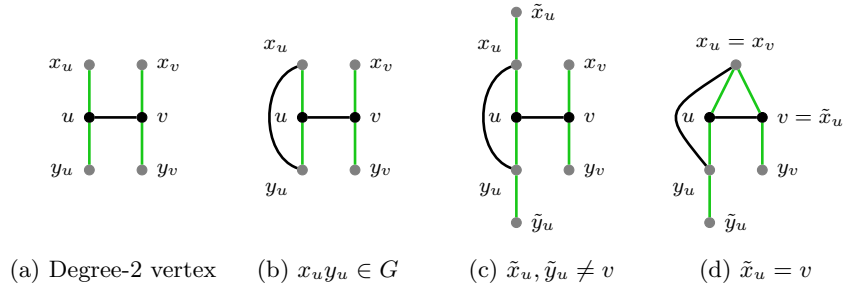
\begin{figure}[ht]
		\centering
		\subcaptionbox[]{Degree-2 vertex\label{fig:3-dec-hist-proof-deg2}}[.225\textwidth][c]{
			\begin{tikzpicture}[scale=0.7]
				\draw
				(-.7,2) node (lo2) [label=right:{\phantom{$\tilde{x}_u$}}] {}
				(-.7,-2) node (lu2) [label=right:{\phantom{$\tilde{y}_u$}}] {}	
				(-.7,1) node (lo) [ivertexgrey,label=left:{$x_u$}] {}
				(-.7,0) node (lm) [ivertexblack,label=left:{$u$}] {}
				(-.7,-1) node (lu) [ivertexgrey,label=left:{$y_u$}] {}
				(.3,1) node (ro) [ivertexgrey,label=right:{$x_v$}] {}
				(.3,0) node (rm) [ivertexblack,label=right:{$v$}] {}
				(.3,-1) node (ru) [ivertexgrey,label=right:{$y_v$}] {};
				\draw[iedge, color=tcol] (lu)--(lm)--(lo) (ru)--(rm)--(ro);
				\draw[iedge] (lm)--(rm);
			\end{tikzpicture}		}
		\subcaptionbox[]{$x_uy_u\in G$\label{fig:3-dec-hist-proof-xuyu}}[.225\textwidth][c]
		{
			\begin{tikzpicture}[scale=0.7]
				\draw
				(-.7,2) node (lo2) [label=right:{\phantom{$\tilde{x}_u$}}] {}
				(-.7,-2) node (lu2) [label=right:{\phantom{$\tilde{y}_u$}}] {}	
				(-.7,1) node (lo) [ivertexgrey,label=above left:{$x_u$}] {}
				(-.7,0) node (lm) [ivertexblack,label=left:{$u$}] {}
				(-.7,-1) node (lu) [ivertexgrey,label=below left:{$y_u$}] {}
				(.3,1) node (ro) [ivertexgrey,label=right:{$x_v$}] {}
				(.3,0) node (rm) [ivertexblack,label=right:{$v$}] {}
				(.3,-1) node (ru) [ivertexgrey,label=right:{$y_v$}] {};
				\draw[iedge, color=tcol] (lu)--(lm)--(lo) (ru)--(rm)--(ro);
				\draw[iedge] (lm)--(rm);
				\draw[iedge] (lo) to[bend right=70] (lu) {};
			\end{tikzpicture}		}
		\subcaptionbox[]{$\tilde{x}_u,\tilde{y}_u \neq v$\label{fig:3-dec-hist-proof-notv}}[.225\textwidth][c]
		{
			\begin{tikzpicture}[scale=0.7]
				\draw
				(-.7,2) node (lo2) [ivertexgrey,label=right:{$\tilde{x}_u$}] {}
				(-.7,1) node (lo) [ivertexgrey,label=above left:{$x_u$}] {}
				(-.7,0) node (lm) [ivertexblack,label=left:{$u$}] {}
				(-.7,-1) node (lu) [ivertexgrey,label=below left:{$y_u$}] {}
				(-.7,-2) node (lu2) [ivertexgrey,label=right:{$\tilde{y}_u$}] {}
				(.3,1) node (ro) [ivertexgrey,label=right:{$x_v$}] {}
				(.3,0) node (rm) [ivertexblack,label=right:{$v$}] {}
				(.3,-1) node (ru) [ivertexgrey,label=right:{$y_v$}] {};
				\draw[iedge, color=tcol] (lu2)--(lu)--(lm)--(lo)--(lo2) (ru)--(rm)--(ro);
				\draw[iedge] (lm)--(rm);
				\draw[iedge] (lo) to[bend right=70] (lu) {};
			\end{tikzpicture}		}
		\subcaptionbox[]{$\tilde{x}_u = v$\label{fig:3-dec-hist-proof-xut}}[.225\textwidth][c]
		{
			\begin{tikzpicture}[scale=0.7]
				\draw
				(-.7,2) node (lo2) [label=right:{\phantom{$\tilde{x}_u$}}] {}
				(-.7,-2) node (lu2) [label=right:{\phantom{$\tilde{y}_u$}}] {}	
				(-.2,1) node (lo) [ivertexgrey,label=above:{$x_u = x_v$}] {}
				(-.7,0) node (lm) [ivertexblack,label=left:{$u$}] {}
				(-.7,-1) node (lu) [ivertexgrey,label=below left:{$y_u$}] {}
				(-.7,-2) node (lu2) [ivertexgrey,label=right:{$\tilde{y}_u$}] {}
				(.3,0) node (rm) [ivertexblack,label=right:{$v=\tilde{x}_u$}] {}
				(.3,-1) node (ru) [ivertexgrey,label=right:{$y_v$}] {};
				\draw[iedge, color=tcol] (lu2)--(lu)--(lm)--(lo) (ru)--(rm)--(lo);
				\draw[iedge] (lm)--(rm);
				\draw[iedge] (lo) .. controls (-1.6, 0.1) .. (lu);
			\end{tikzpicture}		}
		\caption{Illustration of the cases appearing in the proof of \Cref{thm: reducible or HIST}.
		\label{fig:3-dec-hist-proof}} 
	\end{figure}
\end{proof}

\equivalentconjectures*
\begin{proof}
	Throughout this proof, whenever a graph is given with a spanning tree we implicitly assume that the edges of the tree are green and all other edges are black.
	
	Let $G$ be a finite connected cubic graph.
	First assume that $G$ satisfies the 3-decomposition conjecture and let $(T, C, M)$ be a 3-decomposition of $G$.
	If $T$ is homeomorphically irreducible, then $G$ trivially satisfies the HIST-extension conjecture (with an empty sequence of extensions).
	Otherwise, by \Cref{thm: reducible or HIST}, there exists a graph $G'$ with a HIST $T'$ such that $(T', C, \emptyset)$ is a 3-decomposition of $G'$ and $G$ can be obtained from $G'$ (with the green-black colouring induced by~$T'$) by a finite sequence of Tutte- and diamond-extensions, that is, $G$ satisfies the HIST-extension conjecture.
	
	Now we assume that $G$ satisfies the HIST-extension conjecture.
	In particular, there exists a graph $G'$ with a homeomorphically irreducible spanning tree $T'$ such that $G$ can be obtained by a finite sequence of Tutte- and diamond-extensions from $G'$.
	Since all vertices of $T'$ are of degree~1 or~3, the graph $G-E(T')$ decomposes into a 2-regular graph $C$ and isolated vertices.
	Thus, $(T',C, \emptyset)$ is a 3-decomposition of $G'$.
	By \Cref{3-dec-preserved}, $G$ satisfies the 3-decomposition conjecture.
\end{proof}

\section{Extensions and reductions}
\label{sec:ext-red}
In the following we formalise the replacement of an induced subgraph $R$ of a cubic graph~$G$ by some other subcubic graph $S$.
Intuitively, this means that we add the graph~$S$ to $G-V(R)$ and connect the vertices of $S$ ``appropriately''.
In order to define what ``appropriately'' means, we augment $R$ to a graph $X$ by adding leaves to all vertices of $R$ of degree less than~3 such that all vertices of $R$ are of degree~3 in $X$.
We do the same for $S$ and use these additional vertices to specify the edges between $S$ and $G-V(R)$.

This leads us to the definition of a \emph{template graph} which is a graph $X$ whose vertex set is partitioned into a set of \emph{inner vertices} $I(X)$ and a set of \emph{outer vertices} $O(X)=\{v_1,\ldots,v_k\}$.
All inner vertices have degree~3 and all outer vertices have degree~1.
We call $X-O(X)$ the \emph{core of $X$} and denote it by $c(X)$.
For a subcubic graph $R$, the unique template graph $X$ (up to labelling the outer vertices) with $R$ as its core is called the template graph of $R$. 
The template graph of the domino can be seen in on the right of \Cref{fig:transformation-illustration}, where the outer vertices are grey.

We can now define the replacements we need:
let $R$, $S$ be subcubic graphs and let $X$, $Y$ be their template graphs.
Furthermore, let $G$ be a cubic graph containing an induced subgraph $R'$ isomorphic to $R$, that is, $\varphi(R) = R'$ for some isomorphism $\varphi$.
Since $G$ is cubic, any vertex $\varphi(v)\in R'$ has exactly as many incident edges to vertices in $G-V(R')$ as $v$ has neighbours in $O(X)$.
Thus, we can identify each neighbour $w$ of $v$ that is an outer vertex with a neighbour $\psi(w)$ of $\varphi(v)$ that is not in $R'$.

The graph
\begin{alignat*}{2}
	G_\varphi[X\to Y] 
	\coloneqq &(G-V(R'))\cup S     	&\,+\, \{v\psi(w)\colon vw\in Y, w\in O(Y)\} \\
	= &(G-\varphi(I(X))) \cup c(Y)  &\,+\, \{v\psi(w)\colon vw\in Y, w\in O(Y)\}
\end{alignat*}
is the \emph{$(X,Y,\varphi)$-transformation} of $G$.
(Note that the resulting graph does not depend on the choice of $\psi$.)
An \emph{$(X,Y)$-transformation} of $G$ is a graph $G[X\to Y] \coloneqq G_\varphi[X\to Y]$ for some isomorphism $\varphi\colon c(X)\to R'$ with $R'\subseteq G$.
Moreover, a pair $(X,Y)$ of template graphs with $O(X) = O(Y)$ is called an \emph{$(X,Y)$-transformation}.
Transformations are called \emph{extensions} if $|V(X)|<|V(Y)|$ and \emph{reductions} if $|V(X)| > |V(Y)|$.

We depict $(X,Y)$-transformations as seen in \Cref{fig:transformation-illustration}.
We use such illustrations to describe the transformations since they are more convenient and easier to understand than formally writing down vertex and edge sets.
We omit the labels of the outer vertices; their positioning shows which of these coincide.
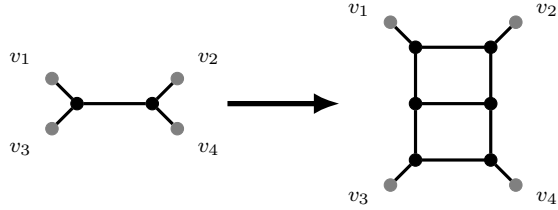
\begin{figure}[htb]
	\centering
	\begin{tikzpicture}
		\drawEdge{0}{0}{1}{B}{{black,black,black,black,black}}{black}{black!50}
		\node at (-1.25,0.6) {$v_1$};
		\node at (1.25,0.6) {$v_2$};
		\node at (-1.25,-0.6) {$v_3$};
		\node at (1.25,-0.6) {$v_4$};
		\transformationArrow{1.5}{0}{1.5}{2}
		\drawDomino{4.5}{0}{1}{A}{{black,black,black,black,black,black,black,black,black,black,black}}{black}{black!50}
		\node at (3.25,1.25) {$v_1$};
		\node at (5.75,1.25) {$v_2$};
		\node at (3.25,-1.25) {$v_3$};
		\node at (5.75,-1.25) {$v_4$};
	\end{tikzpicture}	\caption{Template graphs and transformations:
	this example illustrates the extension where $c(X)$ is the $K_2$ and $c(Y)$ is the domino.
	The outer vertices of the template graphs are grey.
	\label{fig:transformation-illustration}}
\end{figure}

Note that the outer vertices of the template graph do not necessarily correspond to distinct vertices of the graph $G$ containing the core. 
As a result, $G$ need not have a subgraph isomorphic to $X$, just to $c(X)$.
Additionally, the graphs $G[X\to Y]$ are not necessarily simple, even if $G$, $X$, and $Y$ are.
However, we only apply transformations which yield a simple graph.
A sufficient condition for this is that no vertex in the core of $Y$ has multiple adjacent outer vertices, or equivalently, that the minimum degree in the core of $Y$ is~2.

\paragraph{Compatible extensions and reducible configurations.}
A \emph{reducible configuration} is a (subcubic) graph $S$ that is not part of a 3-connected minimum counterexample to the 3-decomposition conjecture.
In this case we also say that $S$ is \emph{reducible}.

Our goal is to describe $(X,Y)$-extensions that allow us to extend 3-de\-com\-po\-si\-tions.
For ease of use, an $(X,Y)$-extension is \emph{3-compatible} if, for every cubic graph $G$ with a 3-decomposition and for every extension $H\coloneqq G[X\to Y]$ of~$G$, the graph $H$ also has a 3-decomposition.
In fact, in our proofs that certain transformations are 3-compatible we show how to construct a 3-decomposition of~$H$ from one of~$G$.

Finding a 3-compatible extension $(X,Y)$ is very helpful for proving that the core of $Y$ is reducible.
The following lemma exhibits a property that implies reducibility if it is satisfied by such an extension.
\begin{lemma}
	\label{compatible-to-reducible}
	Let $(X,Y)$ be a 3-compatible extension.
	The core $c(Y)$ is reducible if for every 3-connected cubic graph $H$ containing (a subgraph isomorphic to) $c(Y)$
	\begin{itemize}
		\item $H$ has a 3-decomposition, or,
		\item $c(Y)$ is induced and the reduction $G=H[Y\to X]$ is a simple 3-connected graph.
	\end{itemize}
\end{lemma}
\begin{proof}
	Assume $H$ is a minimal counterexample containing $c(Y)$ as an induced subgraph and let $G\coloneqq H[Y\to X]$ be both simple and 3-connected.
	Since $H$ is a minimal counterexample, $G$ has a 3-decomposition and, because $(X,Y)$ is 3-compatible, $H=G[X\to Y]$ also has a 3-decomposition, which is a contradiction.
\end{proof}

Whenever possible, we want to apply this lemma.
Therefore, in order to prove that a certain configuration $S$ is reducible, we proceed as follows:
\begin{enumerate}[(i)]
	\item Prove a certain $(X,Y)$-extension to be 3-compatible where $S=c(Y)$.
	\item Check that in any simple 3-connected cubic graph $H$ containing a subgraph~$S'$ isomorphic to $S$, $S'$ is actually induced and the $(Y,X)$-reduction of $H$ yields another simple 3-connected graph.
\end{enumerate}
The first option of the lemma lets us assume that the graph obtained from the reduction is still sufficiently large by excluding some small graphs that already have 3-decompositions.
The procedure described here is actually entirely constructive in the sense that given a graph $H$, we can perform reductions using the reducible structures as long as they are present.
Let $G$ be the resulting graph.
If we can find a 3-decomposition for this smaller graph $G$, then we can undo the reductions step by step and extend our decompositions to the larger graphs, finally obtaining one for $H$.

To use this procedure means we regularly need to check whether subgraphs are induced and whether reductions yield simple 3-connected graphs.
To facilitate this, we end this section with two lemmas that are helpful in this respect.
Their proofs use Menger's theorem~\cite{Men27,Die10} when verifying 3-connectivity.
\begin{lemma}
	\label{node-reductions}
	Let $H$ be a 3-connected cubic graph that contains the core of a template graph $Y$ with three outer vertices and let $X$ be the template graph of the~$K_1$.
	Then $c(Y)$ is induced and the reduction $G\coloneqq H[Y\to X]$ is 3-connected.
	If, additionally, $H-V(c(Y))$ consists of more than one vertex, then $G$ is simple.
\end{lemma}
\begin{proof}
	The 3-connectivity of $H$ implies that there are at least three edges joining $H-V(c(Y))$ with $c(Y)$.
	Since $Y$ has three outer vertices, there are exactly three such edges and $c(Y)$ is induced.
	Furthermore, if $H-V(c(Y))$ has more than one vertex, then the three neighbours of $c(Y)$ are distinct:
	if all three of the neighbours coincide, then $H$ is not connected, which is a contradiction. If exactly two of them coincide, then we find a 2-edge cut in $H$, which contradicts to $H$ being 3-connected.
	Consequently, $G$ is simple in this case.
	
	To see that $G$ is 3-connected, consider two distinct vertices $u$ and $v$ of $G$.
	If neither $u$ nor $v$ is the vertex in $c(X)$, then they are in $H$ and we obtain three internally vertex-disjoint paths linking them in $H$.
	At most one of these paths can use vertices of $c(Y)$ since $|E(V(H)\setminus V(c(Y)), V(c(Y)))| = 3$.
	Thus, by potentially replacing the path segment through $c(Y)$ by the vertex in $c(X)$, we obtain three paths in $G$.
	
	If $u$ is the vertex in $c(X)$, then let $u' \in V(c(Y))$.
	We obtain three vertex-disjoint $u'v$-paths in $H$, each of which uses one of the edges from $c(Y)$ to the rest of $H$.
	Thus we can replace these initial path segments by the edges incident to $u$ in $G$ and obtain the required paths there.
\end{proof}

\begin{lemma}
	\label{square-reductions}
	Let $H$ be a 3-connected cubic graph that contains the core of a template graph $Y$ with four outer vertices and $|V(c(Y))| \geq 5$.
	Moreover, let $X$ be the template graph of the square.
	Then $c(Y)$ is induced and the reduction $G\coloneqq H[Y\to X]$ is simple and 3-connected.
\end{lemma}
\begin{proof}
	The 3-connectivity of $H$ implies that there are at least $3$ edges between $c(Y)$ and $H-c(Y)$.
	Since $Y$ only has four outer vertices, $c(Y)$ is induced in $H$ and $G$ is simple because the square has minimum degree~2.
	
	To verify the 3-connectivity of $G$, we consider two vertices $u$ and $v$ of $G$.
	If both are not vertices of the square, then we obtain three internally vertex-disjoint paths linking them in $H$.
	At most two of these paths use vertices of $c(Y)$ and all other paths exist in $G$.
	If a single path passes $c(Y)$, then it can be replaced by one in the square and so can two paths unless they both need to connect non-adjacent vertices of the square.
	In this case, we can cross the paths by pairing the start of one with the end of the other and vice versa.
	In all cases we obtain the three desired disjoint paths.
	
	If one vertex, say $v$, is in the square and the other is not, then the paths obtained in $H$ consist of one path that only meets $c(Y)$ in $v$ and two more paths that start with a path segment in $c(Y)$.
	However, the square contains such path segments as well.
	Finally, if both vertices are in the square, replace $u$ its neighbour $x$ outside of the square.
	By the previous case, we obtain a path from $x$ to $v$ that only meets the square in $v$.
	We extend this path to $u$ and take the two disjoint paths the square provides.
	Thus $G$ is also 3-connected.
\end{proof}

\section{New reducible configurations}
\label{sec: reducible configurations}
In this section we prove \Cref{reducible-graphs}.
\reduciblegraphs*

For convenience, we repeat \Cref{fig:reducible-graphs} here.
\begin{figure}[htb]
	\centering
	\begin{tikzpicture}[scale=.75]
		\drawTriangle{-8.6}{-0.1}{1}{tr}{{black,black,black,black,black,black}}{black}{black!50}
		\node[] at (-8.6,-1.5) {triangle};
		
		\drawKTwoThree{-5.75}{-0.1}{1}{k33}{{black,black,black,black,black,black,black,black,black}}{black}{black!50}
		\node[] at (-5.75,-1.5) {$K_{2,3}$};
		
		\drawPetMinusV{-2.5}{-.08}{0.5}{pet}{{black,black,black,black,black,black,black,black,black,black,black,black,black,black,black}}{black}{black!50}
		\node[] at (-2.5,-1.5) {\petv};
		
		\drawGFive{0.25}{.05}{0.75}{g5}{{black,black,black,black,black,black,black,black,black,black,black,black,black,black}}{black}{black!50}
		\node[] at (.45,-1.5) {claw-square};
		
		\drawGFour{2.9}{0}{1}{g4}{{black,black,black,black,black,black,black,black,black,black,black}}{black}{black!50}
		\node[] at (2.9,-1.5) {twin-house};
		
		\drawDomino{5.3}{0}{1}{g3}{{black,black,black,black,black,black,black,black,black,black,black}}{black}{black!50}
		\node[] at (5.3,-1.5) {domino};
	\end{tikzpicture}
\end{figure}

All but one of these proofs adhere to the procedure described in the previous section and we illustrate it using the triangle.
For the remaining graphs we focus on the more challenging cases and refer to the appendix for the others.
The first step for the triangle is to find an $(X,Y)$-extension that is 3-compatible and where $c(Y)$ is the triangle.
For this we choose $X$ such that $c(X)$ is a single vertex.
\begin{lemma}
	\label{3-compatible-triangle}
	The extension shown in \Cref{fig:extension-reduction-triangle} is 3-compatible.
	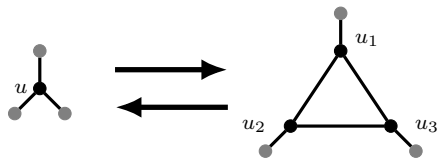
\begin{figure}[htb]
		\centering
		\begin{tikzpicture}
			\drawVertex{0}{0}{1}{B}{{black,black,black}}{black}{black!50}
			\node at (-0.25,0) {$u$};
			\transformationArrow{1}{0.25}{1.5}{2}
			\transformationArrow{2.5}{-0.25}{-1.5}{2}
			\drawTriangle{4}{0}{1}{A}{{black,black,black,black,black,black}}{black}{black!50}
			\node at (4.35,0.65) {$u_1$};
			\node at (2.85,-0.5) {$u_2$};
			\node at (5.15,-0.5) {$u_3$};
		\end{tikzpicture}		\caption{Transforming a node to a triangle.
		\label{fig:extension-reduction-triangle}}
	\end{figure}
\end{lemma}
\begin{proof}
	Let $G$ be a cubic graph with a 3-decomposition $(T,C,M)$ and let $H$ be an $(X,Y)$-extension of $G$, where the extension is the one from \Cref{fig:extension-reduction-triangle}.
	We wish to extend the decomposition of $G$ to it.
	To this end, we start by determining the possible behaviours at $u$.
	As $T$ is a spanning tree, at least one of the edges incident to $u$ is in $T$.
	If it is exactly one, then the other two are part of a cycle.
	If there are exactly two, then the missing edge is in the matching and, otherwise, there are three.
	These options are shown in \Cref{fig:behaviour-node}, up to rotational symmetry.
	Recall that edges of $T$ are green, those of $C$ red, and those in $M$ blue.
	\begin{figure}[htb]
		\centering
		\begin{tikzpicture}
			\drawVertex{0}{0}{1}{A}{{tcol,ccol,ccol}}{black}{black!50}
			\drawVertex{2}{0}{1}{A}{{tcol,tcol,mcol}}{black}{black!50}
			\drawVertex{4}{0}{1}{A}{{tcol,tcol,tcol}}{black}{black!50}
		\end{tikzpicture}		\caption{Possible behaviour of a 3-decomposition at a single node.
		\label{fig:behaviour-node}}
	\end{figure}
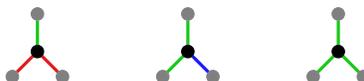
	
	We now need to turn the decomposition $(T,C,M)$ of $G$ into one of $H$, which we do by distinguishing between the possible behaviours described above.
	The extensions obtained are shown in \Cref{fig:3-dec-extension-triangle}.
	We describe how these depictions are to be read, using the first case as an example.
	The remaining ones should then be self-explanatory.
	\begin{figure}[htb]
		\centering
		\begin{tikzpicture}
			\drawVertex{0}{0}{0.9}{A}{{tcol,ccol,ccol}}{black}{black!50}
			\drawVertex{4}{0}{0.9}{A}{{tcol,tcol,mcol}}{black}{black!50}
			\drawVertex{8}{0}{0.9}{A}{{tcol,tcol,tcol}}{black}{black!50}
			\transformationArrow{0.5}{0}{1}{1.5}
			\transformationArrow{4.5}{0}{1}{1.5}
			\transformationArrow{8.5}{0}{1}{1.5}
			\drawTriangle{2.25}{0}{0.75}{B}{{tcol,tcol,ccol,tcol,ccol,ccol}}{black}{black!50}
			\drawTriangle{6.25}{0}{0.75}{C}{{mcol,tcol,tcol,tcol,tcol,mcol}}{black}{black!50}
			\drawTriangle{10.25}{0}{0.75}{D}{{mcol,tcol,tcol,tcol,tcol,tcol}}{black}{black!50}
		\end{tikzpicture}		\caption{Extending a 3-decomposition to a triangle.
		\label{fig:3-dec-extension-triangle}}
	\end{figure}
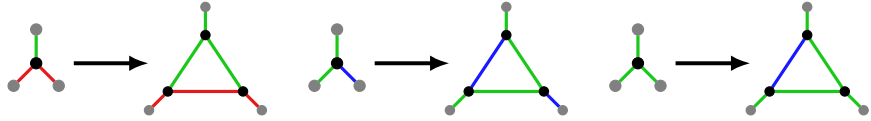
	
	If we only have a single $T$-edge, say $uv_1$, incident to $u$, we make the following extension:
	we begin by setting $T'=T-I(X)+I(Y)$, $C' = C-I(X) + I(Y)$, and $M' = M-EX$.
	These are a spanning forest, a disjoint union of cycles together with a $v_2v_3$-path, and a matching. 
	The forest has four components, three of them being the vertices $I(Y)$ of the triangle.
	
	By adding the edges $u_1v_1,\ u_1u_2$, and $u_1u_3$ to $T'$ we obtain a spanning tree of~$G$ and adding the remaining edges incident to $u_2$ and $u_3$ to $C'$ turns it a disjoint union of cycles.
	As a result $(T',C',M')$ is a 3-decomposition of $G$.
	
	The figure illustrates how to extend the decomposition for the remaining two behaviour types, giving us a 3-decomposition of $H$ in all cases, and thus showing that the extension is 3-compatible.
\end{proof}

We can now complete the second step to obtain the reducibility of the triangle.
\begin{corollary}
	\label{reducible-triangle}
	The triangle is reducible.
\end{corollary}
\begin{proof}
	Here we regard the inverse $(Y,X)$-reduction shown in \Cref{fig:extension-reduction-triangle}.
	Let $H$ be a simple 3-connected cubic graph containing a triangle $u_1u_2u_3u_1$.
	By \Cref{compatible-to-reducible} it suffices to prove that the triangle is induced and that the $(Y,X)$-reduction $G$ is simple and 3-connected.
	In fact, we may assume that $H\neq K_4$ as the $K_4$ has a 3-decomposition (for example by \Cref{small-graphs-3-dec}).
	But all these properties follow directly from \Cref{node-reductions}.
\end{proof}

Now that we have illustrated the general concept, let us fix some notation for the remainder of this section.
Whenever we regard a transformation, we denote the smaller graph by $X$ and the larger one by $Y$, unless explicitly stated otherwise.
Furthermore, when we prove that an extension is 3-compatible, we write $G$ for the smaller graph containing the core of $X$ and $H=G[X\to Y]$ is an $(X,Y)$-extension of $G$.
The 3-decomposition of $G$ is $(T,C,M)$ and we want to construct a 3-decomposition $(T',C', M')$ for $H$.

Next, we look at three further examples of 3-compatible extensions that are straight-forward.
As such the proof of the following lemma can be found in the appendix.
\begin{lemma}
	\label{3-compatible-naive-extensions}
	The extensions shown in \Cref{fig:extension-reduction-k23,fig:extension-reduction-claw-square,fig:extension-reduction-edge-domino} are 3-compatible.
\end{lemma}
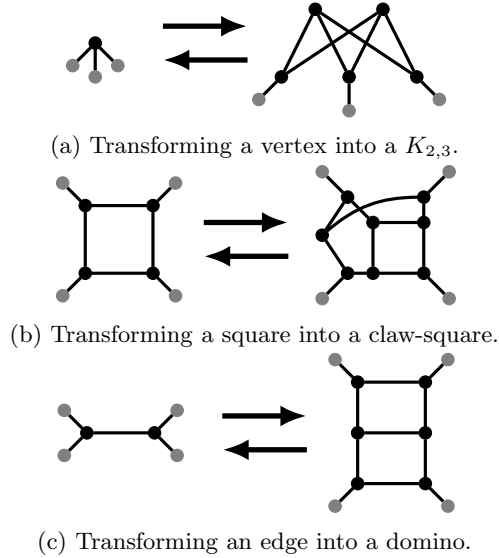
\begin{figure}[htb]
	\centering
	\subcaptionbox[]{Transforming a vertex into a $K_{2,3}$.\label{fig:extension-reduction-k23}}[\textwidth][c]
	{
		\begin{tikzpicture}[scale=.9]
			\drawVertexDown{0}{0}{1}{B}{{black,black,black}}{black}{black!50}
			\transformationArrow{1}{0.25}{1.25}{2}
			\transformationArrow{2.25}{-0.25}{-1.25}{2}
			\drawKTwoThree{3.75}{0}{1}{A}{{black,black,black,black,black,black,black,black,black}}{black}{black!50}
		\end{tikzpicture}	}
	\subcaptionbox[]{Transforming a square into a claw-square.\label{fig:extension-reduction-claw-square}}[\textwidth][c]
	{
		\begin{tikzpicture}[scale = .9]
			\drawSquare{0}{0}{1}{B}{{black,black,black,black,black,black,black,black}}{black}{black!50}
			\transformationArrow{1.25}{0.25}{1.25}{2}
			\transformationArrow{2.5}{-0.25}{-1.25}{2}
			\drawGFive{3.75}{0.25}{0.75}{A}{{black,black,black,black,black,black,black,black,black,black,black,black,black,black}}{black}{black!50}
		\end{tikzpicture}	}
	\subcaptionbox[]{Transforming an edge into a domino.\label{fig:extension-reduction-edge-domino}}[\textwidth][c]
	{
		\begin{tikzpicture}[scale = .9]
			\drawEdge{0}{0}{1}{B}{{black,black,black,black,black}}{black}{black!50}
			\transformationArrow{1.5}{0.25}{1.25}{2}
			\transformationArrow{2.75}{-0.25}{-1.25}{2}
			\drawDomino{4}{0}{1}{A}{{black,black,black,black,black,black,black,black,black,black,black}}{black}{black!50}
		\end{tikzpicture}	}
	\caption{Three straight-forward 3-compatible extensions.
	\label{naively-extendable-extensions}}
\end{figure}

By completing the second step for the $K_{2,3}$ and the claw-square we get reducibility for the next two configurations.
\begin{corollary}
	\label{reducible-k23-pet-g5}
	The $K_{2,3}$ and the claw-square are reducible.
\end{corollary}
\begin{proof}
	Regard the $(Y,X)$-reductions in \Cref{fig:extension-reduction-k23,fig:extension-reduction-claw-square}.
	By \Cref{small-graphs-3-dec} and \Cref{compatible-to-reducible} it suffices to prove that, given a 3-connected cubic graph $H$ with $|H|>6$ containing the core of $Y$, $c(Y)$ is induced and the reduction $G\coloneqq H[Y\to X]$ is simple and 3-connected.
	This holds for the $K_{2,3}$ by \Cref{node-reductions} (since we excluded the $K_{3,3}$) and for the claw-square by \Cref{square-reductions}.
\end{proof}

Next we take a look at the \petv{} and the twin-house.
\begin{lemma}
	\label{3-compatible-petv}
	The extension shown in \Cref{fig:extension-reduction-petv} is 3-compatible.
\end{lemma}
\begin{figure}[ht]
	\centering
	\subcaptionbox{Transforming a vertex into a \petv.\label{fig:extension-reduction-petv}}[.5\textwidth][c]
	{
		\begin{tikzpicture}[scale=.95]
			\drawVertexDown{0}{0}{1}{B}{{black,black,black}}{black}{black!50}
			\node at (0,0.25) {$u$};
			\transformationArrow{1}{0.25}{1.25}{2}
			\transformationArrow{2.25}{-0.25}{-1.25}{2}
			\drawPetMinusV{3.75}{-0.25}{0.5}{A}{{black,black,black,black,black,black,black,black,black,black,black,black,black,black,black}}{black}{black!50}
			\node at (3.55,0.35) {$u_1$};
			\node at (4,0.35) {$u_2$};
			\node at (3.15,-0.2) {$u_3$};
			\node at (4.35,-0.2) {$u_4$};
			\node at (3.5,-0.75) {$u_5$};
			\node at (3.05,0.75) {$u_6$};
			\node at (4.5,0.75) {$u_7$};
			\node at (2.5,-0.5) {$u_8$};
			\node at (5,-0.5) {$u_9$};
		\end{tikzpicture}	}
	\hfill
	\subcaptionbox{Extending a local behaviour at $u$.\label{fig:extension-petv-not-naive}}[.475\textwidth][c]
	{
		\begin{tikzpicture}[scale=.95]
			\drawVertexDown{0}{0}{1}{B}{{mcol,tcol,tcol}}{black}{black!50}
			\transformationArrow{1}{0}{1.25}{2}
			\drawPetMinusV{3.75}{-0.25}{0.5}{A}{{tcol,tcol,tcol,ccol,ccol,ccol,ccol,ccol,tcol,tcol,tcol,tcol,tcol,tcol,tcol}}{black}{black!50}
		\end{tikzpicture}	}
	\caption{Transforming the \petv.}
\end{figure}
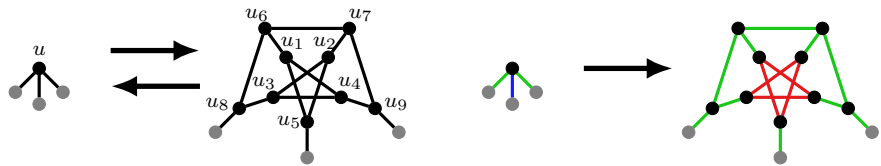
\begin{proof}
	We note that the \petv, like the triangle, is symmetric.
	Thus, we only need to check the three possible behaviours seen in \Cref{fig:behaviour-node}.
	Of these, the first and the last are straight-forward and can be found in the appendix.
	This leaves the second one, where an $M$-edge is at the single vertex, and we extend the decomposition as shown in~\Cref{fig:extension-petv-not-naive}.
	
	Note that the removal of the edge $uv_1$ and $uv_3$ from $T$ (where the vertex names are taken from \Cref{fig:extension-reduction-petv}) creates three connected components, one containing only $u$ and one with $v_2$ and $v_3$, respectively.
	By adding all edges (including $v_2u_5$) but those on the cycle $u_1u_4u_3u_2u_5u_1$ to this forest we obtain a spanning tree~$T'$ and adding the excluded cycle to $C$ yields $C'$.
	Now $(T',C',M\setminus\{v_2u_5\})$ is a 3-decomposition of $H$.
\end{proof}

The twin-house is reduced to the square, which has significantly more possible local behaviours than the $K_1$ or the $K_2$.
In order to reduce the amount of cases that need to be extended, we show that some can be transformed into one another.
Again, this is found in the appendix and we only show how to deal with the problematic cases in the proofs here.
\begin{lemma}
	\label{3-compatible-g4}
	The extension shown in \Cref{fig:extension-reduction-twin-house} is 3-compatible.
\end{lemma}
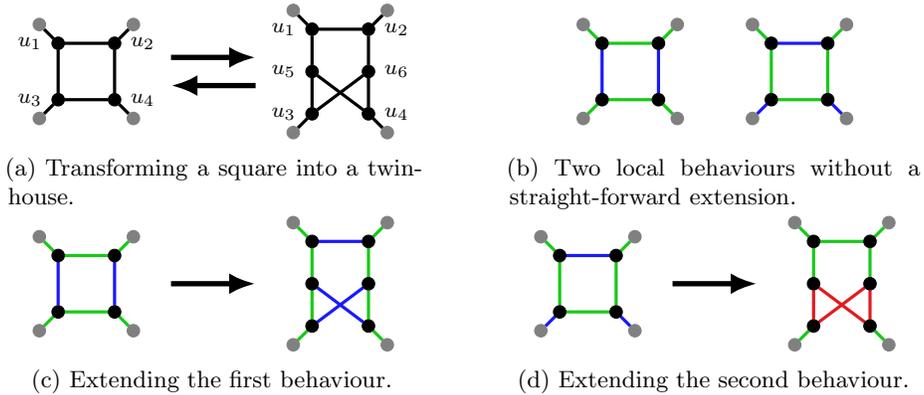
\begin{figure}[htb]
	\centering
	\subcaptionbox[]{Transforming a square into a twin-house.\label{fig:extension-reduction-twin-house}}[.45\textwidth][c]
	{
		\begin{tikzpicture}[scale=.75]
			\drawSquare{0}{0}{1}{B}{{black,black,black,black,black,black,black,black}}{black}{black!50}
			\node at (-1,0.5) {$u_1$};
			\node at (1,0.5) {$u_2$};
			\node at (-1,-0.5) {$u_3$};
			\node at (1,-0.5) {$u_4$};
			\transformationArrow{1.5}{0.25}{1.5}{2}
			\transformationArrow{3}{-0.25}{-1.5}{2}
			\drawGFour{4.5}{0}{1}{A}{{black,black,black,black,black,black,black,black,black,black,black}}{black}{black!50}
			\node at (3.5,0.75) {$u_1$};
			\node at (5.5,0.75) {$u_2$};
			\node at (3.5,0) {$u_5$};
			\node at (5.5,0) {$u_6$};
			\node at (3.5,-0.75) {$u_3$};
			\node at (5.5,-0.75) {$u_4$};
		\end{tikzpicture}	}
	\hfill
	\subcaptionbox[]{Two local behaviours without a straight-forward extension.\label{fig:bad-forests-twin-house}}[.45\textwidth][c]
	{
		\begin{tikzpicture}[scale =.75]
			\drawSquare{0}{0}{1}{A}{{tcol,mcol,mcol,tcol,tcol,tcol,tcol,tcol}}{black}{black!50}
			\drawSquare{3}{0}{1}{B}{{mcol,tcol,tcol,tcol,tcol,tcol,mcol,mcol}}{black}{black!50}
			\phantom{\drawGFour{1}{0}{1}{A}{{black,black,black,black,black,black,black,black,black,black,black}}{black}{black!50}}
		\end{tikzpicture}	}
	\subcaptionbox[]{Extending the first behaviour.\label{fig:bad-forest-twin-house-1}}[.45\textwidth][c]
	{
		\begin{tikzpicture}[scale=.75]
			\drawSquare{0}{0}{1}{A}{{tcol,mcol,mcol,tcol,tcol,tcol,tcol,tcol}}{black}{black!50}
			\transformationArrow{1.5}{0}{1.5}{2}
			\drawGFour{4.5}{0}{1}{A}{{mcol,tcol,tcol,tcol,tcol,mcol,mcol,tcol,tcol,tcol,tcol}}{black}{black!50}
		\end{tikzpicture}	}
	\hfill
	\subcaptionbox[]{Extending the second behaviour.\label{fig:bad-forest-twin-house-2}}[.45\textwidth][c]
	{
		\begin{tikzpicture}[scale =.75]
			\drawSquare{0}{0}{1}{A}{{mcol,tcol,tcol,tcol,tcol,tcol,mcol,mcol}}{black}{black!50}
			\transformationArrow{1.5}{0}{1.5}{2}
			\drawGFour{4.5}{0}{1}{A}{{tcol,tcol,tcol,ccol,ccol,ccol,ccol,tcol,tcol,tcol,tcol}}{black}{black!50}
		\end{tikzpicture}	}
	\caption{Transforming the twin-house.}
\end{figure}
\begin{proof}
	For this extension, we regard the local behaviours in~\Cref{fig:bad-forests-twin-house} and refer to the appendix for the remaining ones.
	
	For the first forest, we note that the removal of the edge $u_1u_2$ (using the terminology from \Cref{fig:extension-reduction-twin-house}) from $T$ creates two connected components, one containing $u_1$ and one containing $u_2$.
	By symmetry we may assume that the component containing $u_2$ also contains $u_3$ and $u_4$.
	Once we remove the edge $u_3u_4$ as well, we end up with three components.
	One of these contains $u_1$ whereas the other two either contain $u_2,u_3$ and $u_4$ or $u_2,u_4$ and $u_3$.
	The first assignment we now make is illustrated in \Cref{fig:bad-forest-twin-house-1}, the second is analogous.
	By replacing the square by the twin-house and adding the edges $u_1u_5, u_5u_3, u_2u_6, u_6u_4$ or $u_1u_5, u_5u_4, u_2u_6, u_6u_3$ to the forest, we obtain a spanning tree $T'$ of $H$.
	The missing edges are $u_1u_2, u_5u_4, u_6u_3$ or $u_1u_2, u_5u_3, u_6u_4$, which form a matching $M'$ together with $M\setminus\{u_1u_3, u_2u_4\}$.
	Thus $(T',C,M')$ is a 3-decomposition of the extended graph.
	
	For the second forest, removing the three edges $u_1u_3, u_3u_4, u_4u_2$ from $T$ yields a spanning forest with four components, two of which are the isolated vertices $u_3, u_4$.
	By replacing the square by the twin-house and adding the edges $u_1u_2, u_1u_5, u_2u_6$ to the forest, we obtain a spanning tree of $G-\{u_3,u_4\}$.
	We can connect these last two vertices by using their incident edges in $M$ to obtain a spanning tree $T'$ of $G$ as shown in \Cref{fig:bad-forest-twin-house-2}.
	We then take $M'$ to be the set $M$ without the edge $u_1u_2$ and the two edges we just used, which is still a matching.
	The remaining edges in the twin-house form a $C_4$ which we add to $C$ to obtain $C'$.
	Then $(T',C',M')$ is a 3-decomposition of $G$.
\end{proof}

We now obtain reducibility for both.
\begin{corollary}
	\label{reducible-g4}
	The \petv{} and the twin-house are reducible.
\end{corollary}
\begin{proof}
	The $(Y,X)$-reductions in \Cref{fig:extension-reduction-petv} and \Cref{fig:extension-reduction-twin-house} yield simple 3-con\-nect\-ed graphs by \Cref{node-reductions,square-reductions}.
	By \Cref{compatible-to-reducible} both graphs are reducible.	
\end{proof}

This just leaves the domino, whose reducibility is by far the most difficult to prove.
If we regard the reduction to the square again, then all but one case is straight-forward, which is the one shown on the left in \Cref{fig:bad-domino}.
However, this case is difficult to remedy and requires us to know more about the structure of the graph $G$ in which this occurs.
To obtain this information, we first attempt to reduce the domino to a single edge, for which the corresponding extension is 3-compatible by \Cref{3-compatible-naive-extensions}.
If this reduction yields a 3-connected graph, then we are done.
Otherwise we have acquired enough information to deal with the only problematic case occurring in the reduction to the square.
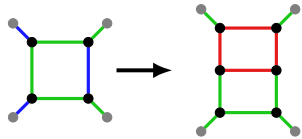
\begin{figure}
	\centering
	\begin{tikzpicture}
		\drawSquare{9}{-8}{0.75}{15}{{tcol,tcol,mcol,tcol,mcol,tcol,mcol,tcol}}{black}{black!50}
		\transformationArrow{9.75}{-8}{0.75}{1.5}
		\drawDomino{11.5}{-8}{0.75}{A}{{ccol,ccol,ccol,ccol,tcol,tcol,tcol,tcol,tcol,tcol,tcol}}{black}{black!50}
	\end{tikzpicture}	\caption{The bad local behaviour for the domino.
	\label{fig:bad-domino}}
\end{figure}

\begin{lemma}
	\label{3-compatible-square-domino}
	Let $(X,Y)$ be the extension and $(X,Z)$ be the reduction of the square shown in \Cref{fig:partly-compatible-square-domino}.
	Moreover, let $G$ be a 3-connected graph containing the core of $X$ such that its $(X,Z)$-reduction is not 3-connected.
	If $G$ has a 3-decomposition, then also the $(X,Y)$-extension of $G$ has a 3-decomposition.
	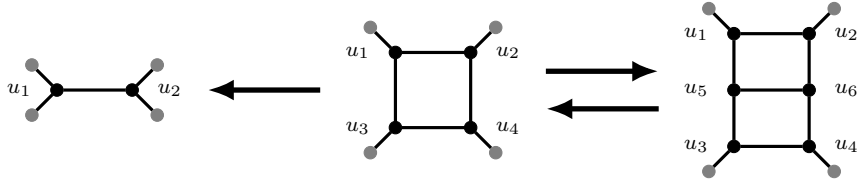
\begin{figure}[htb]
		\centering
		\begin{tikzpicture}
			\drawEdge{0}{0}{1}{C}{{black,black,black,black,black}}{black}{black!50}
			\node at (-1,0) {$u_1$};
			\node at (1,0) {$u_2$};
			\transformationArrow{3}{0}{-1.5}{2}
			\drawSquare{4.5}{0}{1}{B}{{black,black,black,black,black,black,black,black}}{black}{black!50}
			\node at (3.5,0.5) {$u_1$};
			\node at (5.5,0.5) {$u_2$};
			\node at (3.5,-0.5) {$u_3$};
			\node at (5.5,-0.5) {$u_4$};
			\transformationArrow{6}{0.25}{1.5}{2}
			\transformationArrow{7.5}{-0.25}{-1.5}{2}
			\drawDomino{9}{0}{1}{A}{{black,black,black,black,black,black,black,black,black,black,black}}{black}{black!50}
			\node at (8,0.75) {$u_1$};
			\node at (10,0.75) {$u_2$};
			\node at (8,0) {$u_5$};
			\node at (10,0) {$u_6$};
			\node at (8,-0.75) {$u_3$};
			\node at (10,-0.75) {$u_4$};
		\end{tikzpicture}		\caption{Transforming a square into a domino or an edge.
		\label{fig:partly-compatible-square-domino}}
	\end{figure}
\end{lemma}
\begin{proof}
	Let $G$ be a graph as in the claim and let $(T,C,M)$ be a 3-decomposition of~$G$.
	Furthermore let $H$ be the extension $G[X\to Y]$ and $H'$ the reduction $G[X\to Z]$.
	If any local behaviour other than the one in \Cref{fig:bad-domino} occurs, we can extend the decomposition of $G$ to $H$.
	Thus, we restrict ourselves to the forest $T_X$ with $E(T_X)=\{v_2u_2, u_2u_1, u_1u_3, u_3u_4, u_4v_4\}$ (where we use the names from \Cref{fig:partly-compatible-square-domino} and denote the outer vertex incident to $u_i$ by $v_i$).
	
	If we remove the three edges of $T_X$ in the square, then we end up with four components, the vertices $u_1$ and $u_3$ and components $u_2\in H_2, u_4\in H_4$.
	If $v_3\notin V(H_4)$, then we take all edges except those on the cycle $u_1u_2u_6u_5u_1$ to be part of the tree component, as shown in \Cref{fig:bad-domino}.
	This results in a spanning tree since each of the five edges added connects different components and its complement is the matching $M\setminus\{u_1v_1, u_3v_3, u_2u_4\}$ and the cycles in $C$ together with $u_1u_2u_6u_5u_1$.
	This lets us assume that $v_3\in V(H_4)$.
	Note that a symmetric assignment (leaving the cycle $u_5u_6u_4u_3u_5$) lets us assume that $v_1\in V(H_2)$.
	Now let $C_1$ and $C_3$ be the unique cycles in $T+u_1v_1$ and $T+u_3v_3$, which are disjoint in the situation we are in now.
	
	We claim this shows that $H'$ is 3-connected, contrary to our assumption.
	The graph $H'$ is simple since the neighbours of $u_1$ are distinct, they are in different components of $T-u_1u_3$.
	The same holds for the neighbours of $u_2$.
	If $H'$ were not 3-connected, then a partition $(A',B')$ of $V(H')$ exists such that $E_{H'}(A',B')$ contains at most two edges.
	We now regard the partition $(A,B)$ of $V(G)$ obtained by adding $u_3$ to the set containing $u_1$ and $u_4$ to the set containing $u_2$.
	Next, we compare the sets $E_{H'}(A',B')$ and $E_{H}(A,B)$.
	An edge $xy\in E_{H'}(A',B')$ is also in $E_{H}(A,B)$ if neither of its ends is $u_1$ or $u_2$.
	If just one of them is, say $x$, then it corresponds to a unique edge in $E_{H'}(A',B')$, where $x$ is potentially replaced by the vertex $u_3$ or $u_4$.
	Only the edge $u_1u_2$ corresponds to multiple edges in $E_{H'}(A',B')$, namely to the edges $u_1u_2$ and $u_3u_4$.
	
	Since $G$ is 3-connected, we get that $|E_{H'}(A',B')| = 2,\ |E_{H}(A,B)|=3, \ u_1u_2\in E_{H'}(A',B')$, and $u_1u_2, u_3u_4\in E_{H}(A,B)$.
	This lets us assume that $u_1,u_3\in A$ and $u_2,u_4\in B$.
	Both cycles $C_1$ and $C_3$ contain an element of $A$ and one of $B$, resulting in a cut of these cycles.
	These cuts have at least two edges crossing, giving us a total of at least four since the cycles are disjoint, which is a contradiction.
\end{proof}

\begin{corollary}
	\label{reducible-domino}
	The domino is reducible.
\end{corollary}
\begin{proof}
	In light of the new situation, we cannot just apply Lemma~\ref{compatible-to-reducible}.
	Let $H$ be a minimal counterexample containing the domino, that is, $H$ is a 3-connected cubic graph without a 3-decomposition.
	Since all cubic graphs on six vertices have a 3-decomposition by \Cref{small-graphs-3-dec}, we may assume that $H$ has more vertices than the ones in the domino.
	Also, because $H$ is 3-connected, we may assume that the domino is induced.
	
	If we can replace the domino by a single edge without harming 3-connectivity as shown in \Cref{fig:extension-reduction-edge-domino}, then we obtain a 3-decomposition of $H$ by \Cref{3-compatible-naive-extensions}, a contradiction.
	(Note that all 3-connected cubic graphs are simple.)
	
	So we may assume that this is not the case.
	By replacing the domino by a square as shown in \Cref{fig:partly-compatible-square-domino}, we obtain a new graph $G$ which is simple and 3-connected by \Cref{square-reductions}.
	By minimality of $H$, $G$ has a 3-decomposition and satisfies the premise of \Cref{3-compatible-square-domino} (since the reduction of the square to an edge corresponds to the reduction of the domino in $H$ to an edge directly).
	This gives us a contradiction.
\end{proof}

\section{Na\"ive extensions and their complexity}
\label{sec: naive and complex}
Our goal in this section is twofold.
First we describe a method of determining whether an extension is 3-compatible.
In the case that it is not, we get a reduced amount of local behaviours that need to be checked manually.
This greatly reduces the amount of work required in the proofs in the previous section, in particular, all the cases in the appendix are covered by this method.
Finally, we prove that the problem solved by this method is \NP-complete.

\paragraph{Algorithmically checking 3-compatibility.} 
For the first part, assume we want to prove that an $(X,Y)$-extension is 3-compatible for template graphs $X,Y$.
Let $T_X$ be a spanning forest of $X$.
We say that $T_X$ is \emph{3-consistent} if
\begin{itemize}
	\item every component of $T_X$ contains an outer vertex and
	\item every component of $X-E(T_X)$ is a single vertex, a single edge, a cycle, or a path joining two outer vertices of $X$.
\end{itemize}
We denote the set of edges that form a component of $X-E(T_X)$ by $M_X$ and write $C_X$ for the union of the cycles in $X-E(T_X)$.
See \Cref{fig:3-consistent-forests-and-naive-extensions} for two examples of 3-consistent forests.

\begin{observation}
	\label{restriction-consistent}
	The requirements for a forest to be 3-consistent are all necessary for it to be a restriction of a 3-decomposition.
	In other words, if a cubic graph $G$ contains $c(X)$ as an induced subgraph and has a 3-decomposition $(T,C,M)$, then the restriction $T_X$ of $T$ to $X$ is 3-consistent.
\end{observation}
Recall that $X$ is not necessarily a subgraph of $G$ (just the core of $X$ is), but we treat it as one since its edges can be identified with ones of $G$.
Thus, the restriction of $T$ to $X$ is the subgraph of $X$ containing the edges that correspond to ones in $T$.
Further, we remark that the restriction of $T$ to $X$ is not necessarily connected, and the restrictions of $C$ and $M$ to $X$ are $C_X$ and $M_X$, where $C_X$ may now contain paths in addition to cycles.

An \emph{assignment} is a function $f\colon O(X) \to \{\texttt{t},\texttt{c},\texttt{m}\}$. 
We say that $T_X$ \emph{realises}~$f$ if for each vertex $v\in O(X)$ its incident edge $e$ satisfies that $e\in T_X$ ($e\in C_X$, $e\in M_X$) if $f(v) = \texttt{t}$ ($f(v) = \texttt{c}$, $f(v) = \texttt{m}$).
\begin{definition}
	A 3-consistent forest $T_X$ is \emph{na\"ively extendable} to $Y$ if there exists a 3-consistent forest $T_Y$ of $Y$ such that
	\begin{itemize}
		\item $T_Y$ realises the same assignment as $T_X$ and
		\item two outer vertices of $X$ are in the same component of $T_X$ if and only if they are in the same component of $T_Y$.
	\end{itemize}
\end{definition}
An example of a na\"ively extendable forest is shown in \Cref{fig:3-consistent-forests-and-naive-extensions}.
\begin{figure}[htb]
	\centering
	\begin{tikzpicture}
		\drawEdge{0}{0}{1}{1}{{ccol,ccol,ccol,tcol,tcol}}{black}{black!50}
		\transformationArrow{1.25}{0}{0.75}{1.5}
		\drawDomino{3}{0}{1}{A}{{ccol,tcol,tcol,mcol,tcol,tcol,mcol,ccol,ccol,tcol,tcol}}{black}{black!50}
	\end{tikzpicture}	\caption{An example of a 3-consistent forest for the template graph of the edge that is na\"ively extendable to the template graph of the domino.
	\label{fig:3-consistent-forests-and-naive-extensions}}
\end{figure}
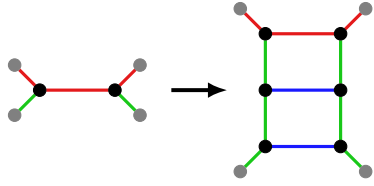

We show that a naively extendable forest $T_X$ allows us to extend a 3-decomposition that locally behaves like $T_X$.
\begin{lemma}
	\label{naive-extensions-decompositions}
	If $G$ has a 3-decomposition $(T,C,M)$ such that the restriction~$T_X$ of $T$ to $X$ is na\"ively extendable to $Y$ with forest $T_Y$, then the extension $H=G[X\to Y]$ has a 3-decomposition.
\end{lemma}
\begin{proof}
	We verify that the decomposition $(T',C',M')$ of $H$ which we now define is, indeed, a 3-decomposition of $H$.
	The graph $T'$ contains all vertices of $H$ and the edges in $E(T)\setminus E(T_X)\cup E(T_Y)$.
	Similarly, $E(C')$ is the set $E(C)\setminus E(C_X)\cup E(C_Y)$ (and $C'$ contains all vertices of $H$ that are incident to one of these edges).
	Finally, $M'\coloneqq M\setminus M_X\cup M_Y$.
	
	We first prove that $T'$ is a spanning tree.
	To show that $T'$ is connected, fix a vertex $v\in H-V(c(Y)) = G-V(c(X))$.
	There exists a path $P$ in $T$ from this vertex to every other vertex in $G-V(c(X))$.
	If $P$ does not use vertices of $c(X)$, then it is present in $T'$.
	Otherwise $P$ contains a subpath $P'$ which links two outer vertices of $X$.
	This places them in the same component of $T_X$ and they are connected in $T_Y$.
	Thus we can replace every such subpath in $T_X$ by one in $T_Y$ to obtain a path in $T'$.
	Any vertex $w\in c(Y)$ is in the same component of $T_Y$ as an outer vertex, which is connected to $v$, making $T'$ connected.
	
	Suppose now that $T'$ contains a cycle $K$.
	It cannot be contained in $H-V(c(Y))$ and, hence, $K$ consists of paths using edges in $E(T)\setminus E(T_X)$ and paths using edges in $E(T_Y)$.
	By choosing a cycle that consists of a minimal amount of such segments, we can ensure that it enters every component of $T_Y$ at most once.
	Any such path can be replaced by one in $T_X$, giving us a cycle in $T$ which is a contradiction.
	
	The set $M'$ is a matching since the edges that are in $M\setminus M_X$ and those in $M_Y$ are independent.
	So if $e,e'\in M'$ share an end, we may assume that $e\in M\setminus M_X$ and $e'\in M_Y$.
	Therefore, their common end is a vertex in $O(X)$.
	In this case $e'\notin M_X$, giving us $e'\notin M_Y$ since $f(e')\neq \texttt{m}$, where $f$ is the assignment realised by $T_X$, which is a contradiction.
	
	Finally, for $v\in H$, we take a look at the number of edges in $C'$ incident to $v$.
	If $v\in c(Y)$ or $v$ is in $H$ but not adjacent to $c(Y)$, then this degree is either~2 or~0.
	This just leaves the vertices in $H$ adjacent to $c(Y)$ or, equivalently, those of $G$ adjacent to the $c(X)$.
	Since $T_X$ and $T_Y$ realise the same assignment, any edge from a vertex of $G-c(X)$ to $c(X)$ is in $C_X$ if and only if the corresponding edge from $H-c(Y)$ to $c(Y)$ is in $C_Y$.
	Consequently, the degree of these vertices remains the same in $C$ and $C'$, yielding a degree of~0 or~2 here as well. 
\end{proof}
The proof of \Cref{naive-extensions-decompositions} is constructive in the sense that, given a 3-de\-com\-po\-si\-tion for $G$ as in the statement, we can construct a 3-decomposition for~$H$.
We can now automate the proofs we relegated to the appendix in the previous section.
\begin{remark}
	\label{3-compatible-proofs}
	To prove that an $(X,Y)$-extension is 3-compatible, we show that any 3-consistent forest of $T_X$ of $X$ satisfies
	\begin{enumerate}[(i)]
		\item $T_X$ is na\"ively extendable to $Y$ or
		\item any $(X,Y)$-extension $H=G[X\to Y]$ of a graph $G$ has a 3-decomposition if $G$ has a 3-decomposition $(T,C,M)$ for which the restriction of $T$ to $X$ is~$T_X$.
		\label{itm:non-naive-manual-check}
	\end{enumerate}
	This suffices as any 3-decomposition $(T,C,M)$ of a graph $G$ containing an induced $c(X)$ has a 3-consistent restriction $T_X$ of $T$ to $X$ by \Cref{restriction-consistent}.
	If $T_X$ is na\"ively extendable, then $H$ has a 3-decomposition by \Cref{naive-extensions-decompositions} and otherwise such a decomposition exists by \eqref{itm:non-naive-manual-check}.
	
	The task of determining all 3-consistent forests and checking whether they are na\"ively extendable can be accomplished algorithmically.
	We have implemented this
	\footnote{\label{gitlab}
		GitLab repository containing 3-decompositions for connected cubic graphs of order at most~20 and code to check na\"ive extensions,
		\url{https://gitlab.rlp.net/obachtle/reductions-for-the-3-decomposition-conjecture},
		February 2022.
	}
	and the local behaviours covered in \Cref{sec: reducible configurations} are exactly the ones that are not na\"ively extendable, that is, where \eqref{itm:non-naive-manual-check} applies (disregarding symmetric cases).
\end{remark}

\paragraph{The complexity of na\"ive extensions.}
We now prove that the existence of a na\"ive extension is \NP-complete.
We actually restrict ourselves to the subproblem of finding a 3-consistent forest $T_Y$ that realises some assignment $f$.
The reduction is based on the one for the integral multicommodity flow problem \cite{EIS76}.
\begin{theorem}
	\label{3-consistent-realisation-np-complete}
	Given a template graph $Y$ and an assignment $f$, checking whether a 3-consistent forest $T_Y$ for $Y$ exists that realises $f$ is \NP-complete.
\end{theorem}
\begin{proof}
	The problem is in \NP{} since a spanning forest $T_Y$ of $Y$ can be checked to be 3-consistent and realise $f$ in polynomial time.
	To prove \NP-hardness we use a reduction from \threesat.
	Let $F$ be a formula in \textsf{3CNF} with variables $x_1,\ldots, x_n$ and clauses $C_1,\ldots, C_m$.
	We assume that for every variable $x$ both $x$ and $\bar{x}$ occur equally often in the clauses.
	This is without loss of generality since we can add additional clauses of the form $x\lor x\lor \bar{x}$ or $x\lor \bar{x}\lor \bar{x}$ as required to reach the desired state.
	
	We first note that we can force edges $uv$ to end up in the tree part $T_Y$ by subdividing them and adding a leaf to the subdivision vertex whose incident edge is in the matching.
	Formally, we remove $uv$ and add edges $uw$, $vw$, $wy$, where $w,y$ are new vertices and $y$ is an outer vertex.
	By setting $f(y)=\texttt{m}$, $wy$ needs to be in $M_Y$ and both $uw$ and $vw$ are necessarily in $T_Y$.
	For simplicity, we refer to such a gadget as a \emph{forced tree-edge} and colour them in green in our illustrations (without depicting the additional vertices they contain).
	
	With this gadget at hand, we can describe the clause and variable gadgets we construct.
	Clause gadgets are very simple and shown in \Cref{fig:clause-gadget}.
	They just consist of a vertex $c_j$ incident to three forced tree-edges, whose ends are in the variable gadgets (we describe where momentarily).
	Each such edge corresponds to one of $C_j$'s literals.
	
	The variable gadget for $x_i$ consists of two paths of length $4\ell_i + 1$ which coincide in their ends $s_i$ and $t_i$ and are disjoint otherwise.
	Here $\ell_i$ is the amount of occurrences of $x_i$ (or $\bar{x}_i$) in $F$.
	We call one of these paths the \emph{upper path} and the other the \emph{lower} one.
	In this gadget, we partition the $4\ell_i$ distinct vertices on the upper and lower paths into $\ell_i$ blocks of four consecutive vertices and each block on the upper path corresponds to an occurrence of $x_i$ while each block on the lower path corresponds to an $\bar{x_i}$ in some clause.
	Furthermore, the $\ell$th vertices of the upper and lower path are connected by a forced tree-edge for all even $\ell$.
	This is visualised in \Cref{fig:var-gadget}.
	
	The complete graph now concatenates all the variable gadgets by adding edges from $s_i$ to $t_{i+1}$ for $i\in\{1.\ldots,n-1\}$.
	It adds a leaf to $s_1$ and $t_n$, where the edges are required to be part of $C_Y$ (by setting the $f$-value of the outer vertices to~\texttt{c}).
	A forced tree-edge in a clause gadget corresponding to a literal~$L$ is connected to the third vertex in the block that corresponds to~$L$.
	Finally, all the first vertices in a block are connected by a forced tree-edge to a path consisting of only forced tree-edges that ends in two more leaves.
	These two leaves are assigned an $f$-value of~\texttt{t}.
	The construction of $Y$ runs in polynomial time, so we just need to prove that $F$ is satisfiable if and only if there exists a 3-consistent forest $T_Y$ realising $f$.
	
	\begin{figure}
		\begin{subfigure}{.7\textwidth}
			\centering
			\begin{tikzpicture}
				\node[node,label=left:{$s_i$}] (s) at (0,0) {};
				\node[node,label=right:{$t_i$}] (t) at (7,0) {};
				
				\node[node] (u1) at (1,1) {};
				\node[node] (u2) at (2,1) {};
				\node[node] (u3) at (3,1) {};
				\node[node] (u4) at (4,1) {};
				\node[node] (u5) at (6,1) {};
				
				\node[node] (l1) at (1,-1) {};
				\node[node] (l2) at (2,-1) {};
				\node[node] (l3) at (3,-1) {};
				\node[node] (l4) at (4,-1) {};
				\node[node] (l5) at (6,-1) {};
				
				\draw[edge] (s) to (u1) to (u2) to (u3) to (u4);
				\draw[dotted] (u4) to (u5);
				\draw[edge] (u5) to (t);
				\draw[edge] (s) to (l1) to (l2) to (l3) to (l4);
				\draw[dotted] (l4) to (l5);
				\draw[edge] (l5) to (t);
				
				\draw[edge,tcol] (u2) to (l2);
				\draw[edge,tcol] (u4) to (l4);
				\draw[edge,tcol] (u5) to (l5);
				
				\draw [decorate,decoration={brace,amplitude=10pt}] (0.9,1.1) to node [black,midway,yshift=0.6cm] {an occurrence of $x_i$} (4.1,1.1);
				\draw [decorate,decoration={brace,amplitude=10pt, mirror}] (0.9,-1.1) to node [black,midway,yshift=-0.6cm] {an occurrence of $\bar{x}_i$} (4.1,-1.1);
			\end{tikzpicture}			\caption{A variable gadget.
			\label{fig:var-gadget}}
		\end{subfigure}
		\begin{subfigure}{.25\textwidth}
			\centering
			\begin{tikzpicture}
				\node[node,label=below:{$c_j$}] (c) at (0,0) {};
				
				\node[node] (u1) at (-1,1) {};
				\node[node] (u2) at (0,1) {};
				\node[node] (u3) at (1,1) {};
				
				\draw[edge,tcol] (c) to (u1);
				\draw[edge,tcol] (c) to (u2);
				\draw[edge,tcol] (c) to (u3);
				
				\node (tbuffer) at (0,2.315) {};
				\node (bbuffer) at (0,-1.315) {};
			\end{tikzpicture}			\caption{A clause gadget.
			\label{fig:clause-gadget}}
		\end{subfigure}
		\caption{An illustration of the gadgets used in the proof of \Cref{naive-extensions-np-complete}.
		\label{fig:naive-extensions-np-complete}}
	\end{figure}
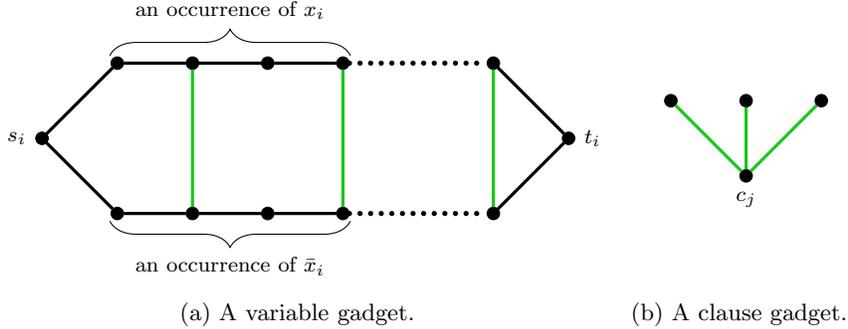
	
	First assume that a 3-consistent $T_Y$ exists that realises $f$.
	The cycle component $C_Y$ contains a path $P$ from $s_1$ to $t_n$ in the graph $Y-E(T_Y)$.
	Since $T_Y$ contains all forced tree-edges, $P$ either uses the upper or lower path in each variable gadget.
	If it uses the lower path in the gadget for $x_i$, then we set $x_i$ to \texttt{true} and otherwise we set it to \texttt{false}.
	Suppose this assignment would not satisfy~$F$.
	This yields a clause $C_j$ for which all three of its literals are not satisfied, meaning that we use the upper path for the positive literals and the lower path for the negated ones.
	Consequently, the vertex $c_j$ and its neighbours in the variable gadgets form a component of $T_Y$ without an outer vertex, contradicting the fact that $T_Y$ is 3-consistent.
	
	Conversely, assume that we are given a satisfying assignment for $F$.
	Let $T_Y$ be the forest containing the following edges:
	$T_Y$ contains all forced tree-edges.
	Additionally, if $x_i$ is set to \texttt{true}, $T_Y$ uses the upper path and otherwise it uses the lower one.
	Here uses a path means that $T_Y$ contains the path's edges incident to $s_i$ and $t_i$ and in each block it contains the first and third edge.
	
	This forest has a component for each clause gadget $C_j$ which contains the third vertices of the blocks corresponding to its literals.
	Disregarding the components consisting of a single outer vertex, there is one other component containing the path of forced tree-edges and the first (and second) vertex of all blocks.
	The cycle component $C_Y$ consists of a path which links the outer vertices at $s_1$ and $t_n$ and all other components are single edges, placing them in $M_Y$.
	As a result, $T_Y$ realises $f$ and we can make it 3-consistent by choosing a satisfied literal in every clause $C_j$ and connecting the component of $C_j$ to the one containing the path by adding the edge from the second to the third vertex in the corresponding block to $T_Y$.
\end{proof}
Using the same notation as in the result, we note that the reduction presented in \Cref{3-consistent-realisation-np-complete} can also be used to prove \Cref{naive-extensions-np-complete}.
To get this result, one simply needs to define a graph $X$ and a 3-consistent forest $T_X$ that also realises~$f$ and in which the only two outer vertices with an incident edge in~$T_X$ are in the same component of $T_X$.


\section{Applications and minimum counterexamples}
\label{sec: applications and small graphs}
In this section, we prove properties of minimum counterexamples to the 3-decomposition conjecture (under the assumption that such a counterexample exists).
To do so we exploit our new reducible configurations.

We prove that all 3-connected cubic graphs of tree-width at most~3 or of path-width at most~4 have a 3-decomposition.
Therefore minimum 3-connected counterexamples have tree-width at least~4 and path-width at least~5.
Both proofs make use of the reducible configurations we determined in \Cref{sec: reducible configurations} by showing that the restriction of the tree- or path-width causes one of them to appear.

\begin{lemma}
	\label{lem: unavoidable set for tw 3}
	Every cubic graph of tree-width at most~3 contains at least one isomorphic copy of the following graphs as a subgraph: the triangle, the $K_{2,3}$, and the domino.
\end{lemma}
\begin{proof}
	First note that cubic graphs have tree-width at least~3: smaller tree-width necessitates a vertex of degree at most~2.
	Let $G$ be a cubic graph of tree-width~3.
	Consider a \emph{smooth tree-decomposition} $(T, \mathcal{V})$ of~$G$ of width~3, that is, all bags contain exactly four vertices of~$G$ and adjacent bags have three vertices in common.
	(See~\cite{Bodlaender1998} for a proof that smooth tree-decompositions exist.)
	We may assume, by \Cref{small-graphs-3-dec}, that $G$ has more than six vertices and, therefore, $T$ has more than three.
	
	Let $r$ be a vertex of $T$ and consider $T$ as an $r$-rooted tree.
	Moreover, let $\ell$ be a leaf of $T$ of maximum distance to $r$ and $V_\ell=\{u,x_1,x_2,x_3\}$.
	By assumption, $\ell$ has a parent $f$ with $V_f=\{v,x_1,x_2,x_3\}$.
	We conclude that $N_G(u) = \{x_1,x_2,x_3\}$.
	
	If $\ell$ does not have a sibling, then regard its parent $g$.
	If $V_g = \{w,x_1,x_2,x_3\}$, then $N_G(v) = N_G(u)$ and $G$ contains a $K_{2,3}$ consisting of $u$, $v$, and their neighbours.
	Otherwise, without loss of generality, $V_g=\{v, w, x_2, x_3\}$ and $N_G(x_1)$ contains at least one vertex in $\{x_2, x_3\}$, yielding a triangle consisting of $u$, $x_1$, and this vertex.
	
	Consequently, we may assume that $\ell$ has sibling $\ell'$.
	If $V_{\ell'} = \{w,x_1,x_2,x_3\}$, we obtains a $K_{2,3}$ again.
	We may thus assume that the bags of all children of $f$ differ from $V_f$ in distinct vertices.
	In particular, we have $V_{\ell'} = \{v,w,x_2,x_3\}$.
	
	If $\ell$ has a further sibling $\ell''$ with $V_{\ell''} = \{v, y, x_1, x_3\}$, then we obtain a domino in $G[\{u, v, w, y, x_2, x_3\}]$.
	Otherwise, $f$ has a parent $g$ and $V_g$ differs from $V_f$ in exactly one vertex.
	If this is the vertex $x_1$, then $x_1$ has a neighbour in $\{x_2, x_3\}$, yielding a triangle together with $u$ and $x_1$.
	Similarly, we obtain a triangle when $v$ is not in $V_g$.
	Should $x_2$ be the vertex not in $V_g$, then $x_2$ has a neighbour in $\{x_1, x_3, v\}$ and each choice yields a triangle together with $x_2$ and one of $u$ or~$w$.
\end{proof}

The result for path-width is conceptually similar and the proof technique is abstracted and turned into an algorithm in~\cite{BH20}, which checks for unavoidable structures in graphs of bounded path-width.
Nevertheless, we give a proof here to keep the paper self-contained.
\begin{lemma} \label{lem: unavoidable set for pw 4}
	Every cubic graph of path-width at most~4 contains at least one isomorphic copy of the following graphs as a subgraph: the triangle, the $K_{2,3}$, the domino, the twin-house, or the claw-square.
\end{lemma}
\begin{proof}
	A graph of path-width at most~3 it is also of tree-width at most~3 and, hence, the lemma follows from~\Cref{lem: unavoidable set for tw 3}.
	
	Consequently, we assume $G$ has path-width 4 and consider a \emph{smooth path-decomposition} $(123\dots n', \mathcal{V})$ of~$G$, that is, all bags contain exactly five vertices of~$G$ and $|V_i \cap V_{i+1}| = 4$ for all $i \in \{1, \dots, n'-1\}$.
	(Again, see~\cite{Bodlaender1998} for a proof that smooth path-decompositions exist.)
	Observe that for each $i \in \{1, \dots, n'-1\}$ there exists unique vertices $v_i \in V_{i+1}\setminus V_{i}$ and $w_i \in V_i\setminus V_{i+1}$.
	In this proof, we make strong use of the following property, which is a direct consequence of the definition of path-decompositions:
	\begin{equation}\label{eq: neighbours of leaving vertices}
		\bigcup_{i=1}^j N_G(w_i) \subseteq \bigcup_{i=1}^j V_i~\text{for all}~j \in \{1, \dots, n'-1\}.
	\end{equation}
	Since the proof involves quite a few case distinctions, we refer to \Cref{fig:unavoidable set for pw 4} for an orientation of which edges are present in each case.
	Moreover, we use a lot of symmetry arguments and omit stating that these are without loss of generality every time.
	
	Let $V_1=\{v_1,u_1,u_2,u_3,u_4\}$ and $w_1=v_1$. 
	(This is a typical example of the symmetries we use: we can always rename the vertices such that this is true.)
	By \eqref{eq: neighbours of leaving vertices}, $N(v_1)\subseteq\{u_1,u_2,u_3,u_4\}$, say $N(v_1) = \{u_1,u_2,u_3\}$.
	We get two cases for $w_2$, namely $w_2=v_2$ or $w_2=u_1$.
	
	In the first case, $N(v_2) \neq N(v_1)$ gives us a $K_{2,3}$.
	Thus, $N(v_2) = \{u_2,u_3,u_4\}$.
	This time, we have three options for $w_3$: $u_1$, $u_2$, or $v_3$.
	The first and last case yield a desired subgraph:
	If $w_3=u_1$, then $N(u_1)$ contains a vertex in $\{u_2, u_3, u_4\}$.
	The first two options give us a triangle, the last a twin-house.
	For $w_3=v_3$ we get a $K_{2,3}$ or $N(v_3) = \{u_1,u_2,u_4\}$ and a domino is present.
	
	This puts us in the case that $w_3 = u_2$, which gets a new neighbour in $\{v_3, u_1, u_3, u_4\}$.
	All but $v_1$ yield a triangle, so we are done or $u_1v_3\in E(G)$.
	We now get four cases for $w_4$: $u_1$, $u_3$, $v_3$, or $v_4$.
	If $w_4=u_1$, $u_1$ has a neighbour in $\{v_3, u_3, u_4\}$, yielding a domino, a triangle, or a twin-house.
	If $w_4=v_3$, $v_3$ has a neighbour in $\{u_1, u_3, u_4\}$, yielding a domino, a twin-house, or a domino.
	If $w_4=v_4$, $v_4$ has three neighbours in $\{u_1, u_3, u_4, v_3\}$.
	From $u_3\in N(v_4)$ we get a domino if the edge $v_4u_1$ or $v_4u_4$ is present, one of which must be.
	Otherwise, $N(v_4)= \{u_1, u_4, v_3\}$ and the graph contains a claw-square.
	
	The last option is $w_4 = u_3$, in which we have $u_3v_4\in E(G)$ since $u_3u_4$ and $u_3u_1$ result in triangles and $u_3v_3$ yields a $K_{2,3}$.
	By regarding $w_5$, for which only two options $w_5=v_3$ or $w_5=v_5$ remain, we complete the branch of the proof.
	We get that $v_3v_4$, $v_3u_1$, and $v_3u_4$ give us a twin-house, and two dominoes, finishing the case $w_5=v_3$.
	For $w_5=v_5$, choosing any three of the possible neighbours (all choices are symmetric) creates a claw-square.
	
	Now we go back to our first split and deal with the case that $w_2=u_1$.
	We get a triangle or $N(u_1)=\{v_1, v_2, u_4\}$.
	This gives two choices for $w_3$ here, namely $v_3$ or $v_2$.
	The first case gives $N(v_3) = \{u_2, u_3, u_4\}$ and a twin-house is present.
	If $w_3 = v_2$, then we need two neighbours in $\{u_2, u_3, u_4, v_3\}$.
	We can exclude $u_4$ as if forms a triangle, and both $u_2$ and $u_3$ yields a $K_{2,3}$.
	Thus $N(v_2) = \{u_2, v_3\}$ and the obtained case is actually one we have already seen before, see \Cref{fig:unavoidable set for pw 4}.
\end{proof}

\begin{landscape}
	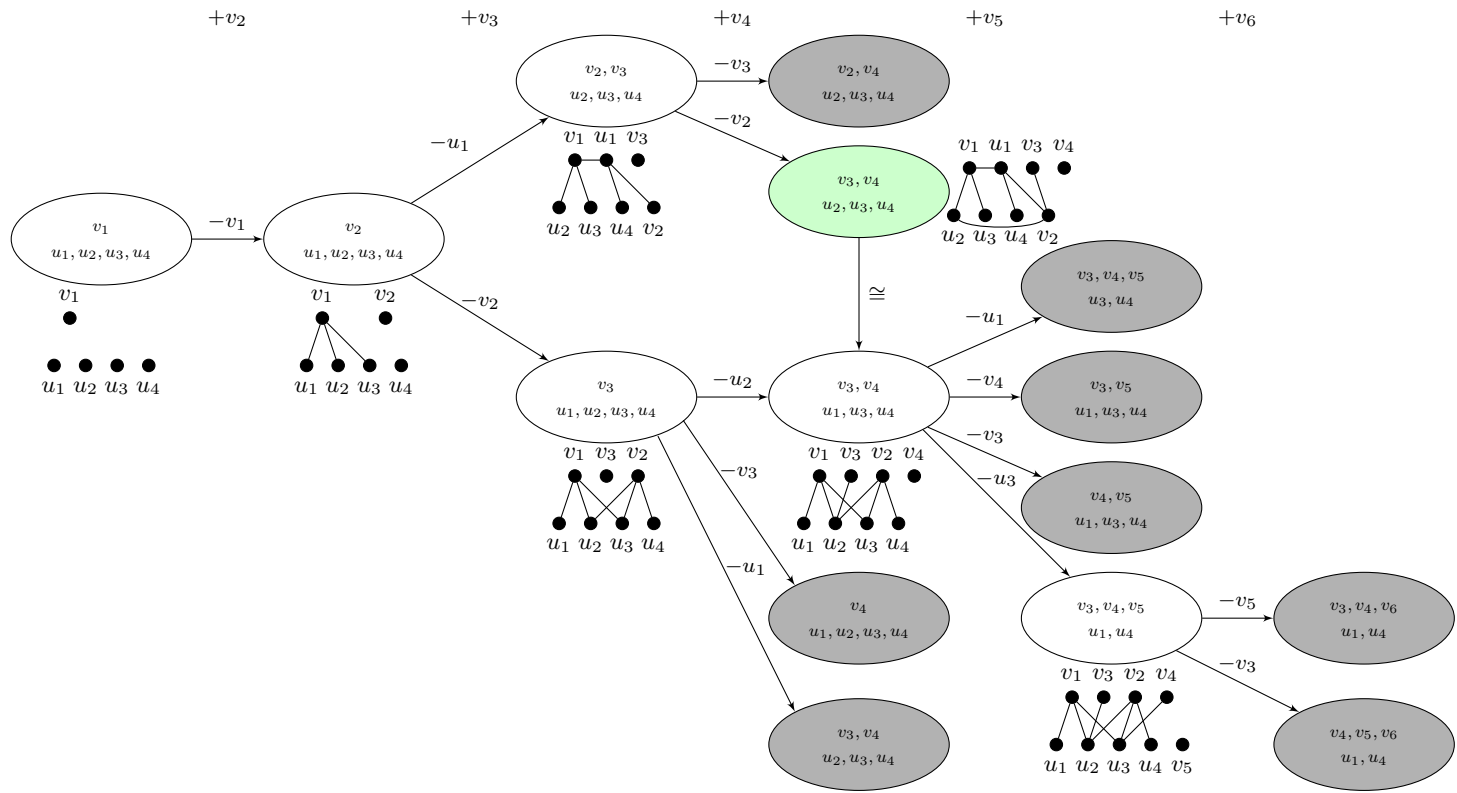
\begin{figure}
		\centering
		\begin{tikzpicture}[scale=0.84]
			\drawBag{(0,0)}{V1}{$v_1$}{$u_1,u_2,u_3,u_4$}{white};
			\drawBag{(4,0)}{V2}{$v_2$}{$u_1,u_2,u_3,u_4$}{white};
			\drawBag{(8,-2.5)}{V3a}{$v_3$}{$u_1,u_2,u_3,u_4$}{white};
			\drawBag{(8,2.5)}{V3b}{$v_2,v_3$}{$u_2,u_3,u_4$}{white};
			\drawBag{(12,-6)}{V4aa}{$v_4$}{$u_1,u_2,u_3,u_4$}{black!30};
			\drawBag{(12,-8)}{V4ab}{$v_3,v_4$}{$u_2,u_3,u_4$}{black!30};
			\drawBag{(12,-2.5)}{V4ac}{$v_3,v_4$}{$u_1,u_3,u_4$}{white};
			\drawBag{(12,0.75)}{V4ba}{$v_3,v_4$}{$u_2,u_3,u_4$}{green!20};
			\drawBag{(12,2.5)}{V4bb}{$v_2,v_4$}{$u_2,u_3,u_4$}{black!30};
			\drawBag{(16,-6)}{V5}{$v_3,v_4,v_5$}{$u_1,u_4$}{white};
			\drawBag{(16,-2.5)}{V5b}{$v_3,v_5$}{$u_1,u_3,u_4$}{black!30};
			\drawBag{(16,-4.25)}{V5c}{$v_4,v_5$}{$u_1,u_3,u_4$}{black!30};
			\drawBag{(16,-0.75)}{V5d}{$v_3,v_4,v_5$}{$u_3,u_4$}{black!30};
			\drawBag{(20,-6)}{V6a}{$v_3,v_4,v_6$}{$u_1,u_4$}{black!30};
			\drawBag{(20,-8)}{V6b}{$v_4,v_5,v_6$}{$u_1,u_4$}{black!30};
			
			\draw[->] (V1) to node[midway, above] {$-v_1$} (V2);
			\draw[->] (V2) to node[midway, above] {$-v_2$} (V3a);
			\draw[->] (V2) to node[midway, above left] {$-u_1$} (V3b);
			\draw[->] ($(V3a.south east)+(0.2,0.15)$) to node[midway, above=2mm] {$-v_3$} (V4aa.north west);
			\draw[->] ($(V3a.south east)-(0.2,0.1)$)to node[midway, above right=-2mm] {$-u_1$} (V4ab.north west);
			\draw[->] (V3a) to node[midway, above] {$-u_2$} (V4ac);
			\draw[->] (V3b) to node[midway, above] {$-v_2$} (V4ba);
			\draw[->] (V3b) to node[midway, above] {$-v_3$} (V4bb);
			\draw[->] (V4ba) to node[midway, right] {$\cong$} (V4ac);
			\draw[->] (V4ac.south east) to node[midway, above=1mm] {$-u_3$} (V5);
			\draw[->] (V4ac) to node[midway, above] {$-v_4$} (V5b);
			\draw[->] (V4ac) to node[midway, above] {$-v_3$} (V5c);
			\draw[->] (V4ac) to node[midway, above=1mm] {$-u_1$} (V5d);
			\draw[->] (V5) to node[midway, above] {$-v_5$} (V6a);
			\draw[->] (V5) to node[midway, above] {$-v_3$} (V6b);
			
			\node[] at (2,3.5) {$+v_2$};
			\node[] at (6,3.5) {$+v_3$};
			\node[] at (10,3.5) {$+v_4$};
			\node[] at (14,3.5) {$+v_5$};
			\node[] at (18,3.5) {$+v_6$};
			
			\node[circle,draw,fill=black,scale=0.5,label=above:{\small$v_1$}] (v2v) at (-0.5,-1.25) {};
			\node[circle,draw,fill=black,scale=0.5,label=below:{\small$u_1$}] (v2u1) at (-0.75,-2) {};
			\node[circle,draw,fill=black,scale=0.5,label=below:{\small$u_2$}] (v2u2) at (-0.25,-2) {};
			\node[circle,draw,fill=black,scale=0.5,label=below:{\small$u_3$}] (v2u3) at (0.25,-2) {};
			\node[circle,draw,fill=black,scale=0.5,label=below:{\small$u_4$}] (v2u4) at (0.75,-2) {};
			
			\node[circle,draw,fill=black,scale=0.5,label=above:{\small$v_1$}] (v2v) at (3.5,-1.25) {};
			\node[circle,draw,fill=black,scale=0.5,label=above:{\small$v_2$}] (v2w) at (4.5,-1.25) {};
			\node[circle,draw,fill=black,scale=0.5,label=below:{\small$u_1$}] (v2u1) at (3.25,-2) {};
			\node[circle,draw,fill=black,scale=0.5,label=below:{\small$u_2$}] (v2u2) at (3.75,-2) {};
			\node[circle,draw,fill=black,scale=0.5,label=below:{\small$u_3$}] (v2u3) at (4.25,-2) {};
			\node[circle,draw,fill=black,scale=0.5,label=below:{\small$u_4$}] (v2u4) at (4.75,-2) {};
			\draw (v2v) to (v2u1);
			\draw (v2v) to (v2u2);
			\draw (v2v) to (v2u3);
			
			\node[circle,draw,fill=black,scale=0.5,label=above:{\small$v_1$}] (v3av) at (7.5,-3.75) {};
			\node[circle,draw,fill=black,scale=0.5,label=above:{\small$v_2$}] (v3aw) at (8.5,-3.75) {};
			\node[circle,draw,fill=black,scale=0.5,label=above:{\small$v_3$}] (v3ax) at (8,-3.75) {};
			\node[circle,draw,fill=black,scale=0.5,label=below:{\small$u_1$}] (v3au1) at (7.25,-4.5) {};
			\node[circle,draw,fill=black,scale=0.5,label=below:{\small$u_2$}] (v3au2) at (7.75,-4.5) {};
			\node[circle,draw,fill=black,scale=0.5,label=below:{\small$u_3$}] (v3au3) at (8.25,-4.5) {};
			\node[circle,draw,fill=black,scale=0.5,label=below:{\small$u_4$}] (v3au4) at (8.75,-4.5) {};
			\draw (v3av) to (v3au1);
			\draw (v3av) to (v3au2) to (v3aw);
			\draw (v3av) to (v3au3) to (v3aw);
			\draw (v3aw) to (v3au4);
			
			\node[circle,draw,fill=black,scale=0.5,label=above:{\small$v_1$}] (v3bv) at (7.5,1.25) {};
			\node[circle,draw,fill=black,scale=0.5,label=above:{\small$u_1$}] (v3bu1) at (8,1.25) {};
			\node[circle,draw,fill=black,scale=0.5,label=above:{\small$v_3$}] (v3bx) at (8.5,1.25) {};
			\node[circle,draw,fill=black,scale=0.5,label=below:{\small$u_2$}] (v3bu2) at (7.25,0.5) {};
			\node[circle,draw,fill=black,scale=0.5,label=below:{\small$u_3$}] (v3bu3) at (7.75,0.5) {};
			\node[circle,draw,fill=black,scale=0.5,label=below:{\small$u_4$}] (v3bu4) at (8.25,0.5) {};
			\node[circle,draw,fill=black,scale=0.5,label=below:{\small$v_2$}] (v3bw) at (8.75,0.5) {};
			\draw (v3bv) to (v3bu1);
			\draw (v3bv) to (v3bu2);
			\draw (v3bv) to (v3bu3);
			\draw (v3bu1) to (v3bu4);
			\draw (v3bu1) to (v3bw);
			
			\node[circle,draw,fill=black,scale=0.5,label=above:{\small$v_1$}] (v4acv) at (11.375,-3.75) {};
			\node[circle,draw,fill=black,scale=0.5,label=above:{\small$v_2$}] (v4acw) at (12.375,-3.75) {};
			\node[circle,draw,fill=black,scale=0.5,label=above:{\small$v_3$}] (v4acx) at (11.875,-3.75) {};
			\node[circle,draw,fill=black,scale=0.5,label=above:{\small$v_4$}] (v4acy) at (12.875,-3.75) {};
			\node[circle,draw,fill=black,scale=0.5,label=below:{\small$u_1$}] (v4acu1) at (11.125,-4.5) {};
			\node[circle,draw,fill=black,scale=0.5,label=below:{\small$u_2$}] (v4acu2) at (11.625,-4.5) {};
			\node[circle,draw,fill=black,scale=0.5,label=below:{\small$u_3$}] (v4acu3) at (12.125,-4.5) {};
			\node[circle,draw,fill=black,scale=0.5,label=below:{\small$u_4$}] (v4acu4) at (12.625,-4.5) {};
			\draw (v4acv) to (v4acu1);
			\draw (v4acv) to (v4acu2) to (v4acw);
			\draw (v4acv) to (v4acu3) to (v4acw);
			\draw (v4acw) to (v4acu4);
			\draw (v4acu2) to (v4acx);
			
			\node[circle,draw,fill=black,scale=0.5,label=above:{\small$v_1$}] (v4bav) at (13.75,1.125) {};
			\node[circle,draw,fill=black,scale=0.5,label=above:{\small$u_1$}] (v4bau1) at (14.25,1.125) {};
			\node[circle,draw,fill=black,scale=0.5,label=above:{\small$v_3$}] (v4bax) at (14.75,1.125) {};
			\node[circle,draw,fill=black,scale=0.5,label=above:{\small$v_4$}] (v4bay) at (15.25,1.125) {};
			\node[circle,draw,fill=black,scale=0.5,label=below:{\small$u_2$}] (v4bau2) at (13.5,0.375) {};
			\node[circle,draw,fill=black,scale=0.5,label=below:{\small$u_3$}] (v4bau3) at (14,0.375) {};
			\node[circle,draw,fill=black,scale=0.5,label=below:{\small$u_4$}] (v4bau4) at (14.5,0.375) {};
			\node[circle,draw,fill=black,scale=0.5,label=below:{\small$v_2$}] (v4baw) at (15,0.375) {};
			\draw (v4bav) to (v4bau1);
			\draw (v4bav) to (v4bau2);
			\draw (v4bav) to (v4bau3);
			\draw (v4bau1) to (v4bau4);
			\draw (v4bau1) to (v4baw);
			\draw (v4bau2) .. controls (13.75,0.15) and (14.75,0.15) .. (v4baw) to (v4bax);
			
			\node[circle,draw,fill=black,scale=0.5,label=above:{\small$v_1$}] (v5v) at (15.375,-7.25) {};
			\node[circle,draw,fill=black,scale=0.5,label=above:{\small$v_2$}] (v5w) at (16.375,-7.25) {};
			\node[circle,draw,fill=black,scale=0.5,label=above:{\small$v_3$}] (v5x) at (15.875,-7.25) {};
			\node[circle,draw,fill=black,scale=0.5,label=above:{\small$v_4$}] (v5y) at (16.875,-7.25) {};
			\node[circle,draw,fill=black,scale=0.5,label=below:{\small$u_1$}] (v5u1) at (15.125,-8) {};
			\node[circle,draw,fill=black,scale=0.5,label=below:{\small$u_2$}] (v5u2) at (15.625,-8) {};
			\node[circle,draw,fill=black,scale=0.5,label=below:{\small$u_3$}] (v5u3) at (16.125,-8) {};
			\node[circle,draw,fill=black,scale=0.5,label=below:{\small$u_4$}] (v5u4) at (16.625,-8) {};
			\node[circle,draw,fill=black,scale=0.5,label=below:{\small$v_5$}] (v5z) at (17.125,-8) {};
			\draw (v5v) to (v5u1);
			\draw (v5v) to (v5u2) to (v5w);
			\draw (v5v) to (v5u3) to (v5w);
			\draw (v5w) to (v5u4);
			\draw (v5u2) to (v5x);
			\draw (v5u3) to (v5y);
		\end{tikzpicture}		\caption{A roadmap of the case distinctions made in the proof of \Cref{lem: unavoidable set for pw 4}.
		\label{fig:unavoidable set for pw 4}}
	\end{figure}
\end{landscape}

Using \Cref{lem: unavoidable set for tw 3,lem: unavoidable set for pw 4} and the reductions from \Cref{sec: reducible configurations} we can prove the following two theorems.
\twthree*
\pwfour*
\begin{proof}
	Suppose towards a contradiction that the class of all 3-connected cubic graphs of tree-width at most~3 (path-width at most~4) does not satisfy the 3-decomposition conjecture.
	Let $H$ be a counterexample of minimum order amongst all graphs of this class.
	By \Cref{lem: unavoidable set for tw 3} (\Cref{lem: unavoidable set for pw 4}) at least one of the graphs listed there is a subgraph of $H$.
	Observe that all the corresponding reductions (as discussed in \Cref{sec: reducible configurations}) produce a minor of $H$ and, hence, are tree-width (and path-width) preserving. 
	Moreover, these transformations preserve 3-connectivity.
	Altogether, $H$ can be reduced to a 3-connected graph $G$ of tree-width at most~3 (path-width at most~4) with $|V(G)| < |V(H)|$.
	Since $H$ was a minimum counterexample, the graph $G$ satisfies the 3-decomposition conjecture and, because the transformations used are 3-compatible, so does $H$, yielding a contradiction.
	(Note that for the domino, we only use the reduction to the edge if it preserves 3-connectivity and that we can extend a decomposition from the square to the domino in the case we do not reduce to the edge.)
\end{proof}

Next, we prove the following properties of minimum counterexamples.
\propertiescounterexample*
\begin{proof}
	Suppose that there exists a 3-connected counterexample to the 3-de\-com\-po\-si\-tion conjecture.
	Choose $G$ to be minimum amongst all such counterexamples.
	By \Cref{reducible-graphs}, the six graphs listed in the claim are reducible and hence not subgraphs of $G$.
	We refer to the property that the triangle is not a subgraph of $G$ by ($\Delta$) and use ($K_{2,3}$), (\petv), (cs), (th), ($\boxminus$) for the respective properties corresponding to the other subgraphs.
	Parts~\eqref{itm: trianglefree} and~\eqref{itm: induced cycles} follow immediately from ($\Delta$) and ($\boxminus$).
	We prove~\eqref{itm: P6}.
	Let $wx$ be some edge of $G$.
	
	\medskip
	\noindent
	\textbf{Claim 1.} If $wx$ is contained in some cycle of length~6, then~\eqref{itm: P6} is satisfied.
	
	\noindent
	\textit{Proof.} Let $C\coloneqq uvwxyzu$ be a cycle which contains $wx$.
	By~\eqref{itm: induced cycles} the cycle $C$ is induced and, hence, there exists a vertex $v' \in N_G(v) \setminus V(C)$.
	Consider the path $P \coloneqq v'vwxyz$ as shown on the left of \Cref{fig:min-ce-paths}.
	If $P$ is induced, then the claim is satisfied.
	Therefore, we may assume that~$P$ has a chord.
	Since $C$ is induced one end of this chord is $v'$ and by~($\Delta$) the other end is not $w$.
	It remains to consider the cases that $v'x$, $v'y$, or $v'z$ is an edge of $G$.
	
	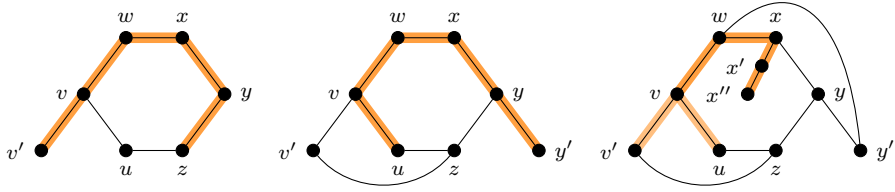
\begin{figure}
		\centering
		\begin{tikzpicture}[scale=0.75]
			\draw[draw,line width=4pt,orange,opacity=0.75,join=round] (-1.5,0) to (-0.75,1) to (0,2) to (1,2) to (1.75,1) to (1,0);
			
			\node[circle,draw,fill=black,scale=0.5,label=below:{$u$}] (u) at (0,0) {};
			\node[circle,draw,fill=black,scale=0.5,label=below:{$z$}] (z) at (1,0) {};
			\node[circle,draw,fill=black,scale=0.5,label=left:{$v$}] (v) at (-0.75,1) {};
			\node[circle,draw,fill=black,scale=0.5,label=left:{$v'$}] (v2) at (-1.5,0) {};
			\node[circle,draw,fill=black,scale=0.5,label=right:{$y$}] (y) at (1.75,1) {};
			\node[circle,draw,fill=black,scale=0.5,label=above:{$w$}] (w) at (0,2) {};
			\node[circle,draw,fill=black,scale=0.5,label=above:{$x$}] (x) at (1,2) {};
			\draw (u) to (v) to (w) to (x) to (y) to (z) to (u);
			\draw (v) to (v2);
			\phantom{
				\draw (v2) to[bend right=50] (z);
			}
		\end{tikzpicture}		\begin{tikzpicture}[scale=0.75]
			\draw[draw,line width=4pt,orange,opacity=0.75,join=round] (0,0) to (-0.75,1) to (0,2) to (1,2) to (1.75,1) to (2.5,0);
			
			\node[circle,draw,fill=black,scale=0.5,label=below:{$u$}] (u) at (0,0) {};
			\node[circle,draw,fill=black,scale=0.5,label=below:{$z$}] (z) at (1,0) {};
			\node[circle,draw,fill=black,scale=0.5,label=left:{$v$}] (v) at (-0.75,1) {};
			\node[circle,draw,fill=black,scale=0.5,label=left:{$v'$}] (v2) at (-1.5,0) {};
			\node[circle,draw,fill=black,scale=0.5,label=right:{$y$}] (y) at (1.75,1) {};
			\node[circle,draw,fill=black,scale=0.5,label=right:{$y'$}] (y2) at (2.5,0) {};
			\node[circle,draw,fill=black,scale=0.5,label=above:{$w$}] (w) at (0,2) {};
			\node[circle,draw,fill=black,scale=0.5,label=above:{$x$}] (x) at (1,2) {};
			\draw (u) to (v) to (w) to (x) to (y) to (z) to (u);
			\draw (v) to (v2);
			\draw (v2) to[bend right=50] (z);
			\draw (y) to (y2);
		\end{tikzpicture}		\begin{tikzpicture}[scale=0.75]
			\draw[draw,line width=4pt,orange,opacity=0.75,join=round] (-0.75,1) to (0,2) to (1,2) to (0.75,1.5) to (0.5,1);
			\draw[draw,line width=4pt,orange,opacity=0.5,join=round] (0,0) to (-0.75,1);
			\draw[draw,line width=4pt,orange,opacity=0.5,join=round] (-1.5,0) to (-0.75,1);
			
			\node[circle,draw,fill=black,scale=0.5,label=below:{$u$}] (u) at (0,0) {};
			\node[circle,draw,fill=black,scale=0.5,label=below:{$z$}] (z) at (1,0) {};
			\node[circle,draw,fill=black,scale=0.5,label=left:{$v$}] (v) at (-0.75,1) {};
			\node[circle,draw,fill=black,scale=0.5,label=left:{$v'$}] (v2) at (-1.5,0) {};
			\node[circle,draw,fill=black,scale=0.5,label=right:{$y$}] (y) at (1.75,1) {};
			\node[circle,draw,fill=black,scale=0.5,label=right:{$y'$}] (y2) at (2.5,0) {};
			\node[circle,draw,fill=black,scale=0.5,label=above:{$w$}] (w) at (0,2) {};
			\node[circle,draw,fill=black,scale=0.5,label=above:{$x$}] (x) at (1,2) {};
			\node[circle,draw,fill=black,scale=0.5,label=left:{$x'$}] (x2) at (0.75,1.5) {};
			\node[circle,draw,fill=black,scale=0.5,label=left:{$x''$}] (x3) at (0.5,1) {};
			\draw (u) to (v) to (w) to (x) to (y) to (z) to (u);
			\draw (v) to (v2) to[bend right=50] (z);
			\draw (y) to (y2) .. controls (2.25,2.5) and (1.25, 3.25).. (w);
			\draw (x) to (x2) to (x3);
		\end{tikzpicture}		\caption{Illustrations of some of the paths regarded in the proof of \Cref{thm: min counterexample}.
		\label{fig:min-ce-paths}}
	\end{figure}
	
	Assume that $v'z \in E(G)$.
	The vertex $y$ has a neighbour $y' \notin E(C) \cup E(P)$ by~($\Delta$).
	Consider the path $uvwxyy'$ illustrated in the middle of \Cref{fig:min-ce-paths}.
	If this path is not induced, then $y'w \in E(G)$ by~($\Delta$) and ($\boxminus$). 
	We label this situation by ($\star$) so we can refer to it later.
	From~($\Delta$) and~($\boxminus$) we obtain that $x$ has a neighbour $x' \notin V(C) \cup \{v',y'\}$.
	We prove that both remaining neighbours of $x'$ are in $V(C) \cup \{v',y'\}$:
	since $G$ is cubic, $x'$ is not adjacent to $v,w,y$, or $z$.
	By~($K_{2,3}$) we obtain that $x'y'\notin E(G)$ and with (cs) we obtain $x'u, x'v' \notin E(G)$.
	Hence, there exists $x'' \in N_G(x')\setminus (V(C) \cup \{v',y'\})$.
	Consider the paths $uvwxx'x''$ and $v'vwxx'x''$ shown on the right of \Cref{fig:min-ce-paths}.
	At least one of the two paths is induced since the only possible chords are $ux''$ or $v'x''$, respectively.
	However, if both chords exist, then~($K_{2,3}$) is violated.
	Altogether, if $v'z \in E(G)$, then $G$ contains a path of the desired form.
	
	Now assume that $v'x \in E(G)$.
	There exists $y' \in N_G(y)\setminus (V(C)\cup\{v'\})$ since~$C$ is chordless and ($\Delta$) holds.
	Consider the path $uvwxyy'$.
	We may assume that this path is not induced, that is, $y'u \in E(G)$ (the remaining potential chord $wy'$ violates~($\boxminus$)).
	Observe that the vertices can be relabelled such that we are in the same situation as~$(\star)$ and, hence, there exists an induced $P_6$.
	
	Finally, let $v'y \in E(G)$.
	Since $C$ is chordless and ($\Delta$) holds, we obtain that there exists $w' \in N_G(w)\setminus (V(C)\cup \{v'\})$.
	The only possible neighbour of $w'$ in $V(C)\cup \{v'\}$ is $z$ by~($\Delta$), (th), and (cs).
	Since $\deg_G(w') =3$ there exists $w'' \in N_G(w')\setminus (V(C)\cup \{v'\})$.
	If $w''w'wxyv'$ is induced, then Claim 1 is true.
	Therefore, assume there is a chord of $w''w'wxyv'$ in $G$.
	By~($\Delta$) the only possible chords are $w''x$ and $w''v'$.
	If $w''x \in E(G)$, then there exists $\tilde{w} \in N_G(w')\setminus (V(C) \cup \{w'', v'\})$ by~($\boxminus$) and~(th), and, the path $\tilde{w}w'wxyv'$ is induced by~(cs).
	If otherwise $w''v' \in E(G)$, then the path $w''w'wxyz$ is induced by~($\Delta$), ($\boxminus$), (cs), and~(\petv). \hfill$\triangleleft$
	
	\medskip
	\noindent
	\textbf{Claim 2.} If $wx$ is the middle edge of a $P_6$ and $wx$ is not contained in a cycle of length~6, then~\eqref{itm: P6} is satisfied.
	
	\noindent
	\textit{Proof.} Let $P\coloneqq uvwxyz$ be a $P_6$ contained as a subgraph $G$.
	If $P$ is induced, then Claim 2 satisfied.
	Therefore, we may assume that $P$ has a chord. Since $G$ is triangle-free and $wx$ is not contained in a $C_6$, the following chords are possible: $ux$, $vy$, $wz$, $uy$, and $vz$.
	By symmetry, it suffices to consider the cases that $ux$, $vy$, or $vz$ is a chord of $P$.
	
	First assume that $ux \in E(G)$.
	By~($\Delta$) and~(th), we obtain that $v$ has a neighbour $v' \notin V(P)$.
	In the same way, we get that $y$ has a neighbour $y' \notin V(P) \cup \{v'\}$.
	The path $v'vwxyy'$ is induced since a chord would either violate~($\Delta$) or ($\boxminus$) or lead to a $C_6$ containing~$wx$.
	
	Now assume that $vy \in E(G)$.
	With~($\Delta$) and~($K_{2,3}$) we obtain that $w$ has a neighbour $w' \notin V(P)$ and $w'$ has a neighbour $w'' \notin V(P)$ by~($\Delta$), ($\boxminus$) and~(th).
	A chord of $w''w'wxyz$ does not exist due to the same three properties and the assumption that no $C_6$ contains~$wx$.
	
	Finally assume that $vz \in E(G)$.
	There is a neighbour $y'$ of $y$ with $y' \notin V(P)$ by~($\Delta$) and~(th).
	The path $uvwxyy'$ is induced since a chord of this path would either violate~($\Delta$) or~(th) or create a $C_6$ with $wx$ as an edge. \hfill$\triangleleft$
	
	\medskip
	\noindent
	\textbf{Claim 3.} There exists a $P_6$ whose middle edge is $wx$.
	
	\noindent
	\textit{Proof.} By ($\Delta$), we obtain $N_G(w) \cap N_G(x) = \emptyset$.
	Let $v \in N_G(w) \setminus \{x\}$.
	There exists a neighbour $u$ of $v$ with $u \notin N_G(w) \cup N_G(x)$ due to ($\Delta$) and ($K_{2,3}$).
	Let $y \in N_G(x)\setminus \{w\}$.
	By ($\Delta$), ($K_{2,3}$), and (th), we obtain that $y$ has a neighbour $z \notin N_G(w) \cup N_G(x) \cup \{u\}$.
	Altogether $uvwxyz$ is the desired $P_6$. \hfill$\triangleleft$
	
	\noindent
	This completes the proof.
\end{proof}

We close this section with a brief discussion of how we proved the theorem below.
\smallgraphs*
The cubic graphs of order at most~20 are known, see~\cite{Brinkmann1996, Meringer1999, BCGM13}.
We used the complete list of cubic graphs of order at most~20 as provided by~\cite{BCGM13}.
For each graph in this list we computed a tree that leads to a 3-decomposition for this graph, i.e., removing the edges of this tree from its host graph results in the disjoint union of a 2-regular graph, a 1-regular graph, and some isolated vertices.
To obtain such a tree for a given cubic graph we by first employed a heuristic.
If the heuristic did not lead to a desired tree,
then our algorithm tackled the problem by an enumeration approach.
The code as well as the computed trees can be found in our GitLab repository\textsuperscript{\ref{gitlab}}.

\section{Further research}
The most pressing question is whether \Cref{thm: reducible or HIST} can be exploited to prove the 3-decomposition conjecture in general.
A natural approach would be the following.
Suppose that there exists some counterexample to the 3-decomposition conjecture.
Reduce a \emph{quasi} 3-decomposition (into a spanning tree, a 2-regular graph and a disjoint union of paths) to some smaller graph which satisfies the conjecture.
Manipulate the reduced graph with its decomposition such that it can be extended again to a 3-decomposition of the original graph. The challenge here is that it is not clear that such a quasi 3-decomposition admits the reductions of \Cref{thm: reducible or HIST} in general.

A second obvious question is: can the 3-decomposition conjecture be proved for more graph classes by extending
the list of properties of a minimum counterexample? 
For example, we believe that there should be a reduction-based proof for the known result that planar graphs have a 3-decomposition~\cite{HoffmannOstenhof2018}.

Finally, it would be desirable to extend the computational results of \Cref{small-graphs-3-dec} to larger graphs.

\newpage
\bibliography{3Decomposition}
\bibliographystyle{alpha}

\newpage
\appendix
\section{Straight-forward extensions}

Here we go over the straight-forward local behaviours of the extensions to the $K_{2,3}$, the claw-square, the domino, the \petv, and the twin-house that we omitted in \Cref{sec: reducible configurations}.
We specify these in the same way we did for the triangle in \Cref{fig:3-dec-extension-triangle}, by determining the possible local behaviours at the smaller graph and showing how to extend them.
We also reduce the amount of local behaviours required for the square by employing switching arguments.

\paragraph{The $K_{2,3}$.}
Recall the extension in \Cref{fig:extension-reduction-k23}.
We determined the local behaviours of the single vertex in \Cref{fig:behaviour-node} and, because the $K_{2,3}$ is symmetric, we only need to check these three cases.
For each, \Cref{fig:3-dec-extension-k23} shows how to extend these to the $K_{2,3}$.
\begin{figure}[htb]
	\centering
	\begin{tikzpicture}
		\drawVertexDown{0}{0}{0.9}{A}{{tcol,ccol,ccol}}{black}{black!50}
		\drawVertexDown{4}{0}{0.9}{A}{{mcol,tcol,tcol}}{black}{black!50}
		\drawVertexDown{8}{0}{0.9}{A}{{tcol,tcol,tcol}}{black}{black!50}
		\transformationArrow{0.5}{0}{1}{1.5}
		\transformationArrow{4.5}{0}{1}{1.5}
		\transformationArrow{8.5}{0}{1}{1.5}
		\drawKTwoThree{2.25}{0}{0.75}{B}{{ccol,tcol,ccol,tcol,tcol,tcol,ccol,tcol,ccol}}{black}{black!50}
		\drawKTwoThree{6.25}{0}{0.75}{C}{{tcol,tcol,mcol,mcol,tcol,tcol,tcol,mcol,tcol}}{black}{black!50}
		\drawKTwoThree{10.25}{0}{0.75}{D}{{tcol,tcol,mcol,mcol,tcol,tcol,tcol,tcol,tcol}}{black}{black!50}
	\end{tikzpicture}
	\caption{Extending a 3-decomposition to a $K_{2,3}$.
	\label{fig:3-dec-extension-k23}}
\end{figure}
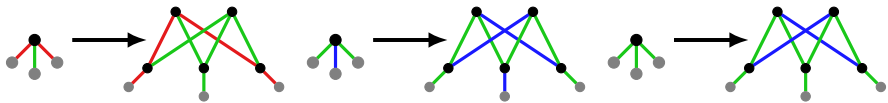

\paragraph{The \petv.}
For the extension from \Cref{fig:extension-reduction-petv} we need the same three behaviours as before, since the \petv{} is also symmetric.
The second one is covered in \Cref{3-compatible-petv}, and the remaining two are shown in \Cref{fig:3-dec-extension-petv}.
\begin{figure}[htb]
	\centering
	\begin{tikzpicture}
		\drawVertexDown{0}{0}{1}{A}{{tcol,ccol,ccol}}{black}{black!50}
		\drawVertexDown{6}{0}{1}{A}{{tcol,tcol,tcol}}{black}{black!50}
		\transformationArrow{1}{0}{1.25}{2}
		\transformationArrow{7}{0}{1.25}{2}
		\drawPetMinusV{3.75}{-0.25}{0.5}{A}{{ccol,ccol,ccol,tcol,mcol,tcol,tcol,tcol,tcol,tcol,tcol,tcol,ccol,ccol,tcol}}{black}{black!50}
		\drawPetMinusV{9.75}{-0.25}{0.5}{A}{{mcol,tcol,mcol,tcol,mcol,tcol,tcol,mcol,tcol,tcol,tcol,tcol,tcol,tcol,tcol}}{black}{black!50}
	\end{tikzpicture}
	\caption{Extending a 3-decomposition to a \petv.
	\label{fig:3-dec-extension-petv}}
\end{figure}
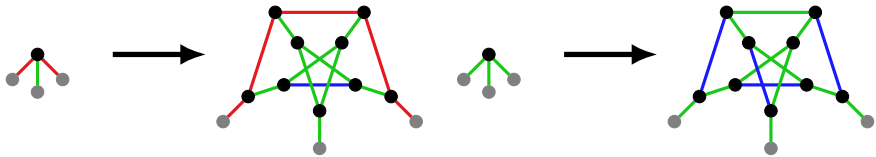

\paragraph{The domino, part I.}
Next we look at the extension in \Cref{fig:extension-reduction-edge-domino}.
For this, we first need to look at the possible behaviours of a 3-decomposition of $G$ at the edge $u_1u_2$ we are extending.
If the edge is in $C$, then one further edge at each end must also be in $C$ with the remaining two being part of $T$.
This yields, up to reflectional symmetry at the horizontal and vertical axis, the first two cases in \Cref{fig:behaviour-edge}.
Note that both the edge and the domino have these two symmetries, such that it suffices to regard one of the symmetrical cases.
\begin{figure}[htb]
	\centering
	\begin{tikzpicture}
		\drawEdge{0}{0}{0.75}{1}{{ccol,ccol,ccol,tcol,tcol}}{black}{black!50}		
		\drawEdge{1.5}{0}{0.75}{2}{{ccol,ccol,tcol,tcol,ccol}}{black}{black!50}
		\drawEdge{3}{0}{0.75}{3}{{mcol,tcol,tcol,tcol,tcol}}{black}{black!50}
		\drawEdge{4.5}{0}{0.75}{4}{{tcol,ccol,tcol,ccol,tcol}}{black}{black!50}
		\drawEdge{6}{0}{0.75}{5}{{tcol,ccol,mcol,ccol,tcol}}{black}{black!50}
		
		\drawEdge{0.75}{-1}{0.75}{10}{{tcol,tcol,tcol,tcol,tcol}}{black}{black!50}	
		\drawEdge{2.25}{-1}{0.75}{11}{{tcol,mcol,tcol,tcol,tcol}}{black}{black!50}
		\drawEdge{3.75}{-1}{0.75}{12}{{tcol,mcol,mcol,tcol,tcol}}{black}{black!50}
		\drawEdge{5.25}{-1}{0.75}{13}{{tcol,mcol,tcol,tcol,mcol}}{black}{black!50}
	\end{tikzpicture}
	\caption{Behaviour of a 3-decomposition at an edge.
	\label{fig:behaviour-edge}}
\end{figure}
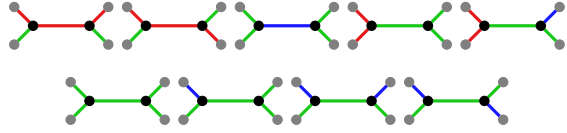

Next, assume that $u_1u_2$ is in $M$, then all remaining edges are in $T$, giving us the third case.
For all other cases, we have $u_1u_2\in T$ and we need to distinguish the behaviour of the remaining edges.
Should one of them be in $C$, say one at $u_1$, then both at this vertex are and at least one at $u_2$ is in $T$.
This gives us the fourth and fifth case.

We may now assume no edges are in $C$ and only need to consider the amount of $M$-edges left.
If there are none, all edges must be in $T$.
A single $M$-edge just yields one case up to symmetries and two give us two more.
More than two are not possible.

In total, there are nine types of behaviours to check, all of which give rise to a straight-forward extension to the domino, as shown in \Cref{fig:3-dec-extension-domino-edge}.
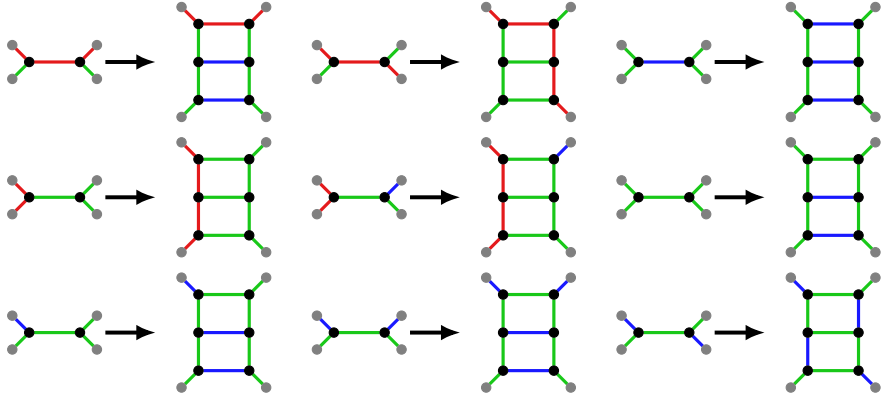
\begin{figure}[htb]
	\centering
	\begin{tikzpicture}[scale=.9]
		\drawEdge{0}{0}{0.75}{1}{{ccol,ccol,ccol,tcol,tcol}}{black}{black!50}
		\transformationArrow{0.75}{0}{0.75}{1.5}
		\drawDomino{2.5}{0}{0.75}{A}{{ccol,tcol,tcol,mcol,tcol,tcol,mcol,ccol,ccol,tcol,tcol}}{black}{black!50}
		\drawEdge{4.5}{0}{0.75}{2}{{ccol,ccol,tcol,tcol,ccol}}{black}{black!50}
		\transformationArrow{5.25}{0}{0.75}{1.5}
		\drawDomino{7}{0}{0.75}{A}{{ccol,tcol,ccol,tcol,tcol,ccol,tcol,ccol,tcol,tcol,ccol}}{black}{black!50}
		\drawEdge{9}{0}{0.75}{3}{{mcol,tcol,tcol,tcol,tcol}}{black}{black!50}
		\transformationArrow{9.75}{0}{0.75}{1.5}
		\drawDomino{11.5}{0}{0.75}{A}{{mcol,tcol,tcol,mcol,tcol,tcol,mcol,tcol,tcol,tcol,tcol}}{black}{black!50}
		
		\drawEdge{0}{-2}{0.75}{4}{{tcol,ccol,tcol,ccol,tcol}}{black}{black!50}
		\transformationArrow{0.75}{-2}{0.75}{1.5}
		\drawDomino{2.5}{-2}{0.75}{A}{{tcol,ccol,tcol,tcol,ccol,tcol,tcol,ccol,tcol,ccol,tcol}}{black}{black!50}
		\drawEdge{4.5}{-2}{0.75}{5}{{tcol,ccol,mcol,ccol,tcol}}{black}{black!50}
		\transformationArrow{5.25}{-2}{0.75}{1.5}
		\drawDomino{7}{-2}{0.75}{A}{{tcol,ccol,tcol,tcol,ccol,tcol,tcol,ccol,mcol,ccol,tcol}}{black}{black!50}
		\drawEdge{9}{-2}{0.75}{6}{{tcol,tcol,tcol,tcol,tcol}}{black}{black!50}
		\transformationArrow{9.75}{-2}{0.75}{1.5}
		\drawDomino{11.5}{-2}{0.75}{A}{{tcol,tcol,tcol,mcol,tcol,tcol,mcol,tcol,tcol,tcol,tcol}}{black}{black!50}
		
		\drawEdge{0}{-4}{0.75}{7}{{tcol,mcol,tcol,tcol,tcol}}{black}{black!50}
		\transformationArrow{0.75}{-4}{0.75}{1.5}
		\drawDomino{2.5}{-4}{0.75}{A}{{tcol,tcol,tcol,mcol,tcol,tcol,mcol,mcol,tcol,tcol,tcol}}{black}{black!50}
		\drawEdge{4.5}{-4}{0.75}{8}{{tcol,mcol,mcol,tcol,tcol}}{black}{black!50}
		\transformationArrow{5.25}{-4}{0.75}{1.5}
		\drawDomino{7}{-4}{0.75}{A}{{tcol,tcol,tcol,mcol,tcol,tcol,mcol,mcol,mcol,tcol,tcol}}{black}{black!50}
		\drawEdge{9}{-4}{0.75}{9}{{tcol,mcol,tcol,tcol,mcol}}{black}{black!50}
		\transformationArrow{9.75}{-4}{0.75}{1.5}
		\drawDomino{11.5}{-4}{0.75}{A}{{tcol,tcol,mcol,tcol,mcol,tcol,tcol,mcol,tcol,tcol,mcol}}{black}{black!50}
	\end{tikzpicture}
	\caption{Extending a 3-decomposition from the edge to the domino.
	\label{fig:3-dec-extension-domino-edge}}
\end{figure}

\paragraph{Local behaviours of the square.}
All remaining extensions start from the square, which is why we first check which local behaviours 3-de\-com\-po\-si\-tions can exhibit.
Afterwards, we use switching arguments to eliminate behaviours that need not be considered.
Let $u_1u_2u_4u_3u_1$ be a square (as seen in \Cref{fig:partly-compatible-square-domino}) in a graph $G$ with 3-de\-com\-po\-sition $(T,C,M)$.
The possibilities obtained initially are shown in \Cref{fig:behaviour-square}, disregarding rotational symmetry.
\begin{figure}[htb]
	\centering
	\begin{tikzpicture}
		\drawSquare{0}{0}{0.75}{1}{{ccol,ccol,ccol,ccol,tcol,tcol,tcol,tcol}}{black}{black!50}	
		\drawSquare{1.5}{0}{0.75}{2}{{ccol,tcol,ccol,tcol,ccol,tcol,tcol,ccol}}{black}{black!50}
		\drawSquare{3}{0}{0.75}{3}{{ccol,tcol,tcol,mcol,ccol,ccol,tcol,tcol}}{black}{black!50}
		\drawSquare{4.5}{0}{0.75}{4}{{ccol,tcol,tcol,tcol,ccol,ccol,mcol,tcol}}{black}{black!50}
		\drawSquare{6}{0}{0.75}{5}{{ccol,tcol,tcol,tcol,ccol,ccol,tcol,mcol}}{black}{black!50}
		\drawSquare{7.5}{0}{0.75}{6}{{ccol,tcol,tcol,tcol,ccol,ccol,tcol,tcol}}{black}{black!50}
		
		\drawSquare{0.75}{-1.5}{0.75}{10}{{mcol,tcol,tcol,mcol,tcol,tcol,tcol,tcol}}{black}{black!50}		
		\drawSquare{2.25}{-1.5}{0.75}{11}{{mcol,tcol,tcol,tcol,tcol,tcol,mcol,mcol}}{black}{black!50}	
		\drawSquare{3.75}{-1.5}{0.75}{12}{{mcol,tcol,tcol,tcol,tcol,tcol,mcol,tcol}}{black}{black!50}
		\drawSquare{5.25}{-1.5}{0.75}{13}{{mcol,tcol,tcol,tcol,tcol,tcol,tcol,mcol}}{black}{black!50}
		\drawSquare{6.75}{-1.5}{0.75}{14}{{mcol,tcol,tcol,tcol,tcol,tcol,tcol,tcol}}{black}{black!50}
	\end{tikzpicture}
	\caption{Behaviours of a 3-decomposition at a square (without rotations).
	\label{fig:behaviour-square}}
\end{figure}
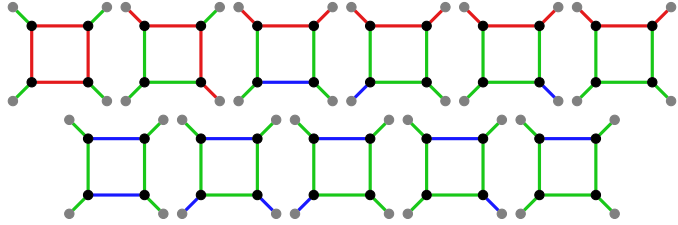

We first distinguish cases based on the amount of edges of $C$ in the square.
If there are four such edges, then the remaining edges with an end in the square must be part of $T$.
It is not possible for exactly three edges to be in $C$ as this isolates an edge.
In the case of two edges, they form a path between two non-adjacent vertices and all other edges are in $T$, yielding just one case with four rotations.
For a single edge, we need its ends to be incident to another edge in $C$ that leaves the square and the missing edge is part of $T$.
The final edge of the square can then be in $M$, resulting in two more $T$-edges, or in $T$, resulting in at most one $M$-edge.
This gives the next four cases, each of which has four rotations.

This just leaves the case without any $C$-edges in the square.
We now differentiate by the amount of $M$-edges, where zero, three, and four are impossible.
Two $M$-edges must be opposite and all other edges are in $T$, resulting in one case with two rotations.
If there is just a single $M$-edge, it comes with five $T$-edges (three in the square and two more at its ends).
The remaining two can be any combination of $M$- and $T$-edges, yielding a total of four more cases with four rotations each.

The total amount of cases is now 39, so significantly more than we want to check for each extension, which is why we show how these can be reduced to the 18 shown in \Cref{fig:behaviour-square-reduced}.
This is done by locally changing the decomposition of the eliminated cases into one that is still present.
\begin{figure}[htb]
	\centering
	\begin{tikzpicture}
		\drawSquare{0}{0}{0.75}{1}{{ccol,ccol,ccol,ccol,tcol,tcol,tcol,tcol}}{black}{black!50}	
		\drawSquare{1.5}{0}{0.75}{2}{{ccol,tcol,ccol,tcol,ccol,tcol,tcol,ccol}}{black}{black!50}
		\drawSquare{3}{0}{0.75}{3}{{tcol,tcol,ccol,ccol,tcol,ccol,ccol,tcol}}{black}{black!50}
		\drawSquare{4.5}{0}{0.75}{4}{{ccol,tcol,tcol,mcol,ccol,ccol,tcol,tcol}}{black}{black!50}
		\drawSquare{6}{0}{0.75}{5}{{tcol,mcol,ccol,tcol,tcol,ccol,tcol,ccol}}{black}{black!50}
		\drawSquare{7.5}{0}{0.75}{6}{{mcol,tcol,tcol,ccol,tcol,tcol,ccol,ccol}}{black}{black!50}
		
		\drawSquare{0}{-1.5}{0.75}{7}{{tcol,ccol,mcol,tcol,ccol,tcol,ccol,tcol}}{black}{black!50}
		\drawSquare{1.5}{-1.5}{0.75}{8}{{ccol,tcol,tcol,tcol,ccol,ccol,tcol,tcol}}{black}{black!50}
		\drawSquare{3}{-1.5}{0.75}{9}{{tcol,tcol,ccol,tcol,tcol,ccol,tcol,ccol}}{black}{black!50}
		\drawSquare{4.5}{-1.5}{0.75}{10}{{tcol,tcol,tcol,ccol,tcol,tcol,ccol,ccol}}{black}{black!50}		
		\drawSquare{6}{-1.5}{0.75}{11}{{tcol,ccol,tcol,tcol,ccol,tcol,ccol,tcol}}{black}{black!50}	
		\drawSquare{7.5}{-1.5}{0.75}{12}{{mcol,tcol,tcol,mcol,tcol,tcol,tcol,tcol}}{black}{black!50}
		
		\drawSquare{0}{-3}{0.75}{13}{{tcol,mcol,mcol,tcol,tcol,tcol,tcol,tcol}}{black}{black!50}
		\drawSquare{1.5}{-3}{0.75}{14}{{mcol,tcol,tcol,tcol,tcol,tcol,mcol,mcol}}{black}{black!50}
		\drawSquare{3}{-3}{0.75}{15}{{tcol,tcol,mcol,tcol,mcol,tcol,mcol,tcol}}{black}{black!50}
		\drawSquare{4.5}{-3}{0.75}{16}{{tcol,tcol,tcol,mcol,mcol,mcol,tcol,tcol}}{black}{black!50}
		\drawSquare{6}{-3}{0.75}{17}{{tcol,mcol,tcol,tcol,tcol,mcol,tcol,mcol}}{black}{black!50}
		\drawSquare{7.5}{-3}{0.75}{18}{{mcol,tcol,tcol,tcol,tcol,tcol,tcol,tcol}}{black}{black!50}
	\end{tikzpicture}
	\caption{The local behaviours of the square that suffice to be checked.
	\label{fig:behaviour-square-reduced}}
\end{figure}
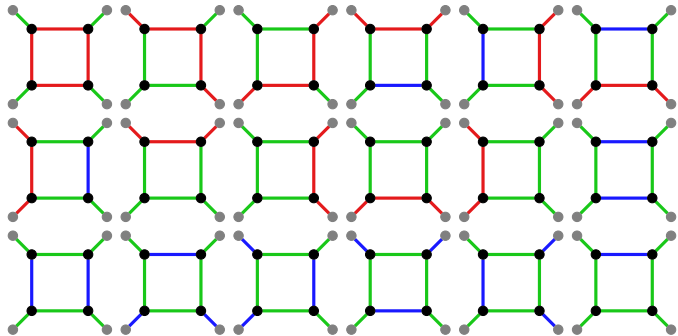

The switches used to achieve this are shown in \Cref{fig:behaviour-square-transformations} and we describe each of them here in turn.
For the first one we note that the length-2 paths through the square can be used interchangeably, removing the two edges of the tree creates two isolated vertices that are reconnected by using those previously in a cycle, saving us two cases.
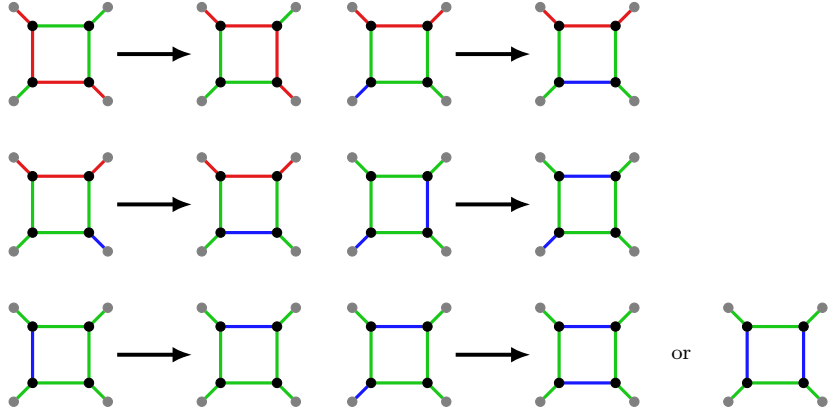
\begin{figure}[htb]
	\centering
	\begin{tikzpicture}
		\drawSquare{0}{0}{0.75}{1}{{tcol,ccol,tcol,ccol,ccol,tcol,tcol,ccol}}{black}{black!50}
		\transformationArrow{0.75}{0}{1}{1.5}
		\drawSquare{2.5}{0}{0.75}{1}{{ccol,tcol,ccol,tcol,ccol,tcol,tcol,ccol}}{black}{black!50}
		
		\drawSquare{4.5}{0}{0.75}{1}{{ccol,tcol,tcol,tcol,ccol,ccol,mcol,tcol}}{black}{black!50}
		\transformationArrow{5.25}{0}{1}{1.5}
		\drawSquare{7}{0}{0.75}{1}{{ccol,tcol,tcol,mcol,ccol,ccol,tcol,tcol}}{black}{black!50}
		
		\drawSquare{0}{-2}{0.75}{1}{{ccol,tcol,tcol,tcol,ccol,ccol,tcol,mcol}}{black}{black!50}	
		\transformationArrow{0.75}{-2}{1}{1.5}
		\drawSquare{2.5}{-2}{0.75}{1}{{ccol,tcol,tcol,mcol,ccol,ccol,tcol,tcol}}{black}{black!50}
		
		\drawSquare{4.5}{-2}{0.75}{1}{{tcol,tcol,mcol,tcol,tcol,tcol,mcol,tcol}}{black}{black!50}	
		\transformationArrow{5.25}{-2}{1}{1.5}
		\drawSquare{7}{-2}{0.75}{1}{{mcol,tcol,tcol,tcol,tcol,tcol,mcol,tcol}}{black}{black!50}
		
		\drawSquare{0}{-4}{0.75}{1}{{tcol,mcol,tcol,tcol,tcol,tcol,tcol,tcol}}{black}{black!50}
		\transformationArrow{0.75}{-4}{1}{1.5}
		\drawSquare{2.5}{-4}{0.75}{1}{{mcol,tcol,tcol,tcol,tcol,tcol,tcol,tcol}}{black}{black!50}
		
		\drawSquare{4.5}{-4}{0.75}{1}{{mcol,tcol,tcol,tcol,tcol,tcol,mcol,tcol}}{black}{black!50}	
		\transformationArrow{5.25}{-4}{1}{1.5}
		\drawSquare{7}{-4}{0.75}{1}{{mcol,tcol,tcol,mcol,tcol,tcol,tcol,tcol}}{black}{black!50}
		\node at (8.25,-4) {or};
		\drawSquare{9.5}{-4}{0.75}{1}{{tcol,mcol,mcol,tcol,tcol,tcol,tcol,tcol}}{black}{black!50}
	\end{tikzpicture}
	\caption{The transformations used to eliminate cases.
	\label{fig:behaviour-square-transformations}}
\end{figure}

Next, regard the case with a single $C$-edge $u_1u_2$ where the edge at $u_3$ that is not in the square is in the matching.
By removing the $T$-edge $u_3u_4$ we get two components, one of which is the edge $u_1u_3$.
By moving the $M$-edge at $u_3$ to the tree, these components are unified, giving us a new tree.
The analogous transformation works if the $M$-edge is at $u_4$ and for all rotations, saving us eight cases in total.

This completes the reduction of cases with $C$-edges.
Let us now regard the case in which the square contains a single $M$-edge, say $u_2u_4$ and there is a second one at $u_3$.
Adding the edge $u_2u_4$ to the tree yields the square as the unique cycle and we can remove $u_1u_2$ to create a new spanning tree and 3-decomposition.
As this works for the four rotations, we save another four cases.
We can also remove three of the four rotations of the square with one $M$-edge and only $T$-edges elsewhere with the same argument.

Finally, the graph with $M$-edges $u_1u_2$ and $u_3v_3$ can be reduced to one of the two cases with two $M$-edges, reducing the total amount of cases by four yet again.
To see this, note that adding the $M$-edge at $u_3$ to the tree yields a cycle.
If it uses the edge $u_3u_4$, we put it into the matching instead and get a new 3-decomposition.
Otherwise, the edge $u_3u_1$ is used and the cycle consists of this edge and a path from $u_1$ to $u_3$ that does not meet the square. 
But we have already seen that we may exchange the edges $u_1u_2$ and $u_2u_4$ and in the resulting graph the cycle obtained when adding the $M$-edge at $u_3$ to the tree remains unchanged.
This lets us swap it with $u_1u_3$ to obtain a new 3-decomposition.

We have now completed the switches.
Note that there are symmetric cases amongst the remaining 18, but the symmetries we may use also depend on which are present in the extension.
We cannot use rotational symmetry, for example, as neither the domino, the twin-house, nor the claw-square have this symmetry.

\paragraph{The domino, part II.}
We can now complete the missing extension for the domino, seen on the right of \Cref{fig:partly-compatible-square-domino}.
Of the behaviours from \Cref{fig:behaviour-square-reduced}, we can omit the following due to the domino's symmetries: 3, 6, 7, 10, 11, 16, and 17, where the number indicates their position.
We have covered behaviour~15 in \Cref{3-compatible-square-domino}, leaving the ten shown in \Cref{fig:3-dec-extension-domino-square}.
\begin{figure}[htb]
	\centering
	\begin{tikzpicture}[scale=0.9]
		\drawSquare{0}{0}{0.75}{1}{{ccol,ccol,ccol,ccol,tcol,tcol,tcol,tcol}}{black}{black!50}
		\transformationArrow{0.75}{0}{0.75}{1.5}
		\drawDomino{2.5}{0}{0.75}{A}{{mcol,tcol,tcol,ccol,ccol,ccol,ccol,tcol,tcol,tcol,tcol}}{black}{black!50}
		\drawSquare{4.5}{0}{0.75}{2}{{ccol,tcol,ccol,tcol,ccol,tcol,tcol,ccol}}{black}{black!50}
		\transformationArrow{5.25}{0}{0.75}{1.5}	
		\drawDomino{7}{0}{0.75}{A}{{ccol,tcol,ccol,tcol,tcol,ccol,tcol,ccol,tcol,tcol,ccol}}{black}{black!50}
		\drawSquare{9}{0}{0.75}{3}{{ccol,tcol,tcol,mcol,ccol,ccol,tcol,tcol}}{black}{black!50}
		\transformationArrow{9.75}{0}{0.75}{1.5}
		\drawDomino{11.5}{0}{0.75}{A}{{ccol,tcol,tcol,mcol,tcol,tcol,mcol,ccol,ccol,tcol,tcol}}{black}{black!50}
		
		\drawSquare{0}{-2}{0.75}{4}{{tcol,mcol,ccol,tcol,tcol,ccol,tcol,ccol}}{black}{black!50}	
		\transformationArrow{0.75}{-2}{0.75}{1.5}
		\drawDomino{2.5}{-2}{0.75}{A}{{tcol,mcol,ccol,tcol,tcol,ccol,tcol,tcol,ccol,tcol,ccol}}{black}{black!50}
		\drawSquare{4.5}{-2}{0.75}{5}{{ccol,tcol,tcol,tcol,ccol,ccol,tcol,tcol}}{black}{black!50}
		\transformationArrow{5.25}{-2}{0.75}{1.5}
		\drawDomino{7}{-2}{0.75}{A}{{ccol,tcol,tcol,mcol,tcol,tcol,tcol,ccol,ccol,tcol,tcol}}{black}{black!50}
		\drawSquare{9}{-2}{0.75}{6}{{tcol,tcol,ccol,tcol,tcol,ccol,tcol,ccol}}{black}{black!50}
		\transformationArrow{9.75}{-2}{0.75}{1.5}
		\drawDomino{11.5}{-2}{0.75}{A}{{tcol,tcol,ccol,tcol,tcol,ccol,tcol,tcol,ccol,tcol,ccol}}{black}{black!50}
		
		\drawSquare{0}{-4}{0.75}{7}{{mcol,tcol,tcol,mcol,tcol,tcol,tcol,tcol}}{black}{black!50}
		\transformationArrow{0.75}{-4}{0.75}{1.5}
		\drawDomino{2.5}{-4}{0.75}{A}{{mcol,tcol,tcol,mcol,tcol,tcol,mcol,tcol,tcol,tcol,tcol}}{black}{black!50}
		\drawSquare{4.5}{-4}{0.75}{8}{{tcol,mcol,mcol,tcol,tcol,tcol,tcol,tcol}}{black}{black!50}
		\transformationArrow{5.25}{-4}{0.75}{1.5}
		\drawDomino{7}{-4}{0.75}{A}{{tcol,mcol,mcol,tcol,tcol,tcol,mcol,tcol,tcol,tcol,tcol}}{black}{black!50}
		\drawSquare{9}{-4}{0.75}{9}{{mcol,tcol,tcol,tcol,tcol,tcol,mcol,mcol}}{black}{black!50}
		\transformationArrow{9.75}{-4}{0.75}{1.5}
		\drawDomino{11.5}{-4}{0.75}{A}{{mcol,tcol,tcol,mcol,tcol,tcol,tcol,tcol,tcol,mcol,mcol}}{black}{black!50}
		
		\drawSquare{4.5}{-6}{0.75}{11}{{mcol,tcol,tcol,tcol,tcol,tcol,tcol,tcol}}{black}{black!50}
		\transformationArrow{5.25}{-6}{0.75}{1.5}	
		\drawDomino{7}{-6}{0.75}{A}{{mcol,tcol,tcol,mcol,tcol,tcol,tcol,tcol,tcol,tcol,tcol}}{black}{black!50}
	\end{tikzpicture}
	\caption{Extending a 3-decomposition from the square to the domino.
	\label{fig:3-dec-extension-domino-square}}
\end{figure}
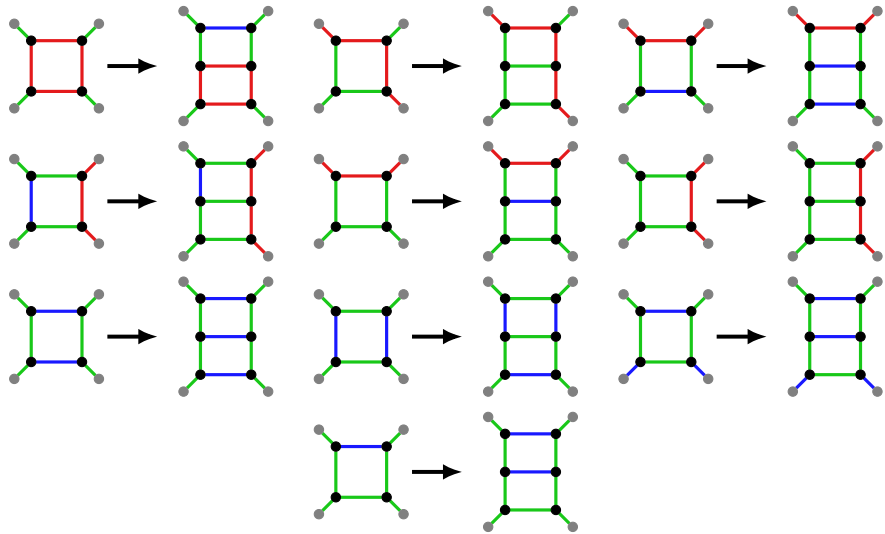

\paragraph{The twin-house.}
The next extension of the square that we regard is the one to the twin-house, seen in \Cref{fig:extension-reduction-twin-house}.
Of the behaviours from \Cref{fig:behaviour-square-reduced}, we can omit the ones at position 7, 11, and 17, by symmetry.
We have covered behaviours~13 and~14 in \Cref{3-compatible-square-domino}, leaving the 13 shown in \Cref{fig:3-dec-extension-twin-house}.
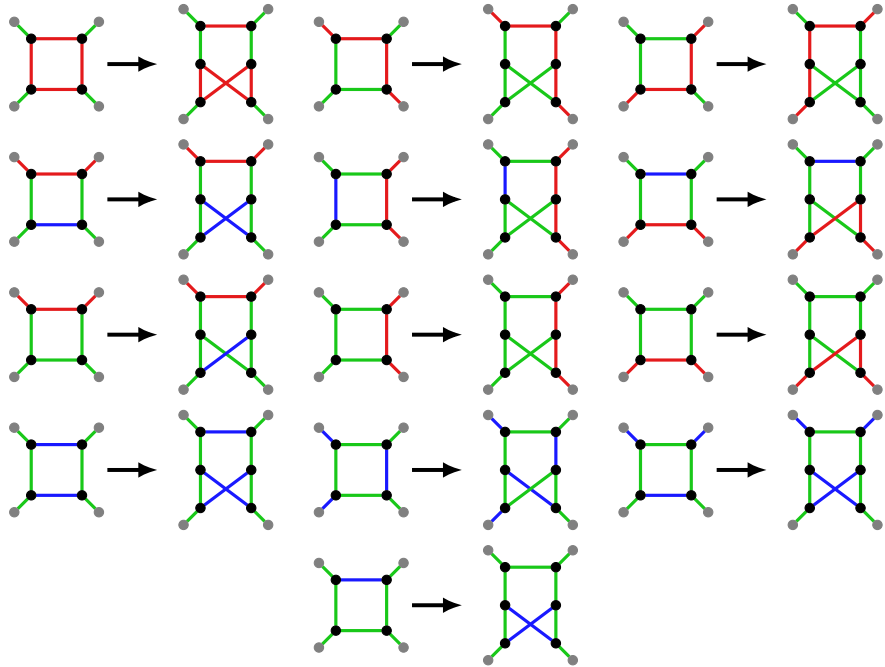
\begin{figure}[htb]
	\centering
	\begin{tikzpicture}[scale=.9]
		\drawSquare{0}{0}{0.75}{1}{{ccol,ccol,ccol,ccol,tcol,tcol,tcol,tcol}}{black}{black!50}	
		\transformationArrow{0.75}{0}{0.75}{1.5}
		\drawGFour{2.5}{0}{0.75}{A}{{ccol,tcol,tcol,ccol,ccol,ccol,ccol,tcol,tcol,tcol,tcol}}{black}{black!50}
		\drawSquare{4.5}{0}{0.75}{2}{{ccol,tcol,ccol,tcol,ccol,tcol,tcol,ccol}}{black}{black!50}
		\transformationArrow{5.25}{0}{0.75}{1.5}	
		\drawGFour{7}{0}{0.75}{A}{{ccol,tcol,ccol,tcol,ccol,tcol,tcol,ccol,tcol,tcol,ccol}}{black}{black!50}
		\drawSquare{9}{0}{0.75}{3}{{tcol,tcol,ccol,ccol,tcol,ccol,ccol,tcol}}{black}{black!50}
		\transformationArrow{9.75}{0}{0.75}{1.5}
		\drawGFour{11.5}{0}{0.75}{A}{{ccol,ccol,tcol,ccol,tcol,tcol,tcol,tcol,ccol,ccol,tcol}}{black}{black!50}
		
		\drawSquare{0}{-2}{0.75}{4}{{ccol,tcol,tcol,mcol,ccol,ccol,tcol,tcol}}{black}{black!50}	
		\transformationArrow{0.75}{-2}{0.75}{1.5}
		\drawGFour{2.5}{-2}{0.75}{A}{{ccol,tcol,tcol,tcol,tcol,mcol,mcol,ccol,ccol,tcol,tcol}}{black}{black!50}
		\drawSquare{4.5}{-2}{0.75}{5}{{tcol,mcol,ccol,tcol,tcol,ccol,tcol,ccol}}{black}{black!50}
		\transformationArrow{5.25}{-2}{0.75}{1.5}
		\drawGFour{7}{-2}{0.75}{A}{{tcol,mcol,ccol,tcol,ccol,tcol,tcol,tcol,ccol,tcol,ccol}}{black}{black!50}
		\drawSquare{9}{-2}{0.75}{6}{{mcol,tcol,tcol,ccol,tcol,tcol,ccol,ccol}}{black}{black!50}
		\transformationArrow{9.75}{-2}{0.75}{1.5}
		\drawGFour{11.5}{-2}{0.75}{A}{{mcol,tcol,tcol,tcol,ccol,tcol,ccol,tcol,tcol,ccol,ccol}}{black}{black!50}
		
		\drawSquare{0}{-4}{0.75}{7}{{ccol,tcol,tcol,tcol,ccol,ccol,tcol,tcol}}{black}{black!50}	
		\transformationArrow{0.75}{-4}{0.75}{1.5}
		\drawGFour{2.5}{-4}{0.75}{A}{{ccol,tcol,tcol,tcol,tcol,tcol,mcol,ccol,ccol,tcol,tcol}}{black}{black!50}
		\drawSquare{4.5}{-4}{0.75}{8}{{tcol,tcol,ccol,tcol,tcol,ccol,tcol,ccol}}{black}{black!50}
		\transformationArrow{5.25}{-4}{0.75}{1.5}
		\drawGFour{7}{-4}{0.75}{A}{{tcol,tcol,ccol,tcol,ccol,tcol,tcol,tcol,ccol,tcol,ccol}}{black}{black!50}
		\drawSquare{9}{-4}{0.75}{9}{{tcol,tcol,tcol,ccol,tcol,tcol,ccol,ccol}}{black}{black!50}
		\transformationArrow{9.75}{-4}{0.75}{1.5}
		\drawGFour{11.5}{-4}{0.75}{A}{{tcol,tcol,tcol,tcol,ccol,tcol,ccol,tcol,tcol,ccol,ccol}}{black}{black!50}
		
		\drawSquare{0}{-6}{0.75}{10}{{mcol,tcol,tcol,mcol,tcol,tcol,tcol,tcol}}{black}{black!50}	
		\transformationArrow{0.75}{-6}{0.75}{1.5}	
		\drawGFour{2.5}{-6}{0.75}{A}{{mcol,tcol,tcol,tcol,tcol,mcol,mcol,tcol,tcol,tcol,tcol}}{black}{black!50}
		\drawSquare{4.5}{-6}{0.75}{11}{{tcol,tcol,mcol,tcol,mcol,tcol,mcol,tcol}}{black}{black!50}
		\transformationArrow{5.25}{-6}{0.75}{1.5}	
		\drawGFour{7}{-6}{0.75}{A}{{tcol,tcol,mcol,tcol,tcol,mcol,tcol,mcol,tcol,mcol,tcol}}{black}{black!50}
		\drawSquare{9}{-6}{0.75}{12}{{tcol,tcol,tcol,mcol,mcol,mcol,tcol,tcol}}{black}{black!50}
		\transformationArrow{9.75}{-6}{0.75}{1.5}
		\drawGFour{11.5}{-6}{0.75}{A}{{tcol,tcol,tcol,tcol,tcol,mcol,mcol,mcol,mcol,tcol,tcol}}{black}{black!50}
		
		\drawSquare{4.5}{-8}{0.75}{14}{{mcol,tcol,tcol,tcol,tcol,tcol,tcol,tcol}}{black}{black!50}
		\transformationArrow{5.25}{-8}{0.75}{1.5}
		\drawGFour{7}{-8}{0.75}{A}{{tcol,tcol,tcol,tcol,tcol,mcol,mcol,tcol,tcol,tcol,tcol}}{black}{black!50}
	\end{tikzpicture}
	\caption{Extending a 3-decomposition to a twin-house.
	\label{fig:3-dec-extension-twin-house}}
\end{figure}

\paragraph{The claw-square.}
We finish by covering the last extension, the one from the square to the claw-square, seen in \Cref{fig:extension-reduction-claw-square}.
Here we can eliminate the behaviours at position 6, 7, 10, 11, 13, 15 and 16 in \Cref{fig:behaviour-square-reduced} by symmetry.
We are left with the 11 behaviours shown in \Cref{fig:3-dec-extension-claw-square}.
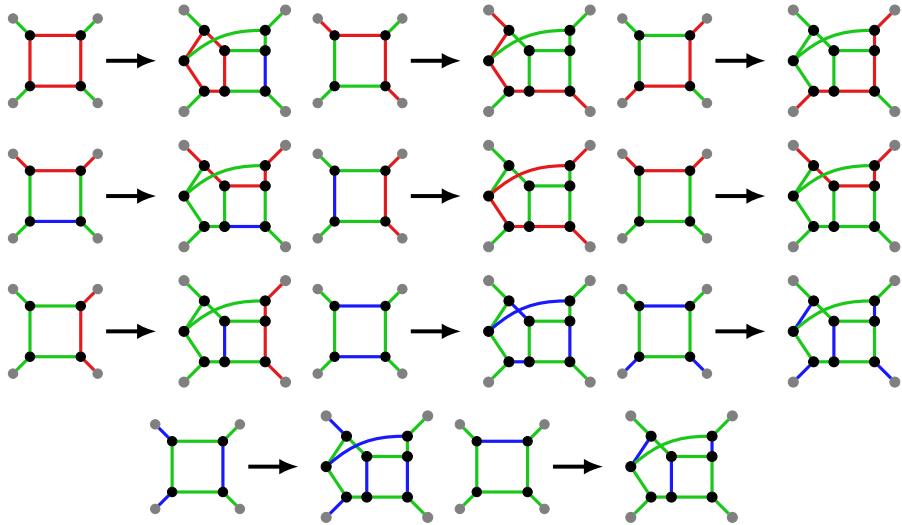
\begin{figure}[htb]
	\centering
	\begin{tikzpicture}[scale=0.9]
		\drawSquare{0}{0}{0.75}{1}{{ccol,ccol,ccol,ccol,tcol,tcol,tcol,tcol}}{black}{black!50}	
		\transformationArrow{0.75}{0}{0.75}{1.5}
		\drawGFive{2.5}{0.15}{0.6}{A}{{tcol,ccol,mcol,tcol,ccol,tcol,ccol,ccol,tcol,ccol,tcol,tcol,tcol,tcol}}{black}{black!50}
		\drawSquare{4.5}{0}{0.75}{2}{{ccol,tcol,ccol,tcol,ccol,tcol,tcol,ccol}}{black}{black!50}
		\transformationArrow{5.25}{0}{0.75}{1.5}	
		\drawGFive{7}{0.15}{0.6}{A}{{tcol,tcol,tcol,ccol,tcol,tcol,ccol,ccol,tcol,ccol,ccol,tcol,tcol,ccol}}{black}{black!50}
		\drawSquare{9}{0}{0.75}{3}{{tcol,tcol,ccol,ccol,tcol,ccol,ccol,tcol}}{black}{black!50}
		\transformationArrow{9.75}{0}{0.75}{1.5}
		\drawGFive{11.5}{0.15}{0.6}{A}{{tcol,tcol,ccol,ccol,tcol,ccol,ccol,tcol,tcol,tcol,tcol,ccol,ccol,tcol}}{black}{black!50}
		
		\drawSquare{0}{-2}{0.75}{4}{{ccol,tcol,tcol,mcol,ccol,ccol,tcol,tcol}}{black}{black!50}	
		\transformationArrow{0.75}{-2}{0.75}{1.5}
		\drawGFive{2.5}{-1.85}{0.6}{A}{{ccol,tcol,tcol,mcol,ccol,ccol,tcol,tcol,tcol,tcol,ccol,ccol,tcol,tcol}}{black}{black!50}
		\drawSquare{4.5}{-2}{0.75}{5}{{tcol,mcol,ccol,tcol,tcol,ccol,tcol,ccol}}{black}{black!50}
		\transformationArrow{5.25}{-2}{0.75}{1.5}
		\drawGFive{7}{-1.85}{0.6}{A}{{tcol,tcol,tcol,ccol,tcol,tcol,ccol,tcol,ccol,ccol,tcol,ccol,tcol,ccol}}{black}{black!50}
		\drawSquare{9}{-2}{0.75}{6}{{ccol,tcol,tcol,tcol,ccol,ccol,tcol,tcol}}{black}{black!50}
		\transformationArrow{9.75}{-2}{0.75}{1.5}
		\drawGFive{11.5}{-1.85}{0.6}{A}{{ccol,tcol,tcol,tcol,ccol,ccol,tcol,tcol,tcol,tcol,ccol,ccol,tcol,tcol}}{black}{black!50}
		\drawSquare{0}{-4}{0.75}{7}{{tcol,tcol,ccol,tcol,tcol,ccol,tcol,ccol}}{black}{black!50}
		\transformationArrow{0.75}{-4}{0.75}{1.5}
		\drawGFive{2.5}{-3.85}{0.6}{A}{{tcol,mcol,ccol,tcol,tcol,ccol,tcol,tcol,tcol,tcol,tcol,ccol,tcol,ccol}}{black}{black!50}
		\drawSquare{4.5}{-4}{0.75}{8}{{mcol,tcol,tcol,mcol,tcol,tcol,tcol,tcol}}{black}{black!50}
		\transformationArrow{5.25}{-4}{0.75}{1.5}
		\drawGFive{7}{-3.85}{0.6}{A}{{tcol,tcol,mcol,tcol,mcol,tcol,mcol,tcol,mcol,tcol,tcol,tcol,tcol,tcol}}{black}{black!50}
		\drawSquare{9}{-4}{0.75}{9}{{mcol,tcol,tcol,tcol,tcol,tcol,mcol,mcol}}{black}{black!50}
		\transformationArrow{9.75}{-4}{0.75}{1.5}
		\drawGFive{11.5}{-3.85}{0.6}{A}{{tcol,mcol,tcol,tcol,tcol,mcol,tcol,mcol,tcol,tcol,tcol,tcol,mcol,mcol}}{black}{black!50}
		
		\drawSquare{2.1}{-6}{0.75}{10}{{tcol,tcol,mcol,tcol,mcol,tcol,mcol,tcol}}{black}{black!50}		
		\transformationArrow{2.85}{-6}{0.75}{1.5}	
		\drawGFive{4.6}{-5.85}{0.6}{A}{{tcol,mcol,mcol,tcol,tcol,tcol,tcol,tcol,mcol,tcol,mcol,tcol,mcol,tcol}}{black}{black!50}
		\drawSquare{6.6}{-6}{0.75}{11}{{mcol,tcol,tcol,tcol,tcol,tcol,tcol,tcol}}{black}{black!50}
		\transformationArrow{7.35}{-6}{0.75}{1.5}	
		\drawGFive{9.1}{-5.85}{0.6}{A}{{tcol,mcol,tcol,tcol,tcol,mcol,tcol,mcol,tcol,tcol,tcol,tcol,tcol,tcol}}{black}{black!50}
	\end{tikzpicture}
	\caption{Extending a 3-decomposition to the claw-square.
	\label{fig:3-dec-extension-claw-square}}
\end{figure}

\end{document}